\documentclass[12pt, reqno]{amsart}

\usepackage[margin=0.95in]{geometry}

\usepackage{amsmath,amsthm,soul,ulem}
\usepackage{dsfont}

\usepackage{upgreek}

\usepackage{mathrsfs}
\usepackage{color}
\usepackage{xcolor}
\usepackage{graphicx}
\usepackage{amssymb}

\usepackage{url}
\usepackage{multicol}
\usepackage{tikz}
\usepackage{amsmath}
\usepackage{fancyhdr}
\usetikzlibrary{matrix}
\usetikzlibrary{arrows}
\usepackage{subfig}
\usepackage{hyperref}
\usepackage{listings}

\usepackage{enumitem}
\usepackage{varwidth}

\allowdisplaybreaks

\definecolor{bluegreen}{rgb}{0.0, 0.3, 0.9}

\newcommand{\cfp}[1]{\color{bluegreen} {\tt [FP: #1]} \color{black}}

\title[Local energy control]{Local Energy control in the presence of a \\ zero-energy resonance}

\author[J.M. Palacios]{Jos\'e M. Palacios}
\address{University of Toronto}
\email{jose.palacios@utoronto.ca}

\author[F. Pusateri]{Fabio Pusateri}
\address{University of Toronto}
\email{fabiop@math.toronto.edu}

\newcommand{\be}{\begin{equation}}
\newcommand{\ee}{\end{equation}}
\newcommand{\bp}{\begin{proof}}
\newcommand{\ep}{\end{proof}}
\newcommand{\bel}{\begin{equation}\label}
\newcommand{\eeq}{\end{equation}}
\newcommand{\bea}{\begin{eqnarray}}
\newcommand{\eea}{\end{eqnarray}}
\newcommand{\bee}{\begin{eqnarray*}}
\newcommand{\eee}{\end{eqnarray*}}
\newcommand{\ben}{\begin{enumerate}}
\newcommand{\een}{\end{enumerate}}

\newcommand{\R}{\mathbb{R}}

\newcommand{\Z}{\mathbb{Z}}

\newcommand{\sech}{\operatorname{sech}}

\newcommand{\supp}{\operatorname{supp}}

\newcommand{\sgn}{\operatorname{sgn}}

\newtheorem{thm}{Theorem}[section]
\newtheorem{cor}[thm]{Corollary}
\newtheorem{lem}[thm]{Lemma}
\newtheorem{prop}[thm]{Proposition}

\newtheorem{rem}{Remark}[section]

\definecolor{codegreen}{rgb}{0,0.6,0}
\definecolor{codegray}{rgb}{0.5,0.5,0.5}
\definecolor{codepurple}{rgb}{0.58,0,0.82}
\definecolor{backcolour}{rgb}{0.95,0.95,0.92}

\lstdefinestyle{mystyle}{
	backgroundcolor=\color{backcolour},   
	commentstyle=\color{codegreen},
	keywordstyle=\color{magenta},
	numberstyle=\tiny\color{codegray},
	stringstyle=\color{codepurple},
	basicstyle=\footnotesize,
	breakatwhitespace=false,         
	breaklines=true,                 
	captionpos=b,                    
	keepspaces=true,                 
	numbers=left,                    
	numbersep=5pt,                  
	showspaces=false,                
	showstringspaces=false,
	showtabs=false,                  
	tabsize=2
}
\numberwithin{equation}{section}



\theoremstyle{definition}

\numberwithin{ej}{section}

\begin{document}





\renewcommand{\sectionmark}[1]{\markright{\thesection.\ #1}}
\renewcommand{\headrulewidth}{0.5pt}
\renewcommand{\footrulewidth}{0.5pt}

\date{\today}


\begin{abstract}

\normalsize

We consider the problem of stability and local energy decay for
co-dimension one perturbations
of the soliton of the cubic Klein-Gordon equation in $1+1$ dimensions.
Our main result gives a weighted time-averaged control of the local energy 
over a time interval which is exponentially long in the size of the initial 
(total) energy.
More precisely, for well-prepared initial perturbations on the center stable manifold 
that are of size $\delta$ in the energy norm, 
we show that the local energy is under control up to times of the order $\exp(c\delta^{-\beta})$
for any $\beta < 4/3$.
A major difficulty is the presence of a zero-energy resonance in the linearized operator,
which is a well-known obstruction to improved local decay properties.




We address this issue 
by using 
virial estimates that are frequency localized in a time-dependent way,
and introducing a ``singular virial functional'' with time dependent weights 
to control the mass 
of the perturbation projected away from small frequencies.
The proof applies to more general models, yielding analogous results for 
perturbations of the kink of the Sine-Gordon model, and small solution of nonlinear Klein-Gordon equations.
In this respect, our result is close to optimal due to the existence of 
wobbling kinks and breathers 
in the Sine-Gordon model which violate our conclusion if $\beta = 2$.
This appears to be the first successful general 
attempt at using virial estimates in the presence of a resonance
to deduce local energy control. 
\end{abstract}

\maketitle

\setcounter{tocdepth}{2}

\begin{quote}

\tableofcontents

\end{quote}


\medskip
\section{Introduction}

We consider the focusing cubic Klein-Gordon equation 
\begin{align}\label{KG}
\phi_{tt}-\phi_{xx}+\phi-\phi^3=0, \qquad t,x\in\R,
\end{align}
$\phi(t,x) \in \R$,
and are interested in the behavior of stable co-dimension one perturbations of its soliton solution
\begin{align}
\label{Q}
Q(x)=\sqrt{2}\sech(x).
\end{align}
Our results apply to more general equations and (static) solutions 
with similar features but we restrict most our discussion and proof to the case of 
even solutions of \eqref{KG} for concreteness; 
see Remark \ref{RemSG}. 

The Hamiltonian associated with \eqref{KG} is
\begin{equation}
\mathcal{E}(\phi,\phi_t) = \frac{1}{2} \int_\R \left[ (\partial_t \phi)^2 + (\partial_x \phi)^2 + \phi^2 \right]dx 
  - \frac{1}{4}\int_\R \phi^4 dx,
\end{equation}
and for initial data in the energy space, $(\phi(0),\phi_t(0)) \in H^1(\R)\times L^2(\R)$,
the equation is locally well-posed, and globally well-posed for small solutions.
Global existence for small energy initial data follows from standard energy estimates, 
see for example \cite{Caz85}. It is also known that finite time blow-up occurs for 
solutions with negative energy initial data \cite{Levine} 
and for large positive initial data \cite{Wang, YX18}, which follows from a convexity argument.



The main purpose of the present paper is to study the decay of the local energy 
for perturbations of $Q$ belonging to the center-stable manifold.
More generally, our interest is on understanding the behavior of special solutions 
of nonlinear dispersive equations in the case where the linearized operator has a zero energy resonance
at the bottom of the continuous spectrum.
While many works in the literature have handled cases where a resonance is absent,
including cases where the spectrum has other discrete stable components,
there is very little literature on the case of a resonance. 
A few recent works have addressed this issue for perturbations in weighted spaces;
see Subsection \ref{Seclit1} for some discussion of the literature.
However, to our knowledge, there is no result on the control of local energy in the presence of a resonance
for general perturbations of finite energy.

\medskip
\subsection{Linearization and zero-energy resonance}\label{introL}
Letting $\phi = Q + v$ denote the perturbed soliton, for a real-valued small $v$, we see that
\begin{align}\label{introlineq}
v_{tt} + \mathcal{L} v = N(v), \qquad N(v) = 3Qv^2 + v^3,
\end{align}
where the linearized operator around $Q$ is
\begin{align}\label{introlinop}
\mathcal{L}:=-\partial_x^2+1-3Q^2(x)=-\partial_x^2+1-6\sech^2(x)
\end{align}
We let $H:=-\partial_x^2-3Q^2$ and will refer to it as the Schr\"odinger operator
associated to the linearization at $Q$.
The spectrum of $\mathcal{L}$ is given by
\[
\hbox{spec}\,\mathcal{L}=\{-3\}\cup\{0\}\cup[1,+\infty).
\]

\setlength{\leftmargini}{1.5em}
\begin{itemize}

\item[-] The first eigenvalue $\lambda = -3$ is a linear exponentially unstable mode
with associated normalized eigenfunction
\[
Y(x):=\tfrac{\sqrt{3}}{2}\sech^2(x), 
\qquad \langle Y(x),Y(x)\rangle=1, \qquad \mathcal{L}Y=-3Y.
\]
If one considers
\begin{align}\label{introYpm}
\vec{\mathbf{Y}}_\pm(x):=\left(\begin{matrix}
Y(x) \\ \pm\sqrt{3}Y(x)
\end{matrix}\right),
\end{align}
then $\vec{\boldsymbol{\varphi}}_\pm(t,x):=e^{\pm\sqrt{3}t}\vec{\mathbf{Y}}_\pm(x)$ solves the linear equation
\begin{align}\label{introlineq'}
\begin{cases}
\dot{\varphi}_1 & = \varphi_2
\\ 
\dot{\varphi}_2 & =-\mathcal{L}\varphi_1,
\end{cases}
\end{align}
which shows the presence of exponentially stable and unstable dynamics near the soliton.
It is then natural to decompose further $v$ as 
\begin{align}\label{introdec}
v(t,x) = a(t) Y(x) + \varepsilon(t,x), \qquad \langle Y, \varepsilon \rangle = 0.
\end{align}

\smallskip
\item[-] The second eigenvalue $\lambda = 0$ is the translational mode associated with $Q'$.
Since $Q'$ is odd, by restricting the analysis to even solutions we may disregard this mode
(see also the discussion in Remark \ref{RemSG}). 
Observe that under this condition we also have $\varepsilon = P_c v$,
where $P_c$ is the projection onto the continuous spectrum.

\smallskip
\item[-] The bottom of the continuous spectrum $\lambda = 1$ is a (zero energy, relative to $H$) 
threshold resonance associated to the (even) bounded function
\[
R(x)=1-\tfrac{3}{2}\mathrm{sech}^2(x).
\]
In this case the potential $V=-3Q^2$ is said to be ``non-generic''.
Observe that there is no internal mode, that is, no eigenvalue in $(0,1)$. 
\end{itemize}

\smallskip
The presence of a bounded solution of 
$H\varphi = 0$ is a well-known obstruction to obtaining improved local-in-space decay properties 
of solutions of the linear equation $e^{it\sqrt{\mathcal{L}}} P_c u_0$.
Indeed, the situation is analogous to that of the flat Schr\"odinger operator $-\partial_x^2$
(for which the constant function $\varphi \equiv 1$ is a threshold resonance)
where the explicit form of linear solutions shows that local decay is no better than
the regular linear decay which has the rate $t^{-1/2}$ for localized data.
The situation is drastically different for ``generic'' potentials where, in the absence of 
a threshold resonance, local-in-space norms decay at a much faster rate,
that is, $t^{-3/2}$ for data in $\langle x \rangle^{-1}L^1$.

As a consequence, in generic cases the expectation is that the local energy of (stable) perturbations,
decays over time, while this may not happen in ``non-generic'' cases.
One important example to keep in mind is that of the Sine-Gordon equation $\phi_{tt} - \phi_{xx} + \sin\phi = 0$
which admits small breather solutions that are spatially localized and time periodic,
as well as wobbling kinks.
We will discuss this in more details below, after stating our main result in Theorem \ref{maintheo}.


\medskip
\subsection{Center-stable manifold and asymptotic stability}\label{introSM}
Before stating our main result we review the definition of the 
center-stable manifold around the soliton $Q$ for even solutions of \eqref{KG},
following 
\cite{KMM21,BJ89}. 
Recall the matrix form of the linear equation in \eqref{introlineq'},
the definition of the exponentially unstable and stable solutions \eqref{introYpm},
and denote 
\begin{align}\label{vecnot}
\vec{\boldsymbol{\phi}} := (\phi_1, \phi_2) = (\phi, \phi_t).
\end{align}
We let
\begin{equation}\label{Zpm}
\vec{\mathbf{Z}}_+
= \begin{pmatrix} Y,  \tfrac{1}{\sqrt{3}} Y \end{pmatrix}
\end{equation}
so that $\langle \vec{\mathbf{Y}}_-, \vec{\mathbf{Z}}_+ \rangle = 0$.
For every $\delta_0>0$, define
\begin{equation}
\label{A-0}
\mathcal{A}(\delta_0) := \left\{  
  \boldsymbol{\varepsilon} \in H^1_x(\R) \times L^2_x (\R) : 
  \quad \boldsymbol{\varepsilon} \,\, \textrm{is even}, \,\, {\| {\boldsymbol{\varepsilon}}\|}_{H^1_x \times L^2_x} < \delta_0 
  \quad \text{and} \quad \langle {\boldsymbol{\varepsilon}}, \vec{\mathbf{Z}}_+ \rangle = 0 \right\}.
\end{equation}
Then, the following stability result was obtained in \cite{BJ89} (see also \cite{KMM21,LiLu22}).

\begin{thm}
\label{thmSM}
There exist $C, \delta_0>0$ and a Lipschitz function $h: \mathcal{A}(\delta_0) \to \R$ 
with $h(0)=0$ and $|h(\boldsymbol{\varepsilon})|\leq C {\| \boldsymbol{\varepsilon} \|}^{3/2}_{H_x^1 \times L^2_x}$ 
such that, denoting 
\begin{equation}
\label{center-manifold}
\mathcal{M}(\delta_0) := \{(Q, 0) + {\boldsymbol{\varepsilon}} + h({\boldsymbol{\varepsilon}}) \vec{\mathbf{Y}}_+: \
  \, {\boldsymbol{\varepsilon}} \in \mathcal{A}(\delta_0)\}
\end{equation}
the following holds:

\begin{itemize}
\smallskip
\item[1.] If $\vec{\boldsymbol{\phi}_0} \in \mathcal{M}(\delta_0)$, then the solution of \eqref{KG} 
with initial data $\vec{\boldsymbol{\phi}_0}$ is global and satisfies
\begin{equation}
\label{global-bound-KMM}
{\big\|\vec{\boldsymbol{\phi}}(t)-(Q,0) \big\|}_{H_x^1 \times L^2_x}
  \leq C {\big\|\vec{\boldsymbol{\phi}_0} - (Q,0) \big\|}_{H_x^1 \times L^2_x} \quad \quad \text{for all} \quad t\geq 0.
\end{equation}

\smallskip
\item[2.] If a global even solution $\vec{\boldsymbol{\phi}}(t)$ of \eqref{KG} satisfies 
\begin{equation*}
{\big\|\vec{\boldsymbol{\phi}}(t) - (Q,0) \big\|}_{H_x^1 \times L^2_x} < \frac{1}{2}\delta_0 \quad \quad \text{for all} \quad t\geq 0,
\end{equation*}
then $\vec{\boldsymbol{\phi}}(t) \in \mathcal{M}(\delta_0)$ for all $t\geq 0$.
\end{itemize}
\end{thm}

Theorem \ref{thmSM} provides a center-stable manifold $\mathcal{M}=\mathcal{M}(\delta_0)$ such that data 
originating in $\mathcal{M}$ is global, 
and stays globally-in-time close to the soliton in the energy norm.
Moreover, any solution that is globally sufficiently close to $Q$, has to lie in $\mathcal{M}$. 
%
%
%
Then, a natural question to ask is what can be said about the long-time 
(and asymptotic) behavior of the above global solutions, and in what sense they are attracted to the soliton.






\medskip
\subsection{Some literature}\label{Seclit1}
There is a vast literature on this type of question for various nonlinear equations
and special solutions (solitons, solitary waves, kinks); 
since it would be impossible to give a complete overview, 
we will concentrate on the results that are more closely connected to \eqref{KG}
and our result; we refer the readers to the cited works for more discussions and references.

For nonlinear equations with large enough power nonlinearities, 
asymptotic stability result have been achieved by several authors. 
For the focusing Klein-Gordon equation with power nonlinearity $\phi^p$ with $p>5$
asymptotic stability (on the center-stable manifold) was proven by Krieger-Nakanishi-Schlag \cite{KNS12}
for data in the energy space.
See also \cite{NSbook,NS3dnr} and reference therein, for a similar problem
and more general results in higher dimensions.
For the nonlinear Schr\"{o}dinger (NLS) equation,
Krieger-Schlag \cite{KriSchNLS} treated the case of a $|u|^{2\alpha}u$ nonlinearity with $\alpha>2$ in $1$d,
and Cuccagna-Pelinovsky \cite{CP14} treated the cubic case.
See also the seminal work of Schlag \cite{Sch09} for the cubic NLS in $3$d,
the recent 
survey by Cuccagna-Maeda \cite{CM21} 
and references therein.

For low power nonlinearities instead (typically the mass sub-critical ones)
one cannot expect similar asymptotic stability result
in the energy space $H^1_x(\R) \times L^2_x(\R)$, essentially because Strichartz estimates
are not very effective in controlling the nonlinearity.
There are then two natural different and complementary ways to look at this question:

\smallskip
\noindent
(a) study asymptotics for energy solutions locally in the space variable $x$, or

\smallskip
\noindent
(b) study asymptotics on the whole real line in a smaller space, e.g., 
a weighted space. 

\smallskip
The first approach is usually based on delicate virial type estimates, and has been 
very successfully employed in several problems; this is the general approach we will follow here.
Recent works in this direction include the result of Kowalczyk-Martel-Mu\~noz \cite{KMM21}
on asymptotic stability in the energy space for KG with power type nonlinearity $|\phi|^{p-1}\phi$, for $p>3$,
the work of Li-L\"uhrmann \cite{LiLu22} on the case $p=2$, 
Cuccagna-Maeda-Murgante-Scrobogna \cite{CMMS23} on the case $p \in (5/3,2]$
and Cuccagna-Maeda-Scrobogna \cite{CMS22} on small solutions of cubic KG.
See also the review 
by Kowalczyk-Martel-Mu\~noz \cite{KMMreview}.
We also mention a series of important related papers on 
analogous results for kink solutions of relativistic scalar fields \cite{KowMarMun,KowMarMunVDB,CM22,KMkinks}.
It is important to note that in all these works a zero energy resonance is always absent,
either because the potential is generic or because of imposed parity restrictions.

The approach (b) has instead been pursued mostly in the context of 
kinks (non-localized solitons) which naturally lead to nonlinear PDEs with potentials and 
low power non-localized nonlinearities;
we refer the reader to Delort-Masmoudi \cite{DMKink}, Germain and the second author \cite{GP20} and Zhang
\cite{KGVSim},
L\"uhrmann-Schlag \cite{LSch}, Kairzhan and the second author \cite{KaPu22},
and Lindblad-L\"uhrmann-Schlag-Soffer \cite{LLSS}.
See also the works 
\cite{CP22,NauWed,ColGer23} for more on NLS-type equation. 
In all of the works cited so far (except the relatively simpler models treated in \cite{LLSS,CP22}) 
there is no zero energy resonance,
or some additional assumptions (e.g. in \cite{GP20}) 
or special structure of the equations (e.g. in the Sine-Gordon case \cite{LSch}) mitigate its effects.
Finally, L\"uhrmann-Schlag \cite{LSch23} considered the same 
codimension-one stability question for \eqref{KG}-\eqref{Q} that we are interested in here
but in the case of perturbations belonging to a weighted Sobolev space;
these authors proved that initial perturbations of size $\delta$ 
remain small and decay (almost) at the linear rate up to times of the order $\exp(\delta^{-1/4})$.
So far \cite{LSch23} appears to be the only general result on long-time stability in the presence of a resonance.

\medskip
\subsection{Main result}\label{intromain}
In this subsection we state our main results and make a few comments.

\begin{thm}\label{maintheo}
Consider \eqref{KG} and its static soliton \eqref{Q}.
If a global even solution $\vec{\boldsymbol{\phi}}(t)$ satisfies
\begin{equation}\label{stable_manifold_hyp_thm}
{\big\|\vec{\boldsymbol{\phi}}(t)-(Q,0) \big\|}_{H_x^1 \times L^2_x}
  \leq C \delta,
  \quad \quad \text{for all} \quad t\geq 0
\end{equation}
for a small enough $\delta>0$,
then, for any $\beta < 4/3$, there exists a constant $c>0$ independent of $\delta$ such that,
using the same notation in \eqref{introlineq}-\eqref{vecnot},
\begin{align}\label{mtconc}
\int_0^{T_{\mathrm{max}}} \!\! \int_I \, \dfrac{1}{\langle t \rangle}
  \big(\varepsilon_{t}^2 + \varepsilon_{x}^2 + \varepsilon^2 \big) \, dx dt \lesssim \delta^2,
\end{align}
for any bounded interval $I$, and
\begin{align}\label{mtconca}
\int_0^{T_\mathrm{max}}\dfrac{1}{\langle t \rangle} \big( a^2(t) + \dot{a}^2(t) \big) dt \lesssim \delta^2,
\end{align}
with
\begin{align}\label{mtT}
T_{\mathrm{max}} := \exp \big(c \delta^{-\beta}\big) .
\end{align}
\end{thm}

\smallskip
We will give a brief sketch of the proof of Theorem \ref{maintheo} in Section \ref{heuristics_sec}.
Before that, let us make some remarks on our main result and related problems.


\smallskip
\begin{rem}[Other power nonlinearities]\label{Remp}
As mentioned above, in \cite{KMM21} Kowalczyk-Martel-Mu\~noz considered the 
pure power Klein-Gordon equation
\begin{align}\label{KGalpha}
\partial_t^2 \phi - \partial_x^2 \phi + \phi = |\phi|^{2\alpha}\phi, \qquad \alpha > 1,
\end{align}
and its static soliton 
\begin{equation}\label{Qalpha}
Q_p(x) := (\alpha+1)^{\frac{1}{2 \alpha}} \sech^{1/\alpha} (\alpha x), \qquad p=2\alpha+1.
\end{equation}
In this case, they proved that stable even solutions (as in Theorem \ref{thmSM}, which
holds verbatim for \eqref{KGalpha}-\eqref{Qalpha}) satisfy
\begin{align}\label{KMMconcalpha}
\int_0^\infty \int_I 
  \big(\varepsilon_{2}^2 + \varepsilon_{1,x}^2 + \varepsilon_1^2 \big) \, dx dt 
  \lesssim \delta^2.
\end{align}
This conclusion is obviously much stronger than \eqref{mtconc}
since one can take $T_{\mathrm{max}}=\infty$ and remove the weight of $(1+t)^{-1}$.
However, such a strong conclusion seems possible, in general,
only because of the absence of a zero-energy resonance.
In particular, the potential of the linearization around $Q_p$ is generic for $p>3$.
Analogous results were proven by Li-L\"uhrmann \cite{LiLu22} for a $u^2$ nonlinearity,
and Cuccagna-Maeda-Murgante-Scobrogna \cite{CMMS23} for the case for $5/3 < p \leq 2$;
in these cases internal modes appear but, again, there is no zero-energy resonance
(or the resonance is odd and it is avoided by considering even perturbations).
\end{rem}

\smallskip
\begin{rem}[Natural limitations]\label{Remlimit}
In the case that we are considering here or, more in general, in cases
where the linearized operator has a resonance (and no additional special structure
that compensates for it)
one should not expect an estimate like \eqref{KMMconcalpha} to hold true.
Indeed, local decay, even for linear solutions from Schwartz data,
is the same as the global one, at the rate of $t^{-1/2}$.


Furthermore, notice that the Klein-Gordon models with non-generic potentials 
that arise in the linearization, may posses
small time-periodic and spatially localized solutions $\varepsilon$ (see \eqref{introlineq} and \eqref{introdec}),
called `breathers'. 
This is actually the case for the Sine-Gordon equation
which possesses a ``wobbling kink'' solution close to the its kink solution.
See Remark \ref{RemSG} for more on this.
The existence of such solutions is an obstruction to asymptotic stability in the energy space.
Moreover, the local energy of such solutions is approximately 
a constant power of the global energy and, therefore, even an inequality like 
\eqref{mtconc} can only hold up to exponentially long times in powers of $1/\delta$.
We will expand on this below.

\end{rem}


\smallskip
\begin{rem}[The case of Sine-Gordon]\label{RemSG}
Our argument is based on virial type estimates and works without any particular assumption
on the model. In fact, it applies also to odd perturbations of the kink of the Sine-Gordon (SG) model 
as well as to general nonlinear Klein-Gordon equations with (localized) variable coefficients quadratic nonlinearities
and non-generic (decaying) potentials (including $V=0$).
It is likely that it also applies in the presence of internal modes.
However, this case would require to modulate the solution. 
For the sake of simplicity we do not consider this scenario here. 

\medskip
Recall the SG model and its static kink solution:
\begin{align}\label{SG}
\phi_{tt} - \phi_{xx} + \sin \phi = 0, \qquad K(x) = 4 \arctan e^x.
\end{align}
The evolution equation for odd perturbations $u(t,x) = \phi(t,x) - K(x)$ is 
\begin{align}\label{SGu}
u_{tt} - u_{xx} -2\sech^2(x) u + u = a(x)u^2 + (\tfrac{1}{6} + b(x)) u^3 + \cdots
\end{align}
with $a(x) = \sech(x)\tanh(x)$ and $b(x) = (1/3) \sech^2(x)$ smooth localized coefficients,
where ``$\dots$'' denote terms of homogeneity higher than three in $u$.
The potential $V(x) = -2\sech^2(x)$ has an odd resonance, $R(x) = \tanh(x)$, and the translation mode $K'$.
Since our proof does not use the exact form of $V$, but just 
that the coefficient of the quadratic terms is localized, we have the following result:

\begin{thm}\label{thmSG}
Let $\phi$ be an odd solution of \eqref{SG} such that ${\| \vec{\phi}(0) - (K,0) \|}_{H^1\times L^2} \leq \delta$;
then, for any $\beta < 4/3$, there exists a constant $c>0$ independent of $\delta$ such that,
for any bounded interval $I$ and for $T_{\mathrm{max}} := \exp \big(c \delta^{-\beta}\big)$
\begin{align}\label{SGconc}
\int_0^{T_{\mathrm{max}}} \!\! \int_I \, \dfrac{1}{\langle t \rangle}
  \big(u_t^2 + u_{x}^2 + u^2 \big) \, dx dt \lesssim \delta^2,
\end{align}
where $u$ is the solution of \eqref{SGu}.
\end{thm}

\end{rem}

\smallskip
\begin{rem}[About the maximal time: comparison with Sine-Gordon]\label{RemSGbr}
For the specific case of the SG equation we can see that the optimal time $T_{\mathrm{max}}$ in Theorem \ref{thmSG} is \[
T_{\mathrm{max}}=\exp\big(c\delta^{-2}\big).
\]
In fact, SG admits (possibly small) breather solutions of the form
\begin{align}\label{Breather}
B_a(t,x) := \arctan\Big( \frac{a \sin(bt)}{b \cosh(ax)} \Big), \qquad a^2+b^2=1,
\end{align}
and wobbling kink solutions 
\begin{align}\label{Wobbling}
\begin{split}
& W_a(t,x) := 4 \mathrm{Arg}(U_a + i V_a)
\\
& U_a(t,x) := \cosh(ax) + a\sinh(ax) - a e^x \cos(bt)
\\
& V_a(t,x) := e^x \big( \cosh(ax) - a\sinh(ax) - a e^{-x} \cos(bt) \big), \qquad a^2+b^2=1.
\end{split}
\end{align}
Let us focus on the wobbling kink. 
Notice that, as $a \rightarrow 0$ this converges to the kink.
Moreover, a direct calculation shows that if $u = W-K$ then, for $a:=\delta^2$ small enough, one has
\begin{align}\label{Wper}
{\| (u(0),u_t(0)) \|}^2_{H^1\times L^2(\R)} \approx \delta^2, 
\end{align}
and, for any finite fixed interval $I$,
\begin{align}
{\| (u,u_t)(t) \|}_{H^1\times L^2(I)}^2 \approx \delta^4 \cos^2(bt).
\end{align}
Then
\begin{align}
\int_0^{T} \!\! \int_I \, \dfrac{1}{\langle t \rangle}
  \big(u_t^2 + u_{x}^2 + u^2 \big) \, dx dt \approx 
  \int_0^{T} \dfrac{1}{\langle t \rangle} \, \cos^2(bt) \, \delta^4 \, dt 
  \approx  \delta^4 \, \log (1+T)
\end{align} 
and, therefore, one must have
\begin{align}
T_{\mathrm{max}} \lesssim e^{C/\delta^2} 
\end{align}
in the inequalities \eqref{SGconc} and \eqref{mtconc}, if no additional assumptions are made.

We refer the reader to the work by Alejo, Mu\~noz and the first author \cite{AMP}
for results on the stability of $K$ for a particular class of initial data 
and for more discussions about wobbling kinks and their stability (see also \cite{MuPa}). 
For the case of small and spatially decaying perturbations of $K$
- which in particular exclude the wobblink kink -
see the works of L\"uhrmann-Schlag \cite{LSch}, 
Chen-Liu-Lu \cite{CLL} and Koch-Yu \cite{KoYu23}.
\end{rem}

\smallskip

\begin{rem}[Implications of \eqref{mtconc}]\label{Remother2}
Reasoning along the same lines of Remark \ref{RemSGbr}, we see that our main theorem
rules out the existence of some perturbations of stationary solutions
that are breather-like, that is, time-periodic and fast decaying.
In the case of cubic Klein-Gordon, 
perturbations $\varepsilon$ that behave like $Ae^{-B|x|}$ as $|x|\rightarrow \infty$,
are, in principle, possible only if $B^5 \lesssim A^4$;
in fact, in such a case the total energy of the perturbation is $\delta^2 = A^2B^{-1}$,
and its local energy on a finite interval is $A^2$.
Then, since we must have $(\log T_{\mathrm{max}}) \cdot A^2 \approx \delta^{-4/3} A^2 \lesssim \delta^2$,
we see that necessarily $B^5 \lesssim A^4$. 
Moreover, the latter inequality along with the relation coming from the total energy of the perturbation, 
implies that 
\[
A\lesssim \delta^{5/3}.
\]
Thus, for example, 
we conclude that a breather-like solution behaving 
like $\delta^{3/2}e^{-\delta \vert x\vert}$ cannot exist
(observe that this solution satisfies that its total energy is of order $\delta^2$). 


\end{rem}

%


\medskip
\section{Ideas of the proof}\label{heuristics_sec}
The purpose of this section is to explain the main ideas and some of the 
details for the proof of Theorem \ref{maintheo}.

\medskip
\noindent
{\it Set-up}.
A simplified model for our analysis is given by the following equations
\begin{align}\label{linear_KG}
\varepsilon_{tt} & = \varepsilon_{xx}-\varepsilon-V\varepsilon + \mathcal{N}_1(\varepsilon),
\\ 
\label{linear_KG'}
\varphi_{tt}&=\varphi_{xx}-\varphi - \mathcal{N}_2(\varepsilon),
\end{align}
where the non-generic potential $V(x)$ is a Schwartz function,
the equation \eqref{linear_KG} represents the equation for the perturbation around the soliton of \eqref{KG}
(see \eqref{system_epsilon}), 
while the second arises as a transformation of the first problem, and it is referred to as 
the `transformed equation'; $\varphi$ is called the `transformed variable'.
The transformed equation is obtained by conjugating the linearized operator $-\partial_{xx} + 1 + V(x)$
to the flat operator $-\partial_{xx} + 1$ 
with a Darboux transformation (see Section \ref{SecTransformed} for details).
For this explanation we are going to disregard the amplitude $a(t)$ of the discrete component.
Also, we will mostly focus on the linear parts of the equation
and disregard $\mathcal{N}_1,\mathcal{N}_2$, although their structure and
analysis is important for the proof. 
From now on we denote by $\mathcal{R}$ all terms coming from $\mathcal{N}_1$ and $\mathcal{N}_2$ that 
appear in 
the following virial estimates and, 
for the sake of this explanation we ignore them. 

Let $f$ be any solution to \eqref{linear_KG} and let
\begin{align}\label{ideasI}
\mathcal{I}[f;\Phi](t) := \int_\R \left( \Phi(x)f_{x}(t,x) 
  + \frac{1}{2}\Phi'(x) f(t,x) \right) f_t(t,x) dx,
\end{align}
be the standard virial functional with a weight function $\Phi$. 
A direct calculation gives (we drop the dependence on the variables)
\begin{align}\label{basicvirial}
\begin{split}
    \dot{\mathcal{I}}[f;\Phi] = \,
    & - 
    \int f_{x}^2\Phi' + \frac{1}{2} \int f^2 \, V'\Phi
    + 
    \frac{1}{4} \int f^2 \Phi'''
    \\
    & + \int \big(  \Phi f_x + \frac{1}{2} \Phi' f \big) \, (\Box + 1 + V)f. 
\end{split}
\end{align}
%
%
We will apply the above virial identity several times to different components of $\varepsilon$
and of the transformed variable $\varphi$ that are frequency localized.
All the main frequency localization will be done in a time dependent way. 
We will also apply the same identity with 
time-dependent weight functions
to obtain information at different spatial scales.
All of the functionals that we will consider are bounded (at all times) in terms of the energy 
of the perturbation.

To be more concrete, let us introduce some notation. We define the following parameters
\begin{align}\label{ideasparam}
A\gg 
  B,\qquad \kappa(t):= \log^{1+\alpha}(e^D+t), 
  \qquad 
  \mu(t):= \langle t \rangle,
\end{align}
with $0<\alpha\ll 1$ small and $D$ a large absolute to be chosen (independently of $A$ and $B$)
in the course of the proof; see, for example, Lemma \ref{weight_triple_prime}.
Recall that $T_{\mathrm{max}} = \exp(c/\delta^{\beta})$ with $\beta < 4/3$ so that, in particular,
for all $t \leq T_{\mathrm{max}}$ we have
\begin{align}\label{condition_relating_tmax_kappa_delta}
\delta^2\kappa^{3/2}(t)\ll 1, 
\end{align}
for $\alpha$ small enough depending on $\beta$. 


We define the ``frequency-localized variables''
\begin{align}\label{ideasfreqloc}
\begin{split}
& \varepsilon_{< }:=\widetilde{\mathbb{P}}_{< \kappa^{-1}}\varepsilon,  
  \qquad \varepsilon_{\geq} := \widetilde{\mathbb{P}}_{\geq \kappa^{-1}}\varepsilon = 
  \varepsilon - \varepsilon_{< },
\\ 
&  \varphi_<:=\mathbb{P}_{< \kappa^{-1}}\varphi,  
  \qquad \varphi_\geq :=\mathbb{P}_{\geq \kappa^{-1}}\varphi 
  = \varphi - \varphi_{< },
\end{split}
\end{align}
where $\widetilde{\mathbb{P}}_{< k}$ denotes the Littlewood-Paley projection to frequencies
smaller than $k$ with respect to the Schr\"odinger operator $-\partial_x^2 + V$ in \eqref{linear_KG},
and $\mathbb{P}_{< k}$ denotes the standard projection (with respect to $-\partial_x^2$)
to frequencies smaller than $k$; see Subsection \ref{secnotation} and Appendix \ref{appdFT}
for the exact definitions.
We refer to $\varepsilon_{< }$, or $\varphi_{< }$, as low-frequency components,
and we will refer to $\varepsilon_{\geq }$, or $\varphi_{\geq }$, as high frequency components,
with a slight abuse since, technically, these latter are just `not-so-low' frequency components.

\medskip
\noindent
{\it Step 1: Virial estimates and integrability of the low frequency component}.
In Section \ref{original_virial_section} we begin our proof 
by 
taking advantage of the above definitions
and, by an application of (distorted) Bernstein inequality, using the frequency localization of $\varepsilon_<$, 
obtain that 
\begin{align}\label{ideas12}
\int_0^{+\infty}\dfrac{1}{\mu(t)}\int_\R \varepsilon_<(t,x)^2 
  \sech^2(x) dxdt\lesssim \Vert \varepsilon\Vert_{L^\infty_tL^2_x}^2\int_0^{\infty}\dfrac{dt}{\kappa(t)\mu(t)}
  = O(\delta^2).
\end{align}
A similar estimate can be obtained for the time derivative of $\varepsilon_<$. All of the analysis is then aimed at showing an estimate of the form
\begin{align}\label{ideas13}
\int_0^{T_{\mathrm{max}}}\dfrac{1}{\mu(t)}\int_\R \varepsilon_\geq(t,x)^2 \sech^2(x) dxdt = O(\delta^2),
\end{align}
and analogous ones for time and space derivatives of the high-frequency component.\footnote{The estimates 
for space derivatives are (as expected) easier to show, and could actually be proven
on larger intervals that grow with time; estimates for the time derivatives 
should be considered at the same level as the ones for the function itself, as the example of 
the wobbling kink shows.}
In order to control the time derivatives, during our arguments we will also make use 
of what we call the ``true momentum'' functional 
\begin{align}\label{ideastruemom}
\begin{split}
& \mathcal{P}[f;\Phi](t) := \int_\R \Phi(x)f_{x}(t,x) \, f_t(t,x) dx,
\\
& \dot{\mathcal{P}}[f;\Phi] =
    - \dfrac{1}{2} \int f_t^2 \Phi' - \dfrac{1}{2} \int f_x^2 
    + \dfrac{1}{2} \int f^2 \Phi' - \dfrac{1}{2} \int f^2 (V\Phi)' 
    + \int \Phi f_x (\Box + 1 + V)f. 
\end{split}
\end{align}


\medskip
\noindent
{\it Step 2: Virial estimates for the high frequency component}.
When considering $\varepsilon_\geq$, it is well understood that one cannot directly obtain 
integrability using basic virial estimates (plus, this should not be expected to hold in general,
as discussed in the remarks after our main theorem). 
%
Applying \eqref{basicvirial} to 
$\mathcal{I}_{\geq} := \mu^{-1} \mathcal{I}[\varepsilon_\geq;\Phi_A]$, one can obtain an estimate of the form
\begin{align}\label{ideas20}
\begin{split}
\dfrac{d}{dt}\mathcal{I}_{\geq} & \lesssim -\dfrac{1}{2\mu} \int \varepsilon_{\geq,x}^2\Phi_A' dx
  + \frac{4}{\mu} \int \varepsilon_{\geq}^2 \sech^{1/2}(x) dx
  + \mathcal{R},
\end{split}
\end{align}
and so the problem reduces to controlling either one of the first two terms
on the right-hand side of \eqref{ideas20},
assuming that the remaining term $\mathcal{R}$ is negligible or absorbable by the previous two terms. 
We then seek to control the integral of $\varepsilon_{1,\geq}^2\sech^{1/2}(x)$ 
by working on the transformed problem 
throughout Sections \ref{SecTransformed}-\ref{sec_Analysis_H}.

\medskip
\noindent
{\it Step 3: The transformed problem and coercivity estimates}.
In Section \ref{SecTransformed} we pass to the so-called transformed problem by conjugating the 
linearized operator $\mathcal{L}$ in \eqref{introlinop} to the flat operator $\mathcal{L}_0 := -\partial_x^2+1$,
via a Darboux transformation: $SU \mathcal{L} = \mathcal{L}_0 SU$ (see \eqref{op_SU} for the definition).
The transformed variable is then defined as in \eqref{defdual} 
by applying (a smoothed-out version of) the conjugating operator $SU$ to $
\varepsilon$. 
We then consider the frequency localized version $\varphi_{\geq}$ as in \eqref{ideasfreqloc}.
This satisfies the system of equations \eqref{system_pde_dual_variables},
of which \eqref{linear_KG'} is a simplified version.
The rest of the proof is then performed using virial estimates based on this transformed equation.

Before proceeding to estimate $\varphi_{\geq}$ we need to show that it 
controls the ``bad term'' in \eqref{ideas20}.
More precisely, in Section \ref{SecTransformed} we prove the following coercivity estimate 
(see Proposition \ref{prop_coer_varepsilon})
\begin{align}\label{coer_chi_varep0}
\int \varepsilon_{\geq}^2\sech^{1/2}(x)
  & \lesssim \int \big(\varphi_{\geq,x}^2 + \varphi_{\geq}^2\big)\sech^{1/4^+}(x)
  + \mathcal{R}.
\end{align}
The proof follows in part the standard procedure from \cite{KowMarMun}. However,
the situation here is slightly more delicate due to the time-dependent frequency localization.


\medskip
\noindent
{\it Step 4: Reduction to estimating $\partial_t \varphi_\geq$.}
Section \ref{SecVirialDual} is dedicated to virial estimates for $\varphi_\geq$ using the functionals 
\begin{align}
\label{ideas41}
\mathcal{J}_\geq & := \mu^{-1} \mathcal{I}[\varphi_\geq;\Phi_B],
\qquad \mathcal{K}_\geq :=  \mu^{-1} \mathcal{P}[\varphi_\geq;\uppsi_B];
\end{align}
here $B \ll A$ is an intermediate (time-independent) large scale, 
$\Phi$ is the same weight function used above,
and 
$\uppsi_{B}$ is such that $\uppsi_{B}' \approx \Phi_{B}' \sech^{1/4}(x)$; see \eqref{def_phi_A} and \eqref{defpsiB}.
With these choices we can show an estimate of the form 
\begin{align}\label{introdj_m_dk_ineq}
\dfrac{d}{dt}\mathcal{J}_\geq-\dfrac{d}{dt}\mathcal{K}_\geq &\lesssim  \dfrac{1}{2\mu}\int \varphi_{\geq,t}^2\sech^{1/4}(x)-\dfrac{1}	{2\mu}\int \varphi_{\geq,x}^2-\dfrac{1}{8\mu}\int \varphi_{\geq}^2\sech^{1/4}(x) + \mathcal{R}
\end{align}
With \eqref{introdj_m_dk_ineq} we have reduced the whole problem to controlling the 
first term on its right-hand side, that is, 
\begin{align}\label{ideas45}
\int (\partial_t \varphi_{\geq})^2(t,x) \sech^{1/4}(x) dx.
\end{align}
This is done in our last step in Section \ref{sec_Analysis_H}.

\medskip
\noindent
{\it Step 5: The ``singular virial momentum'' and conclusion of the argument.}
To estimate \eqref{ideas45} we first make the following simple, but important, observation.
By decomposing $\partial_t \varphi_{\geq}$ in frequency space, 
letting $\varphi_{t,N^{-1}} := \mathbb{P}_{N^{-1}} (\partial_t \varphi)$, we can bound 
\begin{align}\label{ideas50}
\int (\partial_t \varphi_{\geq})^2(t,x) \sech^{1/4}(x) dx 
  \lesssim \sum_{1 \leq N \leq 2\kappa(t)} \frac{\log\kappa}{N} \int \varphi_{t,N^{-1}}^2(t,x)  
  \, \sech^2 \big(x / \lambda_N \big) dx + \mathcal{R}
\end{align}
where $\lambda_N(t) := C N \log \kappa(t)$, for some constant $C > 1$.
In particular, we have increased the spatial support of the localizing weight to a scale of $O(\lambda_N)$;
this allows us to create a certain separation between the frequency support of 
the argument (which is at scale $\approx N^{-1}$) and that of the weight (which is at scales $\lesssim \lambda_N^{-1}$,
up to exponentially decaying tales).
Moreover, we have gained a factor of $N^{-1}$ by a Bernstein's type inequality. 

We then introduce a ``singular virial momentum'' functional
\begin{align}\label{ideasHN}
\mathcal{H}_{N^{-1}} & := 
\mathcal{I}\big[\partial_x^{-1} \varphi_{t,N^{-1}}; \, \Phi_{\lambda_N} \big].
\end{align}
From \eqref{ideasI} we see that the leading order term of $\dot{\mathcal{H}}_{N^{-1}}$ 
can be used to control the term in the sum in \eqref{ideas50}. 
Moreover, the frequency separation mentioned above allow us, 
in some informal sense, to factor out 
the $\partial_x^{-1}$ operator, 
and 
show that the leading order term in $\dot{\mathcal{H}}_{N^{-1}}$ also
controls the term
\begin{align*}
\int \big(\partial_t \partial_x^{-1} \varphi_{N^{-1}}\big)^2 \Phi_{\lambda_N}''' dx
\end{align*}
which corresponds to the third integral on the right-hand side of the basic identity \eqref{basicvirial};
see the ``scale separation'' Lemma \ref{weight_triple_prime}. 

Recalling that for the transformed variable the identity \eqref{basicvirial} holds with $V\equiv0$, 
we see that, essentially, we are only left with estimating 
the nonlinear terms in $\dot{\mathcal{H}}_{N^{-1}}$. 
These terms 
now appear with factors of $\partial_x^{-1}$ in front,
but with a careful use of Bernstein's inequalities 
we can 
control them up to the desired time $T_{\mathrm{max}}$.
Combining this last virial estimates with all the previous ones,
and adding control on the discrete component of the solution, we 
finally obtain the result of Theorem \ref{maintheo}; 
see Section \ref{secprmt} for this last part of the proof.

\medskip
\subsection{Notation}\label{secnotation}
We adopt the following notation, most of which are standard: First, we use 
$\langle x \rangle$ as a short-hand for $\sqrt{1+|x|^2}$. 
%
%
%
%
%
We use $a\lesssim b$ when $a \leq Cb$ for some absolute constant $C>0$ independent on $a$ and $b$;
$a \approx b$ means that $a\lesssim b$ and $b\lesssim a$.
When $a$ and $b$ are expressions depending on our main variables or parameters, the inequalities
are assumed to hold uniformly over these.

\smallskip
\noindent
We denote by $a+$ (respectively $a-$) a number $b>a$ (respectively $b<a$) that can be chosen arbitrarily close to $a$. Also, we use standard notation for Lebesgue and Sobolev norms, such as $L^p$, $W^{s,p}$, with $H^s = W^{s,2}$.


\medskip
\noindent
{\it Fourier transforms}.
We denote by 
\begin{align}\label{FT}
\widehat{f}(\xi) = \widehat{\mathcal{F}}(f)(\xi) := \frac{1}{(2\pi)^{d/2}} \int_{\R^d} e^{-ix \cdot \xi} f(x) \, dx
\end{align}
the standard Fourier transform of $f$, and by 
\begin{align}\label{dFT}
\widetilde{f}(\xi) = \widetilde{\mathcal{F}}(f):= \frac{1}{\sqrt{2\pi}} \int_{\R} \overline{e(x,\xi)} f(x) \, dx
\end{align}
the distorted Fourier transform (dFT) of $f$ associated to the linearized operator 
$\mathcal{L}$ in \eqref{introlinop}, where $e(x,\xi)$ is given in \eqref{eformula};
see Section \ref{appdFT} for a brief intro to the dFT.
We have the standard inversion formula 
\[
\mathcal{F}^{-1}\phi(x)=\dfrac{1}{\sqrt{2\pi}}\int e^{ix\xi}\phi(\xi)d\xi.
\]
If we define 
\[
\widetilde{\mathcal{F}}^{-1}\phi(x)=\dfrac{1}{\sqrt{2\pi}} \int e(x,\xi) \phi(\xi)d\xi
\]
we have $\widetilde{\mathcal{F}}^{-1} \circ \widetilde{\mathcal{F}} = \mathrm{id}_{L^2_c}$,
where $L^2_c$ is the projection onto the continuous spectral subspace of $L^2(\R)$ relative to $H=-\partial_x^2+V$.


\medskip
\noindent
Note that with our definition
\begin{align}\label{ftsech}
\widehat{\mathcal{F}}[\sech(x)](\xi) = \sqrt{\tfrac{\pi}{2}} \sech\big( \tfrac{\pi}{2} \xi \big).
\end{align}

\medskip
\noindent
{\it Cutoffs}.
We fix a smooth even cutoff function  $\eta: \R \to [0,1]$ 
supported in $[-8/5,8/5]$ and equal to $1$ on $[-5/4,5/4]$.
For $k \in \mathbb{D} := 2^\Z$, we define $\eta_k(x) := \eta(k^{-1}x) - \eta(2k^{-1}x)$, 
so that the family $(\eta_k)_{k \in\Z}$ forms a partition of unity,
\begin{equation*}
 \sum_{k\in\mathbb{D}}\eta_k(\xi)=1, \quad \xi \neq 0.
\end{equation*}
We let
\begin{align}\label{cut0}
\eta_{I}(x) := \sum_{k \in I \cap \mathbb{D}}\eta_k, \quad \text{for any} \quad I \subset \R, \quad
\eta_{\leq a}(x) := \eta_{(-\infty,a]}(x), \quad \eta_{> a}(x) = \eta_{(a,\infty)}(x),
\end{align}
with similar definitions for $\eta_{< a}$ and $\eta_{\geq a}$.

\def\wt{\widetilde}

\smallskip
\noindent
We denote by $\mathbb{P}_k$, the standard Littlewood-Paley projections associated to the cutoff $\eta_k$,
that is, $\mathbb{P}_k f = \mathcal{F}^{-1} \eta_k \mathcal{F}$, and 
$\mathbb{P}_{\leq k} f = \mathcal{F}^{-1} \eta_{\leq k} \mathcal{F}$.
Analogously, we let $\widetilde{\mathbb{P}}_k$ denote 
the Littlewood-Paley projections adapted to the distorted Fourier transform:
\begin{equation}\label{defLP}
\wt{\mathbb{P}_k f}(\xi) = \eta_k(\xi) \wt{f}(\xi), \quad \wt{\mathbb{P}_{\leq k} f}(\xi) 
  = \eta_{\leq k}(\xi) \wt{f}(\xi), \quad \textrm{ etc.}
\end{equation}
Note that we are adopting the somewhat unconventional choice of using directly the dyadic index
(as opposed to the corresponding integer exponent)
as an index in the definition of the cutoffs and corresponding projections.

\medskip
\noindent
{\it Other conventions}.
We will often use a dot symbol `$\cdot$' to distinguish the bounds on various quantities 
involved in multilinear estimates; 
this is meant to help the reader navigate some of the bounds.

\smallskip
\subsection*{Acknowledgements}
J.M.P. 
and
F. P. are 
supported in part by a start-up grant from the University of Toronto, 
and NSERC grant RGPIN-2018-06487.

\bigskip
\section{Decomposition}

Let $\vec{\boldsymbol{\phi}}(t,x)=(\phi_1,\phi_2)$ be a solution to equation \eqref{KG} satisfying the stable manifold hypothesis \eqref{stable_manifold_hyp_thm} for some $\delta>0$ small. We decompose the solution in the following fashion \begin{align}\label{deco_phi}
\phi_1(t,x)&=Q(x)+a_1(t)Y(x)+\varepsilon_1(t,x),
\\ \phi_2(t,x)&=a_2(t)Y(x)+\varepsilon_2(t,x)\label{deco_phi_t},
\end{align}
where the coefficients $a_1(t)$ and $a_2(t)$ are explicitly given by 
\begin{align*}
a_1(t)&:=\langle \phi_1(t,x)-Q(x),Y(x)\rangle_{L^2_x} 
	\qquad \hbox{and}\qquad  a_2(t):=\langle \phi_2(t,x),Y(x)\rangle_{L^2_x}.
%
%
%
%
\end{align*}
Note that, by definition of the decomposition, the perturbative term $(\varepsilon_1,\varepsilon_2)$ satisfies 
\begin{align}\label{ort_eps_Y}
\langle \varepsilon_1(t),Y\rangle_{L^2_x} =\langle  \varepsilon_2(t),Y\rangle_{L^2_x}=0.
\end{align}
We seek to prove some control over the growth of the energy over long periods of time. 
Let us start by finding the equations 
satisfied by these new variables $(\varepsilon_1,\varepsilon_2,a_1,a_2)$.
Differentiating \eqref{deco_phi}-\eqref{deco_phi_t} in time and 
then taking the inner product of the resulting equations against $Y(x)$, using 
\eqref{ort_eps_Y}, we obtain that $(a_1,a_2)$ satisfies the system \begin{align*}
\begin{cases}
\dot{a}_1(t)=a_2(t)
\\ 
\dot{a}_2(t)=3a_1(t) 
  + \langle \mathcal{N},Y\rangle,
\end{cases}
\end{align*} 
where 
\begin{align}\label{definition_N_nonlinearity_epsilon}
\begin{split}
\mathcal{N} = \mathcal{N}(t,x) & := (Q+a_1Y+\varepsilon_1)^3-Q^3-3a_1Q^2Y-3Q^2\varepsilon
\\ 
& = (a_1Y+\varepsilon_1)^2(3Q+a_1Y+\varepsilon_1),
\end{split}
\end{align}
having used used that 
\[
Q''-Q+Q^3=0, \qquad \mathcal{L}Y=-3Y, \quad \hbox{ and hence }\quad \langle \mathcal{L}\varepsilon,Y\rangle=0.
\]
For later use we also write
\begin{align}\label{Nepsilon'}
\begin{split}
\mathcal{N} = \mathcal{N}(t,x) & = \varepsilon_1^3 + \varepsilon_1^2 (3Q + 3a_1Y) + \varepsilon_1 (6Qa_1Y + 3(a_1Y)^2) 
+ (a_1Y)^2 (3Q + a_1Y),
\end{split}
\end{align}
which will be useful when performing estimates according to the different degrees of homogeneity
in $\varepsilon_1$.


In a similar fashion, differentiating system \eqref{deco_phi}-\eqref{deco_phi_t}
with respect to time, using the identities above, 
we infer that $\vec{\boldsymbol{\varepsilon}}=(\varepsilon_1,\varepsilon_2)$ 
satisfies the following system 
\begin{align}\label{system_epsilon}
\begin{cases}
\varepsilon_{1,t}=\varepsilon_2,
\\ \varepsilon_{2,t}=-\mathcal{L}\varepsilon_1+\mathcal{N}-\langle \mathcal{N},Y\rangle Y.
\end{cases}
\end{align}

Next, with the notation from \eqref{defLP} we define (see also \eqref{ideasparam})
\begin{align}\label{eps><}
\varepsilon_{i,\geq}:=\widetilde{\mathbb{P}}_{\geq \kappa^{-1}}\varepsilon_i,  
  \quad \varepsilon_{i,<} := \widetilde{\mathbb{P}}_{<\kappa^{-1}}\varepsilon_i, \quad i=1,2,
  \qquad \kappa(t) = \log^{1+\alpha}(e^D+t).
\end{align}
with $D>1$ a large absolute constant to be determined in the course of the argument. 
Then, the frequency-localized variables satisfy the systems
\begin{align}\label{system_epsilon_hf}
\begin{cases}\varepsilon_{1,\geq,t}=\varepsilon_{2,\geq}
  + \dot{\widetilde{\mathbb{P}}}_{\geq\kappa^{-1}}\varepsilon_1,
\\ 
\varepsilon_{2,\geq,t} = -\widetilde{\mathbb{P}}_{\geq \kappa^{-1}}\mathcal{L}\varepsilon_{1}
  +\widetilde{\mathbb{P}}_{\geq \kappa^{-1}}\mathcal{N}-\langle \mathcal{N},Y\rangle\, 
  \widetilde{\mathbb{P}}_{\geq \kappa^{-1}}Y+\dot{\widetilde{\mathbb{P}}}_{\geq\kappa^{-1}}\varepsilon_2,
\end{cases}
\end{align}
where $\dot{\widetilde{\mathbb{P}}}_{\geq\kappa^{-1}}$ is the 
distorted projection with symbol $\partial_t(\eta_{\geq \kappa^{-1}})$,
and, with the obvious analogous notation,
\begin{align}\label{system_epsilon_lw}
\begin{cases}
\varepsilon_{1,<,t}=\varepsilon_{2,<}+\dot{\widetilde{\mathbb{P}}}_{<\kappa^{-1}}\varepsilon_1,
\\ 
\varepsilon_{2,<,t}=-\widetilde{\mathbb{P}}_{< \kappa^{-1}}\mathcal{L}\varepsilon_{1} 
  + \widetilde{\mathbb{P}}_{< \kappa^{-1}}\mathcal{N}-\langle \mathcal{N},Y\rangle\, 
  \widetilde{\mathbb{P}}_{< \kappa^{-1}}Y+\dot{\widetilde{\mathbb{P}}}_{<\kappa^{-1}}\varepsilon_2.
\end{cases}
\end{align}


\bigskip
\section{Virial Identities for the original problem}\label{original_virial_section}
In this section we establish the first preliminary virial identities required in our analysis.  The main conclusion is contained in \eqref{final_2_og_virial_hf} with \eqref{def_r1_og_virial_fh}
where the main virial functionals $\mathcal{I}_\geq$ and $\mathcal{P}_\geq$
are defined in \eqref{I>} and \eqref{true_momentum_original_virial_high_frequency}.
These functional concern the ``high frequencies'' of the perturbation.
The ``low frequencies'' are handled directly in Lemma \ref{lemma<}.

\subsection{Set-up and low frequencies}
We consider $\chi:\R\to\R$, a smooth even function satisfying
\[
\chi\equiv 1 \ \hbox{ on } \ [-1,1], \qquad \chi\equiv 0 \ \hbox{ on } (-2,2)^c 
  \quad \hbox{ and } \quad \chi'(x)\leq 0 \ \hbox{ for all } \ x\geq0.
\]
For $A>0$, we define $\zeta_A(x)$, 
a smooth approximation of $\exp(-\vert x\vert/A)$, as 
\begin{align}\label{def_zeta_A}
\zeta_A(x):=\exp\Big(-\tfrac{1}{A}(1-\chi(x))\vert x\vert \Big).
\end{align}
Then, we define the weight function $\Phi_A(x)$ as a bounded approximation of the identity:
\begin{align}\label{def_phi_A}
\Phi_A(x):= \int_0^x \zeta_A^2(y)dy.
\end{align}
As mentioned in Section \ref{heuristics_sec}, see \eqref{ideasparam}, 
we consider the parameters 
\begin{align*}
A\gg1, \qquad \mu(t)=\langle t\rangle , \qquad \kappa(t) = \log^{1+\alpha}(e^D+t)
  \quad \hbox{ and } \quad T_{\mathrm{max}}=e^{1/\delta^{\frac43(1-\alpha)}},
\end{align*}
with $\alpha\in(0,1)$ being fixed arbitrarily small. 
With this choice $\mu^{-1} \kappa^{-1} \in L^1_t([0,\infty))$.

With all the above notation and definitions, 
we introduce the frequency-localized \textit{modified momentum functional} $\mathcal{I}_\geq(t)$, which is given by  
\begin{align}
\label{I>}
\mathcal{I}_\geq & := \dfrac{1}{\mu(t)} \int_\R\left(\Phi_A(x)\partial_x\varepsilon_{1,\geq}(t,x)
  + \tfrac{1}{2}\Phi_A'(x)\varepsilon_{1,\geq}(t,x)\right)\varepsilon_{2,\geq}(t,x) dx.
\end{align}
We shall also need the following 
\begin{align}
\label{true_momentum_original_virial_high_frequency}
\mathcal{P}_\geq&:=\dfrac{1}{\mu(t)}\int_\R \partial_x\varepsilon_{1,\geq}(t,x)\varepsilon_{2,\geq}(t,x)\Psi\big(x\big)dx,
\qquad \Psi(x) := \tanh(x),
\end{align}
that, from now on, will be referred to as the \textit{true momentum} (with weight $\Psi$).

\medskip
We begin our analysis with a simple lemma which takes care of 
the local energy of the perturbation for frequencies smaller than $\kappa^{-1}$:

\begin{lem}\label{lemma<}
Let $(\varepsilon_1,\varepsilon_2)\in C(\R,H^1\times L^2)$ be any solution to system \eqref{system_epsilon} 
satisfying the stable manifold hypothesis \eqref{stable_manifold_hyp_thm}.
Then, for all times $t\in\R$, the following holds
\begin{align}\label{lemma<conc}
\int_0^{\infty}
  \dfrac{1}{\mu(t)}\int_\R \big(\varepsilon_{2,<}^2\sech^2(x)
  + \varepsilon_{1,<,x}^2\sech^2\big(\tfrac{x}{A}\big) + \varepsilon_{1,<}^2\sech^2(x)\big)dxdt = O(\delta^2).
\end{align}
\end{lem}

\begin{proof}
Denoting $f_< = \widetilde{\mathbb{P}}_{< \kappa^{-1}} f$ 
and using the distorted version of Bernstein's inequality (see (i) of Remark \ref{remdFTadd}), we can estimate
\begin{align*}
\int_\R f_{<}^2 \sech^2(x) \lesssim {\| \widetilde{\mathbb{P}}_{< \kappa^{-1}} f \|}_{L^\infty}^2 
  \lesssim \big( \kappa^{-1/2} {\| f \|}_{L^2} \big)^2.
\end{align*}
With $f = \varepsilon_1,\varepsilon_2$, and since $\mu^{-1} \kappa^{-1} \in L^1_t$
we obtain the desired bound for the first and third terms in the integral \eqref{lemma<conc}.

To see that the same estimate holds for $\partial_x \varepsilon_{1,<}$ we use the boundedness
of wave operators $\mathcal{W} = \widetilde{\mathcal{F}}^{-1} \widehat{\mathcal{F}} $ 
and $\mathcal{W}^{-1} = \mathcal{W}^\ast = \widehat{\mathcal{F}}^{-1} \widetilde{\mathcal{F}}$  
on $W^{1,\infty}$ and $L^2$, see (i) of Remark \ref{remdFTadd}, to obtain
\begin{align*}
{\| \partial_x \varepsilon_{1,<} \|}_{L^\infty} 
  \leq {\| \widetilde{\mathbb{P}}_{< \kappa^{-1}} \varepsilon_1 \|}_{W^{1,\infty}} 
  \lesssim {\| \mathbb{P}_{< \kappa^{-1}} \mathcal{W}^\ast \varepsilon_1 \|}_{W^{1,\infty}}
  \lesssim \kappa^{-1/2} {\| \varepsilon_1 \|}_{L^2}.
\end{align*}
It follows that
\begin{align*}
\int_\R \varepsilon_{1,<,x}^2\sech^2\big(\tfrac{x}{A}\big) \lesssim {\| \partial_x \varepsilon_{1,<} \|}_{L^\infty}^2 
  \lesssim \kappa^{-1} {\| \varepsilon_1 \|}_{L^\infty}^2
\end{align*}
where the implicit constant depends only on $A$.
Then we obtain \eqref{lemma<conc} using again that $\mu^{-1} \kappa^{-1} \in L^1_t$.
\end{proof}

%
%

\medskip
\begin{rem}[Notation for `high' and `low' frequencies]
From now on we shall refer to all frequencies with magnitude greater 
than $\kappa^{-1}(t)$ as high (or `large') frequencies
and to frequencies less than $\kappa^{-1}(t)$ as low (or `small') frequencies. 
Note that these are not high or low in any standard sense. 
\end{rem}

\medskip 
\subsection{High-frequency case}
In contrast with the low frequency case, 
in this case we cannot  directly prove time integrability of weighted norms 
for $(\varepsilon_{1,\geq},\varepsilon_{2,\geq})$. 
Instead, we just seek to prove an inequality of the form \begin{align*}
\dfrac{d}{dt}\mathcal{I}_{\geq}&\lesssim -\int \varepsilon_{1,\geq,x}^2\Phi_A'
+\int \varepsilon_{1,\geq}^2\sech^{1/2}(x)+\text{remaining terms}.
\end{align*}
The following lemma provides us with the first virial identity required to obtain this.

\begin{lem}\label{original_virial_indetity_lem} 
Let $(\varepsilon_1,\varepsilon_2)\in C(\R,H^1\times L^2)$ 
be any solution to the system \eqref{system_epsilon} and $(\varepsilon_{1,\geq},\varepsilon_{2,\geq})$ defined as in \eqref{eps><}. 
Then, with $\mathcal{I}_{\geq}$ as in \eqref{I>}, the following identity holds for all times $t\in\R$: 
\begin{align}
\dfrac{d}{dt}\mathcal{I}_\geq & =-\dfrac{1}{\mu}\int\varepsilon_{1,\geq,x}^2\Phi_A'   
  -  \dfrac{3}{\mu}\int  \varepsilon_{1,\geq}^2QQ'\Phi_A      
  +\dfrac{1}{4\mu}\int \varepsilon_{1,\geq}^2\Phi_A'''\label{original_virial_indetity_formula}
\\ 
& \qquad +  \dfrac{1}{\mu} \int \left(\varepsilon_{1,\geq ,x}\Phi_A
 + \tfrac{1}{2}\varepsilon_{1,\geq }\Phi_A'\right)
 \Big(\widetilde{\mathbb{P}}_{\geq \kappa^{-1}}\mathcal{N} 
 - \langle \mathcal{N},Y\rangle\widetilde{\mathbb{P}}_{\geq \kappa^{-1}}Y \Big)
 + \dfrac{1}{\mu}\mathsf{P}_{1,\geq}(t)-\dfrac{\mu'}{\mu}\mathcal{I}_\geq,                                        \nonumber %
\end{align}
where $\mathsf{P}_{1,\geq}(t)$ is explicitly given by
\begin{align}\label{P1}
\mathsf{P}_{1,\geq}(t):= \int \left(\Phi_A \dot{\widetilde{\mathbb{P}}}_{\geq \kappa^{-1}}\varepsilon_{1,x}   
  + \tfrac{1}{2}\Phi_A'\dot{\widetilde{\mathbb{P}}}_{\geq\kappa^{-1}}\varepsilon_1\right)\varepsilon_{2,\geq}
  + \int \left(\Phi_A \varepsilon_{1,\geq,x} 
  + \tfrac{1}{2}\Phi_A'\varepsilon_{1,\geq}\right)\dot{\widetilde{\mathbb{P}}}_{\geq\kappa^{-1}}\varepsilon_{2}.
\end{align}

\end{lem}

\begin{proof}
The proof is a direct computation. Compared to the standard case we have additional terms coming from differentiating the time dependent frequency localization
and weights. In fact, taking the time-derivative of \eqref{I>}, 
using system \eqref{system_epsilon_lw}, and performing
several integration by parts we obtain 
\begin{align*}
\dfrac{d}{dt}\mathcal{I}_\geq 
& = \dfrac{1}{\mu}\int\left(\Phi_A\varepsilon_{2,\geq,x}+\tfrac{1}{2}\Phi_A'\varepsilon_{2,\geq}\right)\varepsilon_{2,\geq}  
  +  \dfrac{1}{\mu}\int \left(\Phi_A\varepsilon_{1,\geq,x}+\tfrac{1}{2}\Phi_A'\varepsilon_{1,\geq}\right)\varepsilon_{2,\geq,t}
\\ 
& \quad  + \dfrac{1}{\mu}\int \left(\Phi_A \dot{\widetilde{\mathbb{P}}}_{\geq \kappa^{-1}}\varepsilon_{1,x} 
  + \tfrac{1}{2}\Phi_A'\dot{\widetilde{\mathbb{P}}}_{\geq\kappa^{-1}}\varepsilon_1\right)\varepsilon_{2,\geq} -\dfrac{\mu'}{\mu^2}\int \left(\Phi_A\varepsilon_{1,\geq ,x}
  + \tfrac{1}{2}\Phi_A'\varepsilon_{1,\geq }\right)\varepsilon_{2,\geq}
\\
& = 
-  \dfrac{1}{\mu}  \int \left(\Phi_A\varepsilon_{1,\geq ,x} + \tfrac{1}{2}\Phi_A'\varepsilon_{1,\geq }\right)
  \widetilde{\mathbb{P}}_{\geq \kappa^{-1}}\mathcal{L}\varepsilon_1  
\\ 
& \quad + \dfrac{1}{\mu}  \int \left(\Phi_A\varepsilon_{1,\geq ,x}+\tfrac{1}{2}\Phi_A'\varepsilon_{1,\geq }\right)
  \widetilde{\mathbb{P}}_{\geq \kappa^{-1}}N
  -  \dfrac{1}{\mu} \langle \mathcal{N},Y\rangle\int \left(\Phi_A\varepsilon_{1,\geq ,x} 
  + \tfrac{1}{2}\Phi_A'\varepsilon_{1,\geq }\right)\widetilde{\mathbb{P}}_{\geq \kappa^{-1}}Y 
%
\\ 
& \quad + \dfrac{1}{\mu}\int \left(\Phi_A\dot{\widetilde{\mathbb{P}}}_{\geq \kappa^{-1}}\varepsilon_{1,x}
  + \tfrac{1}{2}\Phi_A'\dot{\widetilde{\mathbb{P}}}_{\geq\kappa^{-1}}\varepsilon_1\right)\varepsilon_{2,\geq}
  + \dfrac{1}{\mu}\int \left(\Phi_A \varepsilon_{1,\geq,x}
  + \tfrac{1}{2}\Phi_A'\varepsilon_{1,\geq}\right)\dot{\widetilde{\mathbb{P}}}_{\geq\kappa^{-1}}\varepsilon_{2}                  
\\ 
& \quad -\dfrac{\mu'}{\mu^2}\int \left(\Phi_A\varepsilon_{1,\geq ,x} 
  + \tfrac{1}{2}\Phi_A'\varepsilon_{1,\geq }\right)\varepsilon_{2,\geq}.
\end{align*}
Notice that we have used 
\[
\int\left(\Phi_A\varepsilon_{2,\geq,x}+\tfrac{1}{2}\Phi_A'\varepsilon_{2,\geq}\right)\varepsilon_{2,\geq}=0.
\]
Using that $\widetilde{\mathbb{P}}_{\geq\kappa^{-1}}\mathcal{L}=\mathcal{L}\widetilde{\mathbb{P}}_{\geq\kappa^{-1}}$, 
after a few integration by parts we get
\begin{align*}
&-\int \left(\varepsilon_{1,\geq,x}\Phi_A+\tfrac{1}{2}\varepsilon_{1,\geq}\Phi_A'\right)
  \widetilde{\mathbb{P}}_{\geq \kappa^{-1}}\mathcal{L}\varepsilon_{1}
\\ & \qquad \qquad  = -\int \left(\varepsilon_{1,\geq,x}\Phi_A
  +\tfrac{1}{2}\varepsilon_{1,\geq}\Phi_A' \right)\big(-\varepsilon_{1,\geq,xx}+\varepsilon_{1,\geq}-3Q^2 \varepsilon_{1,\geq}\big)
\\ 
& \qquad \qquad = -\int\varepsilon_{1,\geq,x}^2\Phi_A'+\dfrac{1}{4}\int \varepsilon_{1,\geq}^2\Phi_A'''
 - 3\int  \varepsilon_{1,\geq}^2QQ'\Phi_A.
\end{align*}
This completes the proof.
\end{proof}

We remark that, since it is not evident that all terms are negligible, in order to distinguish between 
``good'' and ``bad'' terms, in this case signs shall be important. Now we seek to bound all terms appearing in 
\eqref{original_virial_indetity_formula} except for the first two of them.

We shall require some additional definitions. 
First of all, we introduce the localized (in space) version of $\varepsilon_{1,\geq}$,
\begin{align}\label{w_def}
w(t,x):=\varepsilon_{1,\geq}\sqrt{\Phi_A'}.
\end{align}
Note that $w$ is no longer localized in frequency. 
This change of variables 
will help us to hide some terms in the virial estimates. 
From the 
definition it follows that
\begin{align}\label{wx_identity}
w_{x}=\varepsilon_{1,\geq,x}\sqrt{\Phi_A'}+\varepsilon_{1,\geq}\cdot\tfrac{\Phi_A''}{2\sqrt{\Phi_A'}}.
\end{align}
Therefore, gathering the first and third term in \eqref{original_virial_indetity_formula} 
for the $\varepsilon_{1,\geq}$ variable, we obtain that 
\begin{align*}
-\int\varepsilon_{1,\geq,x}^2\Phi_A'+\dfrac{1}{4}\int\varepsilon_{1,\geq}^2\Phi_A''' & 
=- \int w_{x}^2+\dfrac{1}{2} \int \partial_x\big(\varepsilon_{1,\geq}^2\big)\Phi_A''
+\dfrac{1}{4}\int \varepsilon_{1,\geq}^2\tfrac{(\Phi_A'')^2}{\Phi_A'}
+\dfrac{1}{4}\int \varepsilon_{1,\geq}^2\Phi_A'''
\\
&  = - \int w_{x}^2+\dfrac{1}{4}\int \varepsilon_{1,\geq}^2\tfrac{(\Phi_A'')^2}{\Phi_A'}
-\dfrac{1}{4}\int \varepsilon_{1,\geq}^2\Phi_A'''.
\end{align*}
Thus, we can group the first and third term in \eqref{original_virial_indetity_formula} 
using the identity above to 
write them as 
\begin{align}\label{eps_geq_into_w_id}
-\int\varepsilon_{1,\geq,x}^2 \Phi'_A + 
  \dfrac{1}{4}\int\varepsilon_{1,\geq}^2 \Phi'''_A
= -\int w_{x}^2 
  + 
 	\dfrac{1}{4}\int w^2\left(\dfrac{(\Phi_A'')^2}{(\Phi_A')^2}-\dfrac{\Phi_A'''}{\Phi_A'}\right).
\end{align}
One advantage of working with $\Phi_A'$ being an approximation of $e^{-\vert x\vert/A}$ 
is that the parenthesis in the latter identity is identically zero,
except possibly where $\Phi_A'$ does not match the exponential. 
Indeed, going back to the definition of $\zeta_A$ in \eqref{def_zeta_A}, 
we calculate
\begin{align*}
\Phi_A''&=\tfrac{2}{A}\big(\chi'(x)\vert x\vert -(1-\chi(x))\sgn(x)\big)\zeta_A^2(x),
\\ 
\Phi_A'''&=\tfrac{4}{A^2}\big(\chi'(x)\vert x\vert -(1-\chi(x))\sgn(x)\big)^2\zeta^2_A(x) 
	+ \tfrac{2}{A}\big(\chi''(x)\vert x\vert+2\chi'(x)\sgn(x)\big)\zeta_A^2(x),
\end{align*}
and find that, for any $A>0$, 
\[
\dfrac{(\Phi_A'')^2}{(\Phi_A')^2}-\dfrac{\Phi_A'''}{\Phi_A'}=0 
  \quad \hbox{ for all } \quad x\in (-\infty,-2]\cup [-1,1]\cup [2,+\infty).
\]
Moreover, due to the explicit definition of $Q$, 
we see that, for $A\gg1$,
\begin{align*}
0\leq -3QQ'\tfrac{\Phi_A}{\Phi_A'} & \leq 3\sech^{1/2}(x) \quad \hbox{ for all } \quad x\in\R.
\end{align*}
Therefore, plugging this information into \eqref{original_virial_indetity_formula}, 
using \eqref{eps_geq_into_w_id}, 
we can write 
\begin{align}\label{original_virial_w_form_inequality}
\begin{split}
\dfrac{d}{dt}\mathcal{I}_\geq&\leq -\dfrac{1}{\mu}\int w_{x}^2
+\dfrac{4}{\mu}\int w^2\sech^{1/2}(x)+ \dfrac{1}{\mu}\mathsf{P}_{1,\geq}(t)
-\dfrac{\mu'}{\mu}\mathcal{I}_\geq
\\ & \qquad +\dfrac{1}{\mu}\int \left(\varepsilon_{1,\geq ,x}\Phi_A
+\tfrac{1}{2}\varepsilon_{1,\geq }\Phi_A'\right)\Big(\widetilde{\mathbb{P}}_{\geq \kappa^{-1}} \mathcal{N}
-\langle \mathcal{N},Y\rangle\widetilde{\mathbb{P}}_{\geq \kappa^{-1}}Y \Big).
\end{split}
\end{align}

\smallskip
\subsubsection{Estimates of the nonlinear terms}
We now deal with the nonlinear terms that appear in $\mathcal{N}$, see \eqref{definition_N_nonlinearity_epsilon}.

\begin{lem}\label{lemma_N_eps_geq_original_virial} 
Let $(\varepsilon_1,\varepsilon_2)\in C(\R,H^1\times L^2)$ be any solution to system \eqref{system_epsilon} 
satisfying the stable manifold hypothesis \eqref{stable_manifold_hyp_thm} and $\epsilon\in(0,1)$ small. 
Then, for all times $t\in\R$, the following holds
\begin{align}\label{lemma_N_original_conc} 
&\left\vert \int \left(\varepsilon_{1,\geq ,x}\Phi_A + \tfrac{1}{2}\varepsilon_{1,\geq }\Phi_A'\right)
  \widetilde{\mathbb{P}}_{\geq \kappa^{-1}}\mathcal{N} \right\vert   
  \lesssim \delta a_1^2(t) + \epsilon 
  \int w^2\sech^{1/2}(x)+O(\delta^4)+O(\delta^3/\kappa^{1/2}).
\end{align}
\end{lem}

Note that reason we wrote $O(\delta^4)+O(\delta^3/\kappa^{1/2})$ instead of just $O(\delta^3/\kappa^{1/2})$ is because we stated the lemma 
holding for all $t\in\R$ instead of $t\in(0,T_\mathrm{max})$. 
Also note that both of these terms, after being multiplied by $\mu^{-1}$,
are integrable up to (optimal) times of $O(e^{1/\delta^{2-}})$.


\begin{proof}[Proof of Lemma \ref{lemma_N_eps_geq_original_virial}] 
Recall 
that
\[
\mathcal{N}= (a_1Y+\varepsilon_1)^2(3Q+a_1Y+\varepsilon_1), \qquad Q\approx\sech(x)\quad 
  \hbox{ and } \quad Y\approx \sech^2(x).
\]
and that we can also write $\mathcal{N}$ as in \eqref{Nepsilon'}. 
We begin by using the decomposition 
\begin{align}\label{N_original_lemma_high_into_I_Pminus}
\widetilde{\mathbb{P}}_{\geq \kappa^{-1}} = \mathrm{id} - \widetilde{\mathbb{P}}_{< \kappa^{-1}}
\end{align}
for the projector acting on $\mathcal{N}$, 
and first handle the low-frequencies contributions. 
In this case we can integrate by parts to obtain
\begin{align*}
&\left\vert \int \left(\varepsilon_{1,\geq ,x}\Phi_A+\tfrac{1}{2}\varepsilon_{1,\geq }\Phi_A'\right)
  \widetilde{\mathbb{P}}_{< \kappa^{-1}} \mathcal{N} \right\vert 
  \leq \left\vert \int \varepsilon_{1,\geq }\Phi_A'\widetilde{\mathbb{P}}_{< \kappa^{-1}} \mathcal{N} \right\vert 
  +\left\vert \int \varepsilon_{1,\geq }\Phi_A \partial_x\widetilde{\mathbb{P}}_{< \kappa^{-1}} \mathcal{N} \right\vert .
\end{align*}
Then, we apply Cauchy-Schwarz,
the boundedness of wave operators on $H^1$ and $L^1$, followed by Bernstein inequality, to estimate
(the implicit constants may depend on $A$)
\begin{align*}
& \left\vert \int \varepsilon_{1,\geq }\Phi_A'\widetilde{\mathbb{P}}_{< \kappa^{-1}} \mathcal{N} \right\vert
  + \left\vert \int \varepsilon_{1,\geq }\Phi_A \partial_x \widetilde{\mathbb{P}}_{< \kappa^{-1}} \mathcal{N} \right\vert
  \\
  &  \lesssim {\| \varepsilon_{1,\geq } \|}_{L^2} {\big\| \widetilde{\mathbb{P}}_{< \kappa^{-1}}\mathcal{N} \big\|}_{H^1}
  \lesssim \delta \cdot {\big\| \mathbb{P}_{< \kappa^{-1}} \mathcal{W}\mathcal{N}\big\|}_{H^1} 
  \\
  &  \lesssim \delta \cdot \kappa^{-1/2} {\|\mathcal{N}\big\|}_{L^1} 
  \lesssim \delta \cdot \kappa^{-1/2} \delta^2;
\end{align*}
therefore these terms are accounted for on the right-hand side of \eqref{lemma_N_original_conc}, as desired.
We are then left with estimating the terms
that corresponds to `$\mathrm{id}$' in \eqref{N_original_lemma_high_into_I_Pminus}, that is,
\begin{align}\label{lemma_N_pr1} 
&\left\vert \int \left(\varepsilon_{1,\geq ,x}\Phi_A + \tfrac{1}{2}\varepsilon_{1,\geq }\Phi_A'\right) \mathcal{N} \right\vert   
\end{align}
First, let us consider the terms in \eqref{Nepsilon'} that are cubic in $\varepsilon_1$.
These can be bounded immediately 
from Cauchy-Schwarz inequality:
\begin{align*}
\left\vert\int\left(\varepsilon_{1,\geq ,x}\Phi_A+\tfrac{1}{2}\varepsilon_{1,\geq }\Phi_A'\right)
  \varepsilon_{1}^3\right\vert
  & \lesssim \big( A {\| \partial_x \varepsilon_{1,>} \|}_{L^2} + {\| \varepsilon_{1,>} \|}_{L^2} \big) 
  {\| \varepsilon_{1}^3 \|}_{L^2}
  \\
  & \lesssim A {\| \varepsilon_1 \|}_{H^1} \cdot {\| \varepsilon_{1}^3 \|}_{L^2} \lesssim A \delta^4,
\end{align*}
having used the boundedness of wave operators in $H^1$ to estimate $\partial_x \varepsilon_{1,>}$ in $L^2$.

%


For the quadratic terms, namely, $(3a_1Y+3Q)\varepsilon^2$, 
we write $\varepsilon_1^2=(\varepsilon_{1,<}+\varepsilon_{1,\geq})^2$ 
and focus first on the term $\varepsilon_{1,\geq}^2Q$.
Integrating by parts, using H\"older's inequality, and recalling the definition of $w$ in \eqref{w_def},
we get that 
\begin{align*}
\left\vert \int \left(\varepsilon_{1,\geq ,x}\Phi_A + \tfrac{1}{2}\varepsilon_{1,\geq }\Phi_A'\right)
  Q\varepsilon_{1,\geq}^2 \right\vert
  & \lesssim \int Q | \varepsilon_{1,\geq} |^3   
  \lesssim \delta\int w^2\sech^{1/2}(x).
\end{align*}
The term $3a_1\varepsilon_{1,\geq}^2Y$ can be treated identically so we skip it.

The remaining quadratic terms contain at least one low frequency projection of $\varepsilon_1$,
which will be enough to obtain acceptable contributions.
Indeed, using Cauchy-Schwarz and H\"older followed by (distorted) Bernstein, 
we have 
\begin{align*}
\left\vert \int \left(\varepsilon_{1,\geq ,x}\Phi_A+\tfrac{1}{2}\varepsilon_{1,\geq }\Phi_A'\right)
  \big(a_1Y+3Q\big)\big(\varepsilon_{1,<}^2+2\varepsilon_{1,<}\varepsilon_{1,\geq}\big) \right\vert 
  \\
  \lesssim {\| \varepsilon_{1,\geq} \|}_{H^1} \cdot {\| \varepsilon_{1} \|}_{L^2} {\| \varepsilon_{1,<} \|}_{L^\infty} 
  \lesssim \delta^3 \kappa^{-1/2},
\end{align*}
which is consistent with the desired \eqref{lemma<conc}. 

Next, we consider the terms of order one in $\varepsilon_{1}$ coming from $\mathcal{N}$, 
and estimate first the contribution from $\varepsilon_{1,\geq}$;
integrating by parts we obtain
\begin{align*}
\left\vert \int \left(\varepsilon_{1, \geq ,x}\Phi_A+\tfrac{1}{2}\varepsilon_{1,\geq }\Phi_A'\right) 
  \big(2a_1YQ+a^2Y^2\big)\varepsilon_{1,\geq} \right\vert
  & = \dfrac{1}{2}\left\vert \int \varepsilon_{1,\geq}^2\Phi_A\partial_x\big(2a_1YQ + a_1^2Y^2\big)\right\vert 
\\ 
& \lesssim \vert a_1\vert \int w^2\sech^{1/2}(x).
\end{align*}
For the corresponding term with the low frequency contribution $\varepsilon_{1,<}$
we use Cauchy-Schwarz, the boundedness of wave operators and Bernstein inequality 
similarly to before, to see that
\begin{align*}
\left\vert \int \left(\varepsilon_{1, \geq ,x}\Phi_A+\tfrac{1}{2}\varepsilon_{1,\geq }\Phi_A'\right)
  \big(2a_1YQ + a_1^2Y^2\big)\varepsilon_{1,<} \right\vert
  \\
  \lesssim {\| \varepsilon_{1, \geq} \|}_{H^1} \cdot |a_1| {\| \varepsilon_{1,<} Q\|}_{L^2}
  \lesssim \delta^3 \kappa^{-1/2}.
\end{align*}

Finally, we look at the terms coming from $\mathcal{N}$ that do not contain any $\varepsilon_1$
and, by the usual application of Cauchy-Schwarz and the boundedness of wave operators, we get
\begin{align*}
&\left\vert \int \left(\varepsilon_{1,\geq ,x}\Phi_A+\tfrac{1}{2}\varepsilon_{1,\geq }\Phi_A'\right)
  \big(3a_1^2Y^2Q+a_1^3Y^3\big) \right\vert \lesssim \delta a_1^2(t),
\end{align*}
Gathering all the above estimates we obtain \eqref{lemma<}. 
\end{proof}

%
%
%

\smallskip
The following lemma gives us control over the last term in  inequality \eqref{original_virial_w_form_inequality}.

\begin{lem}\label{lem_og_virial_hf_eps_Y} Let $(\varepsilon_1,\varepsilon_2)\in C(\R,H^1\times L^2)$ 
be any solution to system \eqref{system_epsilon} satisfying the stable manifold 
hypothesis \eqref{stable_manifold_hyp_thm} and $\epsilon\in(0,1)$ small.
Then, for all times $t\in\R$, the following holds
\begin{align}\label{estimate_Y_original_virial_hf}
\left\vert\langle \mathcal{N},Y\rangle \int \left(\varepsilon_{1,\geq ,x}\Phi_A
  + \tfrac{1}{2}\varepsilon_{1,\geq }\Phi_A'\right)\widetilde{\mathbb{P}}_{\geq \kappa^{-1}}Y \right\vert 
  & \lesssim 
  \delta a_1^2(t) + 
  \epsilon \int w^2\sech^{1/2}(x)+O(\delta^3/\kappa).
\end{align}
\end{lem}

\begin{proof}
First of all observe that $\langle \mathcal{N},Y\rangle$ is quadratic: In fact, 
from the stable manifold hypothesis \eqref{stable_manifold_hyp_thm} 
and the definition of $\mathcal{N}$, we see that 
\begin{align}\label{NY}
\begin{split}
\big\vert \langle \mathcal{N},Y\rangle\big\vert &\lesssim a_1^2(t)+\int \varepsilon_{1}^2\sech^{2}(x)
\\ & \lesssim a_1^2(t)+\int \varepsilon_{1,<}^2\sech^2(x)+\int w^2\sech^{1/2}(x).  
\end{split}
\end{align}
Then, in order to conclude, it is enough to observe that
\[
\left\vert  \int \left(\varepsilon_{1,\geq ,x}\Phi_A+\tfrac{1}{2}\varepsilon_{1,\geq }\Phi_A'\right)
  \widetilde{\mathbb{P}}_{\geq \kappa^{-1}}Y \right\vert \lesssim \delta.
\]
\end{proof}

The following corollary gathers and summarizes all the information obtained so far.

\begin{cor} 
Let $(\varepsilon_1,\varepsilon_2)\in C(\R,H^1\times L^2)$ be any solution to system \eqref{system_epsilon}
satisfying  \eqref{stable_manifold_hyp_thm} and $(\varepsilon_{1,\geq},\varepsilon_{2,\geq})$ defined as above. 
Then, for all times $t\in\R$ the following inequality holds
\begin{align}\label{hf_og_virial}
\begin{split}
\dfrac{d}{dt}\mathcal{I}_{\geq}&\leq -\dfrac{1}{\mu}\int w_x^2+\dfrac{5}{\mu}\int w^2\sech^{1/2}(x)
  + \dfrac{1}{\mu}a_1^2(t) \, O(\delta)
\\ 
& \qquad + \dfrac{1}{\mu}\mathsf{P}_{1,\geq}(t)-\dfrac{\mu'}{\mu}\mathcal{I}_\geq
+\dfrac{1}{\mu}O(\delta^4)+\dfrac{1}{\mu}O(\delta^3/\kappa^{1/2}).
\end{split}
\end{align}
\end{cor}

\subsubsection{Adding the `true' momentum}
To conclude this section we consider the functional $\mathcal{P}_\geq$ 
and add it to $\mathcal{I}_\geq$ for the purpose of adding the local norm of the time derivative
of the perturbation 
to the inequality \eqref{hf_og_virial}.
The main conclusion of this whole section is summarized in the following:

\begin{prop}\label{propOR}
Let $(\varepsilon_1,\varepsilon_2)\in C(\R,H^1\times L^2)$ be any solution to system \eqref{system_epsilon}
satisfying  \eqref{stable_manifold_hyp_thm} and $(\varepsilon_{1,\geq},\varepsilon_{2,\geq})$ defined as above. 
Let  $\mathcal{I}_\geq$  and $\mathcal{P}_\geq$  be defined 
as in \eqref{I>} and \eqref{true_momentum_original_virial_high_frequency}.
Then, for all times $t\in\R$ the following holds:
\begin{align}\label{final_2_og_virial_hf}
\begin{split}
& \dfrac{d}{dt}\Big(\mathcal{I}_\geq+\mathcal{P}_\geq\Big)
\\ & \leq 
  -\dfrac{1}{2\mu}\int w_t^2\sech^{1/2}(x)-\dfrac{C}{\mu}\int w_x^2 + \dfrac{1}{\mu}\int w^2\sech^{1/2}(x) 
  + O(\delta) \frac{1}{\mu} a_1^2 + \mathfrak{R}_1(t),
\end{split}
\end{align}
for some absolute $C>0$,
where 
\begin{align}\label{propORR1}
\int_0^{T_{\max}} \mathfrak{R}_1 (t) dt \leq \delta^2.
\end{align}
\end{prop}

\begin{proof}
First, from the definition \eqref{true_momentum_original_virial_high_frequency}, 
and following very similar computations as the ones in Lemma \ref{original_virial_indetity_lem},
one can see that 
\begin{align}\label{dp_dt}
\begin{split}
\dfrac{d}{dt}\mathcal{P}_\geq& = -\dfrac{1}{2\mu}\int \varepsilon_{2,\geq}^2\Psi'
 - \dfrac{1}{2\mu}\int\varepsilon_{1,\geq,x}^2\Psi' + \dfrac{1}{2\mu}\int \varepsilon_{1,\geq}^2\Psi' 
 - \dfrac{3}{2\mu}\int \varepsilon_{1,\geq}^2\partial_x\big(Q^2\Psi\big)
\\
& \quad  +  \dfrac{1}{\mu} \int \varepsilon_{1,\geq ,x}\Psi\Big(\widetilde{\mathbb{P}}_{\geq \kappa^{-1}}\mathcal{N}
-\langle \mathcal{N},Y\rangle\widetilde{\mathbb{P}}_{\geq \kappa^{-1}}Y\Big)
- \dfrac{\mu'}{\mu}\mathcal{P}_\geq + \mathsf{P}_{2,\geq}
%
\end{split}
\end{align}
where 
\begin{align}\label{P2def}
\mathsf{P}_{2,\geq} := \int \left(\dot{\widetilde{\mathbb{P}}}_{\geq \kappa^{-1}}(\varepsilon_{1,x})\varepsilon_{2,\geq}
  +\varepsilon_{1,x}\dot{\widetilde{\mathbb{P}}}_{\geq \kappa^{-1}}(\varepsilon_{2,\geq})\right)\Psi .
\end{align}

%
%

The terms in \eqref{dp_dt} involving the nonlinearity $\mathcal{N}$ can be dealt with 
by applying exactly the same proofs as in Lemmas 
\ref{lemma_N_eps_geq_original_virial} and \ref{lem_og_virial_hf_eps_Y}; in particular, we have
\begin{align*}
\left\vert \int \varepsilon_{1,\geq ,x}\Psi\Big(\widetilde{\mathbb{P}}_{\geq \kappa^{-1}}\mathcal{N}
-\langle \mathcal{N},Y\rangle\widetilde{\mathbb{P}}_{\geq \kappa^{-1}}Y\Big) \right\vert &\lesssim 
 \delta a_1^2(t)+\delta\int w^2\sech^{1/2}(x)+O(\delta^4)+O(\delta^3/\kappa^{1/2}).
\end{align*}

Furthermore, thanks to the specific choice of $\Psi(x) = \tanh(x)$ we can match (in each of the above terms) 
the additional decay required in the first virial estimate for $\varepsilon_{1,\geq}$;  
in particular, we can bound
\begin{align*}
\dfrac{1}{2\mu}\int \varepsilon_{1,\geq}^2\Psi' 
 - \dfrac{3}{2\mu}\int \varepsilon_{1,\geq}^2\partial_x\big(Q^2\Psi\big) \lesssim  \dfrac{1}{\mu}\int w^2\sech^{1/2}(x).
\end{align*}
Then, using the above estimate, and combining \eqref{hf_og_virial} and \eqref{dp_dt},
produces the following inequality:
\begin{align}\label{final_og_virial_hf}
\begin{split}
 \dfrac{d}{dt}\mathcal{I}_\geq + \dfrac{d}{dt}\mathcal{P}_\geq
 & \leq 
 - \dfrac{1}{2\mu}\int w_t^2\sech^{1/2}(x)-\dfrac{1}{\mu}\int w_x^2
 + \dfrac{C}{\mu} \int w^2\sech^{1/2}(x) + O(\delta) \frac{1}{\mu} a_1^2(t)
\\ 
& \quad + \dfrac{1}{\mu}\mathsf{P}_{1,\geq}(t) + \dfrac{1}{\mu}\mathsf{P}_{2,\geq}(t)
 - \dfrac{\mu'}{\mu}\mathcal{I}_\geq
 - \dfrac{\mu'}{\mu}\mathcal{P}_\geq + \dfrac{1}{\mu}O(\delta^4)
 + \dfrac{1}{\mu}O(\delta^2/\kappa).
\end{split}
\end{align}
Notice that we have introduced the integral of $w_t = \partial_t (\varepsilon_{1,\geq})$ 
as the first term on the right-hand sides of \eqref{final_og_virial_hf} 
in place of the first integral 
on the right-hand side of \eqref{dp_dt} using the following:
$$w_t = \partial_t (\varepsilon_{1,\geq} \sqrt{\Phi'_A}) = \varepsilon_{2,\geq} \sqrt{\Phi'_A} 
+ \dot{\widetilde{\mathbb{P}}}_{\geq \kappa^{-1}} \varepsilon_1 \sqrt{\Phi'_A}$$
which gives
\begin{align}\label{propORaux}
 -\dfrac{1}{2\mu}\int \varepsilon_{2,\geq}^2\Psi'
  \leq - \dfrac{1}{2\mu}\int w_t^2
  +  \dfrac{1}{2\mu} \int  \big( \dot{\widetilde{\mathbb{P}}}_{\geq \kappa^{-1}} \varepsilon_1\big)^2 \Psi'
  =  - \dfrac{1}{2\mu}\int w_t^2 + \frac{1}{\mu} O(\delta^2 / \kappa^2).
\end{align}
For the last estimate we have used the fact that  $\dot{\widetilde{\mathbb{P}}}_{\geq \kappa^{-1}}$
is the `projection' with symbol $\partial_t \eta(\xi/\kappa) = -(\xi/\kappa) \eta'(\xi/\kappa)  \kappa^{-1} \kappa'$
and that $\kappa^{-1} \kappa' = o(\kappa^{-2})$.

The identity \eqref{final_2_og_virial_hf} then follows by setting
\begin{align} \label{def_r1_og_virial_fh}
\mathfrak{R}_1(t) & :=   \dfrac{1}{\mu}\mathsf{P}_{1,\geq}(t)  +\dfrac{1}{\mu}\mathsf{P}_{2,\geq}(t) 
	- \dfrac{\mu'}{\mu}\mathcal{I}_\geq - \dfrac{\mu'}{\mu}
	\mathcal{P}_\geq + \dfrac{1}{\mu}O(\delta^4)+\dfrac{1}{\mu}O(\delta^2/\kappa),
\end{align}
where $\mathsf{P}_{1,\geq}$ is defined by \eqref{P1}, and $\mathsf{P}_{2,\geq}$ is in \eqref{P2def}.
The property \eqref{propORR1} follows from our choice of $\mu = \langle t \rangle$ and $T_{\max}$ ,
from the fact that $| \mathcal{I}_\geq| + |\mathcal{P}_\geq| \lesssim \mu^{-1} \delta^2$
and from the estimates  for the terms where the projections are differentiated
\begin{align}\label{P12est}
|\mathsf{P}_{i,\geq}(t)| \lesssim \delta^2 \kappa^{-1} \kappa' \lesssim \delta^2 \mu^{-1}, \qquad i=1,2,
\end{align}
which can be obtained arguing as we did above for \eqref{propORaux}.

This concludes the proof of the proposition.
\end{proof}


\smallskip
Proposition \ref{propOR} concludes our virial analysis for $(\varepsilon_1,\varepsilon_2)$.
The rest of the paper is devoted to estimating the third term on the right hand side of \eqref{final_2_og_virial_hf},
that is, the local norm of $w^2$.
This is done using a transformed problem.


\bigskip
\section{Transformed problem and coercivity}\label{SecTransformed}

\subsection{The transformed equations}
We consider the cubic Klein-Gordon linear operator around $Q$, and the corresponding flat operator
\begin{align*}
\mathcal{L}=-\partial_x^2+1- 3Q^2 \qquad \hbox{ and } \qquad  \mathcal{L}_0:=-\partial_x^2+1.
\end{align*}
Next, define the operators
\begin{align}\label{op_SU}
\begin{split}
& S:= Q\cdot \partial_x\cdot Q^{-1} = \partial_x + \tanh(x),
\\
& U:=Y\cdot\partial_x \cdot Y^{-1} = \partial_x + 2\tanh(x), 
\end{split}
\end{align}
which conjugate $\mathcal{L}$ to the flat operator:
\begin{align}\label{op_SUconj}
SU\mathcal{L}=\mathcal{L}_0SU, \qquad S U = \partial_x^2+3 \tanh(x) \partial_x+2.
\end{align}
In particular, if we consider $\vec{\boldsymbol{u}}=(u_1,u_2)$ solution to \[
\begin{cases}u_{1,t}=u_2,
\\ u_{2,t}=-\mathcal{L}u_1,
\end{cases}
\]
and define $\vec{\boldsymbol{ v}}=(v_1,v_2)=(SUu_1,SUu_2)$, then $\vec{\boldsymbol{ v}}$ satisfies the system \[
\begin{cases}
v_{1,t}=v_2,
\\ v_{2,t}=-\mathcal{L}_0v_1.
\end{cases}
\]
An important spectral observation regarding the operator $SU$ is that 
\begin{align}\label{SUY}
SUY=0, \quad SUQ'=0, \quad \hbox{and}\quad SUR=2.
\end{align}
One issue with defining the variable $\vec{\boldsymbol{v}}$ as above is the loss of regularity due to the composition with $SU$. One way to go around this problem is to define the ``transformed'' variables as follows: 
 for $\gamma\in(0,1)$ small, we define the \textit{transformed} variables as 
\begin{align}\label{defdual}
\varphi_1 := \big(1-\gamma\partial_x^2 \big)^{-1}SU  \varepsilon_1  
  \qquad\hbox{ and } \qquad \varphi_2 := \big(1-\gamma\partial_x^2 \big)^{-1}SU   \varepsilon_2 = \partial_t \varphi_1  .
\end{align} 

Additionally, we introduce the frequency-localized transformed variables 
\[
\varphi_{i,\geq}:=\mathbb{P}_{\geq \kappa^{-1}}\varphi_i 
  \qquad \hbox{and} \qquad \varphi_{i,<} := \mathbb{P}_{<\kappa^{-1}}\varphi_i,\qquad i=1,2.
\]
Using the conjugation property \eqref{op_SUconj} we can commute frequency projections as follows:
\begin{align}\label{projcomm}
SU \widetilde{\mathbb{P}}_k = \mathbb{P}_k SU, \qquad k > 0.
\end{align}
Therefore,
\begin{align}\label{phiideps}
\begin{split}
\varphi_{i,\geq}&=\big(1-\gamma\partial_x^2 \big)^{-1}SU\,\widetilde{\mathbb{P}}_{\geq\kappa^{-1}} \varepsilon_i,
\\ 
\varphi_{i,<}&=\big(1-\gamma\partial_x^2 \big)^{-1}SU\,\widetilde{\mathbb{P}}_{<\kappa^{-1}} \varepsilon_i,
  \qquad i=1,2.
\end{split}
\end{align}

Using the system \eqref{system_epsilon} for $(\varepsilon_1,\varepsilon_2)$
and the identity \eqref{phiideps} 
we then derive equations for $(\varphi_{1,\geq},\varphi_{2,\geq})$: 

\begin{lem}\label{lemtransf}
Given the above definitions, we have that $(\varphi_{1,\geq},\varphi_{2,\geq})\in C(\R,H^1\times L^2)$
satisfy the following system of equations
\begin{align}\label{system_pde_dual_variables}
\begin{cases}
\varphi_{1,\geq,t}=\varphi_{2,\geq}+\dot{\mathbb{P}}_{\geq\kappa^{-1}}\varphi_1,
\\ \varphi_{2,\geq,t}=-\mathcal{L}_0\varphi_{1,\geq}
  + \big(1-\gamma\partial_x^2\big)^{-1} \mathbb{P}_{\geq \kappa^{-1}}SU \mathcal{N} + 
  \dot{\mathbb{P}}_{\geq\kappa^{-1}}\varphi_2,
\end{cases}
\end{align}
where $\mathcal{N}$ is defined in \eqref{definition_N_nonlinearity_epsilon}.
\end{lem}

\begin{proof}
The proof is a direct computation. 
Recall that $\varepsilon_2$ satisfies 
$
\varepsilon_{2,t}=-\mathcal{L}\varepsilon_1+ \mathcal{N} -\langle \mathcal{N},Y\rangle Y.
$
Differentiating the definition of $\varphi_{2,\geq}$ 
we find that
\begin{align}
\nonumber
\varphi_{2,\geq,t} & = \big(1-\gamma\partial_x^2\big)^{-1}\mathbb{P}_{\geq\kappa^{-1}} SU\varepsilon_{2,t}
+ \big(1-\gamma\partial_x^2\big)^{-1}\dot{\mathbb{P}}_{\geq\kappa^{-1}} SU\varepsilon_{2}
\\
\label{lemtransf1}
& = - \big(1-\gamma\partial_x^2\big)^{-1} \mathbb{P}_{\geq\kappa^{-1}} SU \mathcal{L}\varepsilon_1
+ \big(1-\gamma\partial_x^2\big)^{-1}\mathbb{P}_{\geq\kappa^{-1}} SU \mathcal{N}
\\
\label{lemtransf2}
& - \big(1-\gamma\partial_x^2\big)^{-1}\mathbb{P}_{\geq\kappa^{-1}} SU \langle \mathcal{N},Y\rangle Y
+ \big(1-\gamma\partial_x^2\big)^{-1}\dot{\mathbb{P}}_{\geq\kappa^{-1}} SU\varepsilon_{2}
\end{align}
Recalling \eqref{op_SUconj}, and commuting $\mathbb{P}_{\geq\kappa^{-1}}$
with $(1-\gamma \partial_x^2)^{-1}$ and $\mathcal{L}_0$, 
we can express the linear term above 
as 
\begin{align*}
-\big(1-\gamma \partial_x^2\big)^{-1} \mathbb{P}_{\geq\kappa^{-1}} SU\mathcal{L}\varepsilon_1
  & = -\mathbb{P}_{\geq\kappa^{-1}} \mathcal{L}_0\big( (1-\gamma\partial_x^2)^{-1}SU \varepsilon_1 \big) 
  = - \mathbb{P}_{\geq\kappa^{-1}} \mathcal{L}_0\varphi_{1} = - \mathcal{L}_0\varphi_{1,\geq}.
\end{align*}
The second term in \eqref{lemtransf1} is already accounted for in \eqref{system_pde_dual_variables}.
The first term in \eqref{lemtransf2} vanishes in view of \eqref{SUY}
and the last one gives the term $\dot{\mathbb{P}}_{\geq\kappa^{-1}}\varphi_2$, 
since $(1-\gamma \partial_x^2)^{-1}$ commutes with the (differentiated) projection. 
\end{proof}



\subsection{Coercivity}
We now want to link the weighted norm of $\varepsilon_{1,\geq}$ with that of $\varphi_{1,\geq}$. 
More specifically, we seek to show that the ``bad term'' in \eqref{hf_og_virial} 
satisfies an estimate of the form 
\[
\int w^2\sech^{1/2}(x)\lesssim \int \big(\varphi_{1,\geq,x}^2+ \varphi_{1,\geq}^2\big)\sech^{1/4^+}(x)
 + O(\delta^2/\kappa).
\]
This inequality shall allow us to combine the virial estimate derived in the previous section with the ones that we shall derive for the transformed problem. 

\begin{prop}\label{prop_coer_varepsilon}
Suppose that $(\varepsilon_{1},\varepsilon_{2})$ is a solution to \eqref{system_epsilon} 
satisfying the stable manifold condition \eqref{stable_manifold_hyp_thm} 
and that $\varepsilon_1$ is even. Let $\varphi_{1,\geq}$ be defined as in \eqref{phiideps}
with some $\gamma>0$ small enough,
and assume that $\langle \varepsilon_1,Y\rangle=0$.
Then, the following inequality holds 
\begin{align}\label{coer_chi_varep}
\int \varepsilon_{1,\geq}^2\sech^{1/2}(x)&\lesssim \int \big(\varphi_{1,\geq,x}^2+\varphi_{1,\geq}^2\big)\sech^{1/4^+}(x)+O(\delta^2/\kappa).
\end{align}
\end{prop}

\begin{proof}
This proof follow similar lines to the one given in \cite{KowMarMun}, however, 
due to the appearance of the frequency projector, some adaptations are required. 
For the sake of completeness we provide the 
details. 
From \eqref{phiideps}, and using the explicit form for the $S$ and $U$ operators in \eqref{op_SU}, we have
\begin{align*}
\varphi_{1,\geq}-\gamma\partial_x^2\varphi_{1,\geq}=SU\widetilde{\mathbb{P}}_\geq\varepsilon_{1},
  = Q\partial_x\left(\dfrac{Y}{Q}\partial_x\left(\dfrac{\widetilde{\mathbb{P}}_\geq\varepsilon_{1}}{Y}\right)\right).
\end{align*}
Now we seek to solve for $\varepsilon_{1,\geq}$ in terms of $\varphi_{1,\geq}$. 
Dividing both sides of the above equality by $Q$ and then integrating from $0$ to $x$, with $x>0$,
we infer that 
\begin{align}\label{coer1}
\dfrac{Y}{Q}\partial_x\left(\dfrac{\widetilde{\mathbb{P}}_\geq\varepsilon_{1}}{Y}\right)
= C_1+\int_0^x\dfrac{1}{Q}\Big(\varphi_{1,\geq}-\gamma\partial_x^2\varphi_{1,\geq}\Big),
\end{align}
for some integration constant $C_1\in\R$. 
However, due to the even parity assumptions on $\varepsilon_1$, and the fact that the distorted projection
preserves this parity, see (ii) of Remark \ref{remdFTadd},
we see that $C_1=0$ by evaluating the identity \eqref{coer1} at $x=0$. 

Proceeding from \eqref{coer1}, 
multiplying by $Q/Y$, then integrating from $0$ to $x$, with $x>0$, 
and re-arranging terms, we get
\begin{align}\label{chi_varepsilon_itov_coer}
\varepsilon_{1,\geq} = C_0Y + \mathrm{I}, \qquad \mathrm{I}(x):= Y\int_0^x\dfrac{Q}{Y}\int_0^y\dfrac{1}{Q}
  \Big(\varphi_{1,\geq}- \gamma\partial_x^2\varphi_{1,\geq}\Big)
\end{align}
for some integration constant $C_0\in\R$. 
Then, the problem reduces to analyse $C_0$ and $\mathrm{I}$. 
We begin by studying $\mathrm{I}$. 
By Cauchy-Schwarz,
the first summand appearing in the integral 
can be bounded as follows:
\begin{align}\label{coercaux}
\begin{split}
\mathrm{I}_1(x) := \left\vert Y\int_0^x\dfrac{Q}{Y}\int_0^y\dfrac{\varphi_{1,\geq}}{Q}\right\vert 
  &\lesssim \Vert \varphi_{1,\geq}\sech^{\theta^-/2}(x)\Vert_{L^2}\cdot Y\int_0^x\dfrac{Q}{Y}\left(\int_0^y \dfrac{1}{Q^2}\cosh^{\theta^-}(r)dr\right)^{1/2}
\\ 
& \lesssim \Vert \varphi_{1,\geq}\sech^{\theta^-/2}(x)\Vert_{L^2}\cdot\cosh^{\theta^-/2}(x),
\end{split}
\end{align}
for any $\theta>0$. 
As a direct consequence of the above estimate we deduce that 
\begin{align}\label{coer_proof_iii1}
\int_0^\infty \mathrm{I}_1^2(x) \sech^\theta(x)\lesssim \int \varphi_{1,\geq}^2\sech^{\theta^-}(x).
\end{align}
Notice that an identical estimate holds for $x<0$. 
In order to deal with the other term in $I$ 
we begin by integrating by parts, recalling that $\varphi_{1,\geq}$ is even, and 
get 
\begin{align*}
\mathrm{I}_2(x) := \left\vert Y\int_0^x\dfrac{Q}{Y}
  \int_0^y\dfrac{\gamma \partial_x^2\varphi_{1,\geq}}{Q}\right\vert
=\left\vert Y\int_0^x\left(\dfrac{\gamma\partial_x\varphi_{1,\geq}}{Y}
 + \dfrac{Q}{Y}\int_0^y\dfrac{\gamma Q'\partial_x \varphi_{1,\geq}}{Q^2}\right)\right\vert.
\end{align*}
Thus, proceeding in exactly the same way as in \eqref{coercaux} 
we conclude that for any $\theta>0$ we have 
\begin{align}\label{coer_proof_iii2}
\int_0^\infty \mathrm{I}_2^2(x) \sech^\theta(x)\lesssim \int \varphi_{1,\geq,x}^2\sech^{\theta^-}(x).
\end{align}
Again, the same estimate holds for $x<0$.

To conclude the proof we need to estimate $C_0$.
Projecting \eqref{chi_varepsilon_itov_coer} on the $Y$ direction 
we obtain
\begin{align}\label{Bsquare_estimate}
C_0^2\lesssim\left\langle \varepsilon_{1,\geq},Y\right\rangle ^2+\big\langle \mathrm{I},Y\big\rangle^2.
\end{align}
It suffices to estimate the first term on the right-hand side above, 
since the second is already estimated using \eqref{coer_proof_iii1} and \eqref{coer_proof_iii2} with $\theta=2$. 
Due to the orthogonality property of $\varepsilon_1$, we have
\[
\langle \varepsilon_{1,\geq},Y\rangle=\langle \varepsilon_1,Y\rangle-\langle \varepsilon_{1,<},Y\rangle
 = -\langle \varepsilon_{1,<},Y\rangle,
\]
hence, 
\[
\big\vert \langle \varepsilon_{1,\geq},Y\rangle\big\vert^2 = 
\big\vert \langle \varepsilon_{1,<},Y\rangle\big\vert^2 \lesssim = {\| \varepsilon_{1,<} \|}_{L^\infty}^2
 = O(\delta^2/\kappa),
\]
having used distorted Bernstein.
This concludes the estimate for $C_0^2$ and the proof of \eqref{coer_chi_varep}.
\end{proof}

\smallskip
Observe that the Proposition \ref{prop_coer_varepsilon} tells us that
if we are able to control the local norms of $\varphi_{1,\geq}$, 
then we can control those of $\varepsilon_{1,\geq}$.
Sections \ref{SecVirialDual} and \ref{sec_Analysis_H} are dedicated to establishing control of the local norms of 
the transformed (and frequency localized) variable $\varphi_{1,\geq}$.


\bigskip
\section{Frequency localized virials for the transformed problem}\label{SecVirialDual}
For $B>1$ a large parameter to be chosen, we define $\Phi_B(x)$ according to \eqref{def_zeta_A}-\eqref{def_phi_A},
and also define
\begin{align}\label{defpsiB}
\uppsi_{B}(x):= \int_0^x \zeta_B^2(y)\sech^{1/4}(y)dy.
\end{align}
The role of $B$ is that of an intermediate scale parameter, 
\[
1\ll B\ll A,
\]
which shall allow us to combine the original virial estimate with those of the transformed problem. 
Our goal is to prove an estimate 
of the form 
\[
\int_0^{T_{\mathrm{max}}} \dfrac{1}{\mu(t)}\int_\R \big(\varphi_{2,\geq}^2\sech^{1/4}(x)+\varphi_{1,\geq,x}^2\sech^2\big(\tfrac{x}{B}\big)+\varphi_{1,\geq}^2\sech^{1/4}(x)\big)dx  dt=O(\delta^2).
\]
In this section we consider the following 
frequency-localized modified momentum functionals 
\begin{align}
\label{defJ}
\mathcal{J}_\geq &:= \dfrac{1}{\mu(t)}\int_\R\left( \Phi_{B}(x)\varphi_{1,\geq,x}(t,x) 
  + \tfrac{1}{2}\Phi'_{B}(x) \varphi_{1,\geq}(t,x)\right)\varphi_{2,\geq}(t,x)dx,
\\
\label{defK}
\mathcal{K}_\geq &:= \dfrac{1}{\mu(t)}\int_\R \uppsi_{B}(x)\varphi_{1,\geq,x}(t,x) \varphi_{2,\geq}(t,x)dx.
%
\end{align}
One more (set of) ``singular'' frequency localized functional(s) will be defined in Section \ref{sec_Analysis_H}.

\smallskip
\subsection{Virial argument for $\mathcal{J}_\geq$}
The next Lemma gives us the basic virial identity for $\mathcal{J}_\geq$. 
\begin{lem}\label{lemdtJ}
Let $(\varepsilon_1,\varepsilon_2)\in C(\R,H^1\times L^2)$ be any solution to system 
\eqref{system_epsilon} and $(\varphi_{1,\geq},\varphi_{2,\geq})$ defined as in \eqref{defdual}
satisfying \eqref{system_pde_dual_variables}. 
Then, for all times $t\in\R$ the following identity holds: 
\begin{align}\label{dt_J_dual}
\begin{split}
\dfrac{d}{dt}\mathcal{J}_\geq&=-\dfrac{1}{\mu}\int\varphi_{1,\geq,x}^2\Phi_B'
 + \dfrac{1}{4 \mu}\int \varphi_{1,\geq}^2\Phi_B'''  + \dfrac{1}{\mu}\mathsf{P}_{3,\geq}
 - \dfrac{\mu'}{\mu}\mathcal{J}_\geq 
\\ 
& \qquad + \dfrac{1}{\mu}\int \left(\varphi_{1,\geq ,x}\Phi_B
 + \tfrac{1}{2}\varphi_{1,\geq }\Phi_B'\right)
 \big(1-\gamma\partial_x^2\big)^{-1}SU\,\widetilde{\mathbb{P}}_{\geq \kappa^{-1}}\mathcal{N},
\end{split}
\end{align}
where 
\[
\mathsf{P}_{3,\geq}:=\int \left(\Phi_B\partial_x\dot{\mathbb{P}}_{\geq\kappa^{-1}}\varphi_1
  + \tfrac{1}{2}\Phi_B'\dot{\mathbb{P}}_{\geq\kappa^{-1}}\varphi_1\right)\varphi_{2,\geq}
  +\int \left(\Phi_B\varphi_{1,\geq,x}+\tfrac{1}{2}\Phi_B'\varphi_{1,\geq}\right) 
  \dot{\mathbb{P}}_{\geq\kappa^{-1}}\varphi_2 
  ,
\]
\end{lem}


\begin{proof}
Directly differentiating the definition of the functional $\mathcal{J}_\geq$, 
using system \eqref{system_pde_dual_variables} and performing
several integration by parts we obtain 
\begin{align*}
\dfrac{d}{dt}\mathcal{J}_\geq&=\dfrac{1}{\mu}\int\left(\varphi_{2,\geq,x}\Phi_B
 + \tfrac{1}{2}\varphi_{2,\geq}\Phi_B'\right)\varphi_{2,\geq}  
 + \dfrac{1}{\mu}\int \left(\varphi_{1,\geq,x}\Phi_B
 + \tfrac{1}{2}\varphi_{1,\geq}\Phi_B'\right)\varphi_{2,\geq,t}
\\
& \ \quad  +\dfrac{1}{\mu}\int \left(\Phi_B\partial_x\dot{\mathbb{P}}_{\geq\kappa^{-1}}\varphi_1
  +\tfrac{1}{2}\Phi_B'\dot{\mathbb{P}}_{\geq\kappa^{-1}}\varphi_1\right)\varphi_{2,\geq}
  -\dfrac{\mu'}{\mu}\mathcal{J}_\geq
%
%
\\ & = -\dfrac{1}{\mu}\int \left(\varphi_{1,\geq ,x} \Phi_B
  +\tfrac{1}{2}\varphi_{1,\geq }\Phi_B'\right)\mathcal{L}_0\varphi_{1,\geq} 
  + \dfrac{1}{\mu}\mathsf{P}_{3,\geq}-\dfrac{\mu'}{\mu}\mathcal{J}_\geq
\\ & \ \quad  + \dfrac{1}{\mu}\int \left(\varphi_{1,\geq ,x}\Phi_B+\tfrac{1}{2}\varphi_{1,\geq }\Phi_B'\right)
  \big(1-\gamma\partial_x^2\big)^{-1}SU\widetilde{\mathbb{P}}_{\geq \kappa^{-1}} \mathcal{N} ,
\end{align*}
where we have used the fact that \[
\int\left(\varphi_{2,\geq,x}\Phi_B+\tfrac{1}{2}\varphi_{2,\geq}\Phi_B'\right)\varphi_{2,\geq} =0.
\] 
Expanding the linear operator $\mathcal{L}_0$ and then integrating by parts we conclude the proof.
\end{proof}

Similarly as before (see \eqref{w_def} and the calculations after that) 
in order the get rid of the $\Phi_B'''$ term, we 
introduce the localized (in space) version of $\varphi_{1,\geq}$,
\begin{align}\label{def_z}
z(t,x):=\varphi_{1,\geq}(t,x)\sqrt{\Phi_B'(x)}.
\end{align}
Differentiating the above definition we obtain 
\begin{align}\label{derivative_zx}
z_{x}=\varphi_{1,\geq,x}\sqrt{\Phi'_B}+\varphi_{1,\geq}\cdot\tfrac{\Phi_B''}{2\sqrt{\Phi_B'}},
\end{align}
and squaring this identity we infer that 
\begin{align*}
-\int\varphi_{1,\geq,x}^2\Phi_B' &=- \int z_{x}^2-\dfrac{1}{2}\int \varphi_{1,\geq}^2\Phi_B'''        
  + \dfrac{1}{4}\int \varphi_{1,\geq}^2\tfrac{(\Phi_B'')^2}{\Phi_B'}.
\end{align*}

Next, 
recall that in Section \ref{original_virial_section} we showed that 
\begin{align*}
\Phi_B''&=\tfrac{2}{B}\big(\chi'(x)\vert x\vert -(1-\chi(x))\sgn(x)\big)\zeta_B^2(x),
\\ \Phi_B'''&=\tfrac{4}{B^2}\big(\chi'(x)\vert x\vert -(1-\chi(x))\sgn(x)\big)^2\zeta^2_B(x)+\tfrac{2}{B}\big(\chi''(x)\vert x\vert+2\chi'(x)\sgn(x)\big)\zeta_B^2(x),
\end{align*}
and hence, for $B>1$ sufficiently large, we have that 
\[
\left\vert -\dfrac{\Phi_B'''}{2\Phi_B'}+\dfrac{(\Phi_B'')^2}{4(\Phi_B')^2}+\dfrac{\Phi_B'''}{4\Phi_B'}\right\vert=\dfrac{1}{2B}\left\vert \chi''(x)\vert x\vert+2\chi'(x)\sgn(x)\right\vert\lesssim \dfrac{1}{B}\sech^{1/4}(x)\Phi_B'(x).
\]
Thus, gathering the first and second term in \eqref{dt_J_dual} 
along with the above identities, 
for some absolute constant $C>0$ we can bound both leading order terms as follows:
\begin{align}\label{varphi_into_z_first_ineq}
-\int\varphi_{1,\geq,x}^2\Phi_B'+\dfrac{1}{4}\int \varphi_{1,\geq}^2\Phi_B'''&\leq 
  -\int z_x^2+\dfrac{C}{B}\int z^2\sech^{1/4}(x).
\end{align}

The rest of this subsection is devoted to estimating 
all the other terms appearing
on the right-hand side of the virial identity \eqref{dt_J_dual}.
We begin by handling the terms containing the nonlinearity $\mathcal{N}$.

\begin{lem}\label{lem_estimate_dual_J_nonlinearity}
Let  $(\varepsilon_1,\varepsilon_2)$ be any solution to system \eqref{system_epsilon} 
satisfying \eqref{stable_manifold_hyp_thm}, let $(\varphi_{1,\geq},\varphi_{2,\geq})$ 
be defined as in \eqref{defdual} and $\epsilon\in(0,1)$ small. 
Then, for all times $t\in\R$ the following inequality holds
\begin{align}\label{estimate_nonlinearterms_dual_J}
\begin{split}
& \left\vert \int \left(\varphi_{1,\geq ,x}\Phi_B
 + \tfrac{1}{2}\varphi_{1,\geq }\Phi_B'\right)\big(1-\gamma\partial_x^2\big)^{-1}\mathbb{P}_{\geq \kappa^{-1}}SU \mathcal{N} 
 \right \vert   
\\
& \qquad \qquad \qquad \lesssim  \delta a_1^2 
  + \epsilon \int w^2\sech^{1/2} (x) 
  + 
  O(\delta^4)+O(\delta^2/\kappa).
\end{split}
\end{align}
\end{lem}

\begin{proof}
The proof is similar to that of Lemma \ref{lemma_N_eps_geq_original_virial} so we will skip some details. Recall the formula for $N$ in \eqref{Nepsilon'} and that $Q=\sqrt{2}\sech(x)$ and $Y\approx \sech^2(x)$.
We split the estimates into several cases.

\smallskip
\begin{itemize}[leftmargin=*]

\item[$\cdot$] 
Control of $\varepsilon_{1}^3$:
Using Cauchy-Schwarz we have 
\begin{align*}
& \left\vert \int \left(\varphi_{1,\geq ,x}\Phi_B+\tfrac{1}{2}\varphi_{1,\geq }\Phi_B'\right)
  \big(1-\gamma\partial_x^2 \big)^{-1}SU\,\widetilde{\mathbb{P}}_{\geq\kappa^{-1}}\varepsilon_1^3 \right \vert
  \\
  & \lesssim \big( B {\| \partial_x \varphi_{1,>} \|}_{L^2} + {\| \varphi_{1,>} \|}_{L^2} \big) 
  {\big\| \big(1-\gamma\partial_x^2 \big)^{-1}SU\,\widetilde{\mathbb{P}}_{\geq\kappa^{-1}}\varepsilon_1^3 \big\|}_{L^2}
  \\
  & \lesssim \delta \cdot {\| \varepsilon_1^3 \big\|}_{L^2} \lesssim \delta^4,
\end{align*}
where the implicit constant may depend on $B$ and $\gamma$,
and having used that $\big(1-\gamma\partial_x^2 \big)^{-1}SU$ is bounded on $L^2$.

\smallskip
\item[$\cdot$] Control of $\varepsilon_1^2\big(Q+a_1Y\big)$: 
In this case it is enough to use the decay of $Q$. 
More precisely, splitting $\varepsilon_1=\varepsilon_{1,\geq}+\varepsilon_{1,<}$, 
we have that 
\begin{align}\label{loce}
\begin{split}
\Vert \big(1-\gamma\partial_x^2 \big)^{-1}SU\,\widetilde{\mathbb{P}}_{\geq\kappa^{-1}}
  \varepsilon_1^2Q\Vert_{L^2}^2 
  & \lesssim \Vert \varepsilon_1^2Q\Vert_{L^2}^2
  \\
  & \lesssim {\| \varepsilon_1 \|}_{L^\infty}^2 \int \big(w^2\sech^{1/2}(x) + \varepsilon_{1,<}^2\sech^2(x)\big)
  \\
  & \lesssim \delta^2\int w^2\sech^{1/2}(x)+O(\delta^4/\kappa),
\end{split}
\end{align}
having used once again (distorted) Bernstein for the low frequency contribution.
Thus, using the above estimate
and Cauchy-Schwarz we infer that 
\begin{align*}
&\left\vert \int \left(\varphi_{1,\geq ,x}\Phi_B+\tfrac{1}{2}\varphi_{1,\geq }\Phi_B'\right)
  \big(1-\gamma\partial_x^2 \big)^{-1}SU\,\widetilde{\mathbb{P}}_{\geq\kappa^{-1}}(\varepsilon_1^2Q) \right \vert 
\\
& \lesssim \big( B {\| \partial_x \varphi_{1,>} \|}_{L^2} + {\| \varphi_{1,>} \|}_{L^2} \big)
   \Big( \delta^2\int w^2\sech^{1/2}(x)+O(\delta^4/\kappa) \Big)^{1/2}
\\  
& \lesssim \epsilon^{-1}\delta^4+\epsilon\int w^2\sech^{1/2}(x) 
  + O(\delta^2/\kappa).
\end{align*}

\smallskip
\item[$\cdot$] 
Control of $a_1\varepsilon_1\big(a_1Y^2+2QY\big)$: 
Similarly, splitting $\varepsilon_1=\varepsilon_{1,\geq}+\varepsilon_{1,<}$ and using the decay of $Q$ we get  
\[
\Vert a_1\varepsilon_1(a_1Y^2+2QY)\Vert_{L^2}^2\lesssim a_1^2 \int w^2\sech^{1/2}(x)+a_1^2O(\delta^2/\kappa),
\]
and hence this terms are handled exactly as the previous one.

\smallskip
\item[$\cdot$] 
Control of $a_1^3Y^3 + 3a_1^2QY^2$: 
In this case 
we directly estimate
\begin{align*}
&\left\vert \int \left(\varphi_{1,\geq ,x}\Phi_B 
  + \tfrac{1}{2}\varphi_{1,\geq }\Phi_B'\right)\big(1-\gamma\partial_x^2 \big)^{-1}
  \mathbb{P}_{\geq\kappa^{-1}}SU(a_1^2Y^2(a_1Y+3Q)) \right \vert \lesssim \delta a_1^2.
\end{align*}

\end{itemize}

Gathering all the above estimate we conclude the proof of the lemma.
\end{proof}

\begin{lem}\label{lem_estimate_dual_J_Upsilon2}
Under the same assumptions of Lemma \ref{lem_estimate_dual_J_nonlinearity},
the following inequality holds for all times $t\in\R$:
\begin{align*}
\left\vert \int \left(\varphi_{1,\geq ,x}\Phi_B + \tfrac{1}{2}\varphi_{1,\geq }\Phi_B'\right)
  \dot{\mathbb{P}}_{\geq\kappa^{-1}}\varphi_2 
 \right\vert  
  & \lesssim 
  O(\delta^2/\kappa).
\end{align*}
\end{lem}

\begin{proof}
%
It suffices to observe that 
$\Vert \dot{\mathbb{P}}_{\geq \kappa^{-1}}\varphi_2\Vert_{L^2} \lesssim \delta \kappa^{-1} \kappa' \ll 
 \delta/\kappa^2$, so that we get
\[
\left\vert \int \big(\varphi_{1,\geq,x}\Phi_B + \tfrac12\varphi_{1,\geq}\Phi_B'\big)
  \dot{\mathbb{P}}_{\geq\kappa^{-1}}\varphi_2 
  \right\vert
  \lesssim {\| \varphi_{1,\geq} \|}_{H^1} \Vert \dot{\mathbb{P}}_{\geq \kappa^{-1}}\varphi_2\Vert_{L^2} 
  \lesssim \delta^2 /\kappa.
\]
\end{proof}

As a direct consequence of
the identities \eqref{dt_J_dual} and \eqref{varphi_into_z_first_ineq},
Lemmas \ref{lem_estimate_dual_J_nonlinearity} and \ref{lem_estimate_dual_J_Upsilon2}, 
with $\epsilon$ small enough 
(but independent of $\delta$, $\kappa$ or $T_\mathrm{max}$), 
and the estimate $|\mathsf{P}_{3,\geq}| \lesssim \delta^2/\kappa$ (see the similar estimate \eqref{P12est},
and the argument after \eqref{propORaux})
we obtain the following corollary:

\begin{cor}\label{cor_final_virial_est_J}
Let $B\gg1\gg\gamma>0$ and $(\varepsilon_1,\varepsilon_2)\in C(\R,H^1\times L^2)$
be any solution to system \eqref{system_epsilon} satisfying hypothesis \eqref{stable_manifold_hyp_thm}. 
Let $w$ and $z$ be defined as in \eqref{w_def} and \eqref{def_z} respectively.
Then, for some absolute constant $C>0$, for all times $t\in\R$ we have
\begin{align}
\dfrac{d}{dt}\mathcal{J}_\geq \leq -\dfrac{1}{\mu}\int z_x^2+\dfrac{C}{B\mu}\int z^2\sech^{1/4}(x)
  + \dfrac{C\epsilon}{\mu}\int w^2\sech^{1/2}(x)\nonumber
\\ 
  + \dfrac{ a_1^2}{\mu}O(\delta)
  + \dfrac{1}{\mu}O(\delta^4) + \dfrac{1}{\mu}O(\delta^2/\kappa). 
\end{align}
\end{cor}

\medskip
\subsection{Virial argument for $\mathcal{K}_\geq$}
In this subsection we seek to show an estimate similar to that of Corollary \ref{cor_final_virial_est_J}
but in this case for the time derivative of the functional $\mathcal{K}_\geq$ defined in \eqref{defK}
with \eqref{defpsiB}.
Following similar calculations as those done in \eqref{dp_dt} for $\mathcal{P}_\geq$ 
and in Lemma \ref{lemdtJ} for $\mathcal{J}_\geq$,
it is not difficult to see that in this case we obtain
\begin{align}\label{dtK0}
\begin{split}
\dfrac{d}{dt}\mathcal{K}_\geq&=-\dfrac{1}{2\mu}\int \varphi_{2,\geq}^2\uppsi_B' 
  -  \dfrac{1}{2 \mu}\int \varphi_{1,\geq,x}^2\uppsi_B'+ \dfrac{1}{2\mu }\int \varphi_{1,\geq}^2\uppsi_B'    
\\ 
& \quad +\dfrac{1}{\mu}\int \varphi_{1,\geq,x} \uppsi_B 
  \big(1-\gamma\partial_x^2\big)^{-1} \mathbb{P}_{\geq \kappa^{-1}} SU \mathcal{N} 
  + \dfrac{1}{\mu}\mathsf{P}_{4,\geq}-\dfrac{\mu'}{\mu}\mathcal{K}_\geq,
\end{split}
\end{align}
where 
\begin{align}\label{def_P4}
\mathsf{P}_{4,\geq} :=\int \Big(\varphi_{2,\geq}\partial_x\dot{\mathbb{P}}_{\geq\kappa^{-1}}\varphi_{1}
  + \varphi_{1,\geq,x}
  \dot{\mathbb{P}}_{\geq\kappa^{-1}}\varphi_2 \Big) 
  \uppsi_B.
\end{align}

In order to rewrite \eqref{dtK0} in terms of the localized (in space) variable 
for the transformed problem $z(t,x)$, see \eqref{def_z}, recall that (see  \eqref{derivative_zx}) 
\begin{align*}
-\varphi_{1,\geq,x}^2\Phi'_B=-z_{x}^2+ \tfrac12\partial_x\big(\varphi_{1,\geq}^2\big)\Phi_B''
 + \varphi_{1,\geq}^2\tfrac{(\Phi_B'')^2}{4\Phi_B'}.
\end{align*}
Thus, since $\uppsi_B' = \Phi_B' \sech^{1/4}(x)$, 
after integration by parts and some manipulations we get  
\begin{align*}
-\dfrac{1}{2\mu}\int\varphi_{1,\geq,x}^2\uppsi_B' &= - \dfrac{1}{2\mu}\int z_{x}^2\sech^{1/4}(x)
 + \dfrac{1}{16\mu}\int \varphi_{1,\geq}^2\Phi_B''\sech^{1/4}(x)\tanh(x)   
\\ 
& \qquad -\dfrac{1}{4\mu}\int \varphi_{1,\geq}^2\Phi_B'''\sech^{1/4}(x) 
 + \dfrac{1}{8\mu}\int \varphi_{1,\geq}^2\tfrac{(\Phi_B'')^2}{\Phi_B'}\sech^{1/4}(x).
\end{align*}
Then, using the trivial estimates $ \vert \Phi_B''\vert\lesssim \tfrac1B\Phi_B'$ 
and $\vert\Phi_B'''\vert\lesssim \tfrac{1}{B^2}\Phi_B'$, we obtain that 
\[
-\dfrac{1}{2\mu} \int\varphi_{1,\geq,x}^2\uppsi_B' \geq  
  - \dfrac{1}{2\mu}\int z_{x}^2 -  \dfrac{1}{B\mu}  \int z^2\sech^{1/4}(x).
\]
Therefore, for $B\gg1$, using the definition of $z$ in \eqref{def_z}, 
we can bound the leading order terms in the first line of \eqref{dtK0} as follows: 
\begin{align}\label{varphi_into_z_dualK_first_ineq}
\begin{split}
& -\dfrac{1}{2\mu}\int \varphi_{2,\geq}^2\uppsi_B'- \dfrac{1}{2\mu} \int\varphi_{1,\geq,x}^2\uppsi_B'
  +\dfrac{1}{2\mu}\int \varphi_{1,\geq}^2\uppsi_B'
\\ 
& \geq -\dfrac{1}{2\mu} \int z_t^2\sech^{1/4}(x)
 - \dfrac{1}{2\mu}\int z_x^2 + \dfrac{1}{4\mu} \int z^2\sech^{1/4}(x) + O(\delta^2/\kappa).
\end{split}
\end{align}
The $O(\delta^2/\kappa)$ contribution above comes from differentiating the 
cutoff $\mathbb{P}_{\geq \kappa^{-1}}$ when replacing the integral involving 
$\varphi_{2,\geq} = \mathbb{P}_{\geq \kappa^{-1}} (\partial_t \varphi_1)$ 
with that involving $z_t = \partial_t (\mathbb{P}_{\geq \kappa^{-1}} \varphi_1) \sqrt{\Phi'_B}$.

The analysis of the other terms in \eqref{dtK0} goes exactly as in the previous subsection.
In particular we have:

\begin{lem}\label{lem_dual_K_estimate_nonlinearity}
Let $(\varepsilon_1,\varepsilon_2)$ be any solution to 
system \eqref{system_epsilon} satisfying \eqref{stable_manifold_hyp_thm},
$(\varphi_{1,\geq},\varphi_{2,\geq})$ defined as in \eqref{defdual} and $\epsilon\in(0,1)$ small. 
Then, for all times $t\in\R$,
\begin{align}\label{lemdualKconc1}
\begin{split}
&\left\vert \int \varphi_{1,\geq ,x}\uppsi_B\big(1-\gamma\partial_x^2\big)^{-1}\mathbb{P}_{\geq \kappa^{-1}}SU \mathcal{N} 
  \right \vert   
  \\
  & \lesssim \delta a_1^2 + \epsilon \int w^2\sech^{1/2}(x) + 
  O(\delta^4) + O(\delta^2/\kappa) 
\end{split}
\end{align}
and
\begin{align}\label{lemdualKconc2}
\big| \mathsf{P}_{4,\geq} \big| & \lesssim  
 \delta^2/\kappa.  
\end{align}
\end{lem}

\begin{proof}
The proof of \eqref{lemdualKconc1} is very similar to the one of
Lemma \ref{lem_estimate_dual_J_nonlinearity}, so we omit it.
The proof of \eqref{lemdualKconc2} also follows from arguments already given before, 
see for example \eqref{P12est} and \eqref{propORaux}.
\end{proof}


\smallskip
As a consequence of this last lemma, the virial identity \eqref{dtK0},
and the inequality \eqref{varphi_into_z_dualK_first_ineq},
we obtain the following:

\begin{cor}\label{cor_final_virial_est_K}
Let $B\gg1$ and $(\varepsilon_1,\varepsilon_2)$ be any solution to system \eqref{system_epsilon} 
satisfying \eqref{stable_manifold_hyp_thm}, 
and $(\varphi_{1,\geq},\varphi_{2,\geq})$ be defined as in \eqref{defdual}. 
Then, for some absolute constant $C>0$, for all times $t\in\R$ the following holds:
\begin{align*}
-\dfrac{d}{dt}\mathcal{K}_\geq &\leq   \dfrac{1}{2\mu}\int z_t^2\sech^{1/4}(x)
 + \dfrac{1}{2\mu}\int z_x^2 -\dfrac{1}{4\mu}\int z^2\sech^{1/4}(x)
\\ & + \dfrac{C\epsilon}{\mu}\int w^2\sech^{1/2}(x) +  O(\delta)\dfrac{1}{\mu}a_1^2
 + \dfrac{1}{\mu}O(\delta^4) + \dfrac{1}{\mu}O(\delta^2/\kappa),
\end{align*}
\end{cor}

\smallskip
Therefore, combining Corollaries \ref{cor_final_virial_est_J} and \ref{cor_final_virial_est_K} we conclude that 
\begin{align}\label{dj_m_dk_ineq}
\dfrac{d}{dt}\mathcal{J}_\geq-\dfrac{d}{dt}\mathcal{K}_\geq 
& \leq  
  \dfrac{1}{2\mu}\int z_t^2\sech^{1/4}(x) - \dfrac{1}{2\mu}\int z_x^2
  - \dfrac{1}{8\mu}\int z^2\sech^{1/4}(x) +  O(\delta) \dfrac{1}{\mu}a_1^2 +\mathfrak{R}_2(t),
\end{align}
where 
\begin{align}\label{def_r2_dual_virial_hf}
\mathfrak{R}_2(t)& := \dfrac{C \epsilon}{\mu}\int w^2\sech^{1/2}(x) 
  + \dfrac{1}{\mu}O(\delta^4) + \dfrac{1}{\mu}O(\delta^2/\kappa).
\end{align}
In particular $\mathfrak{R}_2(t)$ is the sum of terms that are integrable in time (up to $T_{\max}$)
and a leading order term with a small constant in front that can therefore be absorbed by the
other leading order terms via Proposition \ref{prop_coer_varepsilon} (recall \eqref{w_def}).

Therefore, 
we have essentially reduced the problem to controlling the local norm of $z_t$ or, equivalently, of $\varphi_{2,\geq}$
(see \eqref{def_z}).
We start our analysis of the local norm of $\varphi_{2,\geq}$ in the next subsection,
and then proceed to estimate it in Section \ref{sec_Analysis_H} 
by introducing some ``singular'' frequency-localized virial functionals
with time-dependent weights.

\medskip
\subsection{Dyadic frequency decomposition of the ``bad term''}
The first step to control the bad term involving $z_t = \partial_t (\varphi_{1,\geq} \sqrt{\Phi'_B})$ 
is the following simple but important observation.
Recall the definition from Subsection \ref{secnotation}, and let
\begin{align}
\varphi_{2,\geq 1} := \mathbb{P}_{\geq 1}\varphi_2, \qquad \varphi_{2, N^{-1}} := \mathbb{P}_{N^{-1}}\varphi_{2}, 
\qquad \varphi_{2,\geq} = \varphi_{2,\geq 1} + \sum_{\substack{N\in\mathbb{D},\\ 1<N\leq\kappa(t)}} \varphi_{2, N^{-1}};
\end{align}
then we can bound
\begin{align}\label{deco_varphi2_sech}
\begin{split}
\int \varphi_{2,\geq}^2\sech^{1/4}(x) & \lesssim \int \varphi_{2,\geq 1}^2\sech^{1/4}(x)
  + \log \kappa(t) \sum_{\substack{N\in\mathbb{D}, \\ 1<N\leq \kappa(t)}}
  \int \varphi_{2, N^{-1}}^2\sech^{1/4}(x)
\\ 
& =: \mathrm{I} + \sum_{\substack{N\in\mathbb{D},\\ 1< N\leq\kappa(t)}} \mathrm{II}_N,
\end{split}
\end{align}
%
Then, for given $N\in\mathbb{D}$ with $N\geq 1 $, 
we introduce a time-dependent weight function $\Psi_{N_*}$ with
\begin{align}\label{PsiN*}
\Psi_{ \lambda_N}'(x) := \sech^2\big(\tfrac{x}{ \lambda_N}\big) \qquad \hbox{and} \qquad  \lambda_N(t) := N\log(\kappa^9),
  \qquad \kappa(t) = \log^{1+\alpha}(e^{D} + t).
\end{align}

\begin{lem}\label{lemIIN}
With the definitions in \eqref{deco_varphi2_sech} above, the following estimates holds true:
\begin{align}\label{estI}
\mathrm{I}(t) \lesssim \int \varphi_{2,\geq 1}^2\sech^2\big(\tfrac{x}{\log \kappa^9}\big) dx.
\end{align}
and
\begin{align}\label{estIIN}
\mathrm{II}_N(t) 
  \lesssim 
  \dfrac{\log \kappa(t)}{N}\int \varphi_{2, N^{-1}}^2 \Psi_{\lambda_N}'dx  + O(\delta^2/\kappa^2) . 
\end{align}
\end{lem}


\begin{proof}
For \eqref{estI} we can directly enlarge the support of the $\sech^2(x)$, 
to produce the desired inequality.
To prove \eqref{estIIN} we first note that
\begin{align}\label{freq_N_varphi2_loc_mass}
\begin{split}
\mathrm{II}_N(t) \lesssim \log \kappa(t) \int_\R  \varphi_{2, N^{-1}}^2(t,x) \Psi'_{\lambda_N}(x)
  \big(\sech^{1/4}(x)(\Psi_{\lambda_N}')^{-1}(x)\big)dx
  \\
  \lesssim \log \kappa(t){\Big\| \varphi_{2, N^{-1}} \sech\big(\dfrac{\cdot}{\lambda_N}\big) \Big\|}_{L^\infty}^2.
\end{split}
\end{align}
It then suffices to show
\begin{align}\label{estIINpr}
\log\kappa{\Big\| \varphi_{2, N^{-1}} \sech\big(\dfrac{\cdot}{\lambda_N}\big) \Big\|}_{L^\infty}^2 
  \lesssim
  \dfrac{\log\kappa}{N}\int \varphi_{2, N^{-1}}^2 \Psi_{\lambda_N}' + O(\delta^2/\kappa^2).
\end{align}
For this we decompose
into low and high-frequency components (with respect to $N^{-1}$) as follows:
\begin{align}\label{estIINpr2}
\varphi_{2, N^{-1}} \sech\big(\dfrac{x}{\lambda_N}\big)
=\mathbb{P}_{\leq 4N^{-1}}\Big(\varphi_{2, N^{-1}} \sech\big(\dfrac{\cdot}{\lambda_N}\big)\Big)
 + \mathbb{P}_{\geq 8N^{-1}}\Big(\varphi_{2, N^{-1}} \sech\big(\dfrac{\cdot}{\lambda_N}\big)\Big).
\end{align}
For the first term we can directly use Bernstein:
\begin{align*}
{\Big\| \mathbb{P}_{\leq 4N^{-1}}\Big(\varphi_{2, N^{-1}} \sech\big(\dfrac{\cdot}{\lambda_N}\big)\Big) \Big\|}_{L^\infty}
 \lesssim N^{-1/2}
 {\big\| \varphi_{2, N^{-1}} \sech\big(\dfrac{\cdot}{\lambda_N}\big) \big\|}_{L^2}
\end{align*}
which is consistent with \eqref{estIINpr}.

For the second term in \eqref{estIINpr2}, observe that since
$\mathrm{supp}\,\widehat{\varphi}_{2,N^{-1}}\subset B(0,2N^{-1})$, 
and the whole term has frequencies larger than $4N^{-1}$,
then the frequencies of the $\sech$ function have to be larger than $2N^{-1}$:
\begin{align*}
\mathbb{P}_{\geq8N^{-1}}\Big(\varphi_{2, N^{-1}} \sech\big(\dfrac{\cdot}{\lambda_N}\big)\Big) = 
  \mathbb{P}_{\geq8N^{-1}} \Big(\varphi_{2, N^{-1}} 
  \cdot \mathbb{P}_{\geq 2N^{-1}} \sech\big(\dfrac{\cdot}{\lambda_N}\big)\Big) 
\end{align*}
Then, using also \eqref{ftsech}
\begin{align}\label{hatsech}
\left\vert \mathcal{F}\Big( \sech\big(\dfrac{\cdot}{\lambda_N}\big)\Big)\right\vert
  \approx \left\vert \lambda_N \sech(\pi\lambda_N \xi/2)\right\vert, 
\end{align}
we can estimate
\begin{align*}
\Big\| \mathbb{P}_{\geq8N^{-1}}\Big(\varphi_{2, N^{-1}} 
  \sech\big(\dfrac{\cdot}{\lambda_N}\big)\Big) \Big\|_{L^2}^2
  & \lesssim \Big\| \varphi_{2, N^{-1}} 
  \cdot \mathbb{P}_{\geq 2N^{-1}}\sech\big(\dfrac{\cdot}{\lambda_N}\big) \Big\|_{L^2}^2
  \\ 
  & \lesssim {\| \varphi_{2, N^{-1}} \|}_{L^2}^2
  \cdot \big\| \eta_{\geq 2N^{-1}}(\xi) \, \lambda_N \sech(\pi\lambda_N \xi/2) \big\|_{L^1_\xi}^2
  \\
  & \lesssim \delta^2 \cdot \exp\big(- \pi N^{-1}\lambda_N\big) \lesssim \delta^2 \kappa^{-9\pi}.
\end{align*}
due to our choice of $\lambda_N = N\log\kappa^9$ in \eqref{PsiN*}.
This completes the proof of \eqref{estIINpr} and the lemma.
\end{proof}

\medskip
It is worth to remark 
that, in the inequality \eqref{estI}, we lose localization (in space) 
when bounding $\sech^2(x)$ by $\sech^2(x/\log\kappa^9)$, 
but we gain the fact that the weight function and the solution are localized
on different scales at the frequency level, which shall be important in the next section.
For a similar reason 
we use the parameter $\lambda_N \gg N$ in the definition of the weight function $\Psi_{\lambda_N}$,
to produces a ``frequency gap'' with the function $\varphi_{2,N^{-1}}$
so to be able to apply Lemma \ref{weight_triple_prime}.


\bigskip
\section{Singular virial and integrability of the bad term}
\label{sec_Analysis_H}
The purpose of this section 
is to control the quantity
\begin{align}\label{secHquantity}
\int_0^{T_\mathrm{max}}\dfrac{1}{\mu (t)}\int_\R \varphi_{2,\geq}^2(t,x)\sech^{1/4}(x)dxdt.
\end{align}
Thanks to \eqref{deco_varphi2_sech} and Lemma \ref{lemIIN} we 
can reduce this task to bounding the right-hand sides of \eqref{estI}
and \eqref{estIIN}; notice that the latter has an additional $N^{-1}$ factor in front of the local intergral.

To begin, recall that 
\begin{align}
\label{secHparam}
\kappa = \kappa(t) :=\log^{1+\alpha}(e^D+t), \qquad  T_{\mathrm{max}} := \exp\big(\delta^{-\frac43(1-\alpha)}\big),
\end{align}
for arbitrarily small fixed $\alpha>0$,
and define 
the following weight function
\begin{align}\label{Psilambda}
\Psi_{\lambda}(x) := \lambda \tanh\big(\tfrac{x}{\lambda}\big).
\end{align}
To control the right-hand sides of  \eqref{estI} and \eqref{estIIN}
we define two ``singular" virial functionals;
the first one, which we will use to control \eqref{estI}, acts on $\varphi_{\geq 1}$ and is localized in space 
at a scale which is slightly larger than $1$:
\begin{align}\label{H>1}
\begin{split}
\mathcal{H}_{\geq 1}(t) := \dfrac{1}{\mu(t)} \int_\R \left( \Psi_{\lambda_1(t)}(x)\varphi_{2,\geq 1}(t,x) 
	+ \tfrac{1}{2}\Psi_{\lambda_1(t)}'(x)\partial_x^{-1}\varphi_{2,\geq 1}(t,x)\right)\partial_x^{-1}\varphi_{2,\geq 1,t}(t,x)dx,
	\\
	\lambda_1(t) := \log \kappa^9(t);
\end{split}
\end{align}
the second one, which we will use to control \eqref{estIIN}, is a ``singular" type functional
which acts on $\partial_x^{-1}\varphi_2$ localized at the frequency scale $1/N$, 
and with space localization $\lambda_N$:
\begin{align}\label{HN-1}
\begin{split}
\mathcal{H}_{N^{-1}}(t) :=
	 \dfrac{1}{\mu(t)} \int_\R \left( \Psi_{\lambda_N}(x)\varphi_{2,N^{-1}}(t,x) + \tfrac{1}{2} \Psi_{\lambda_N}'(x) \partial_x^{-1}\varphi_{2,N^{-1}}(t,x)\right)	
	 \partial_x^{-1}\varphi_{2,N^{-1},t}(t,x)dx, 
	 \\
	\lambda_N(t) := N \log \kappa^9(t).
\end{split}
\end{align}

In what follows we study the evolution of both functionals. The main results are Propositions \ref{propdotH1} and \ref{propdotHN}.
Later on in Section \ref{secprmt} we will show how to use these functionals to control the desired quantity \eqref{secHquantity}. 
Our starting point for the analysis of \eqref{H>1} and \eqref{HN-1} is the following basic modification of the virial identity:

\begin{lem}\label{lemH*}
Let $\mathcal{H}_{\ast}$ with the symbol $\ast \in \{ \geq \! 1, N^{-1} \}$ 
denote either of the functionals \eqref{H>1} or \eqref{HN-1}.
Then  
%
\begin{align}\label{dt_H_identity}
\begin{split}
\dfrac{d}{dt}\mathcal{H}_\ast & =  -\dfrac{1}{\mu}\int\varphi_{2,\ast}^2 \Psi_{\lambda}' 
	+ \dfrac{1}{4\mu} \int \big(\partial_x^{-1}\varphi_{2,\ast}\big)^2 \Psi_{\lambda}''' 
\\ & \qquad   + \dfrac{1}{\mu} \int \left(\varphi_{2,\ast} \Psi_{\lambda} + \tfrac{1}{2}\partial_x^{-1}\varphi_{2,\ast} \Psi_{\lambda}'\right)
	\big(1-\gamma\partial_x^2\big)^{-1}\partial_x^{-1} \partial_t \mathbb{P}_{\ast} SU \mathcal{N}
\\ &  \qquad  - \dfrac{\dot{\lambda}}{\mu } \int \left(\dfrac{x}{\lambda}\right) \left(\Psi'_{\lambda} \varphi_{2,\ast}
	+ \tfrac12\Psi_{\lambda}''\partial_x^{-1} \varphi_{2,\ast}  \right)\partial_x^{-1}\varphi_{2,\ast,t}
	+ \dfrac{\dot\lambda}{\mu\lambda} \int \Psi_{\lambda}\varphi_{2,\ast}\partial_x^{-1} \varphi_{2,\ast,t}-\dfrac{\dot{\mu}}{\mu}\mathcal{H}_\ast,
\end{split}
\end{align}
where the scale parameter $\lambda$ in the expression above is equal to $\lambda_1$ when $\ast$ stands for $\geq \! 1$ 
and is $\lambda_N$ when $\ast = N^{-1}$.

Moreover, the second term on the right-hand side of \eqref{dt_H_identity} can be absorbed by the first one
up to acceptable remainders:
\begin{align}\label{dxm1_upphi_triple_prime_into_LO}
\left\vert \int (\partial_x^{-1} \varphi_{2,\ast})^2 \Psi_{\lambda}''' \right\vert 
  \leq 
  \frac{1}{8} \int \varphi_{2,\geq}^2\Psi_{\lambda}'+O(\delta^2/\kappa).
\end{align}
\end{lem}

Notice that, thanks to \eqref{dxm1_upphi_triple_prime_into_LO}, the problem of controlling 
$\tfrac{d}{dt}\mathcal{H}_\ast$ essentially reduces to bounding the nonlinear terms in the second line of \eqref{dt_H_identity}.

\begin{proof}[Proof of Lemma \ref{lemH*}]
The proof is a direct calculation, similar to the one in \eqref{dt_J_dual}, up to small modifications.
The main point is that the functionals $\mathcal{H}_\ast$ are the virial momentum applied to
the ``singular" variable $\partial_x^{-1} \varphi_{2,\ast}$, instead of the standard $\varphi_1$.
Then, it suffices to use the basic identity \eqref{basicvirial} with $f = \partial_x^{-1} \varphi_{2,\ast}$,
$\Phi = \Psi_\lambda$ and $V=0$, and the equation
\begin{align*}
(\Box+1)\partial_x^{-1} \varphi_{2,\ast} = 
\partial_x^{-1} \partial_t  
(\Box+1) \varphi_{1,\ast} =
  \big(1-\gamma\partial_x^2\big)^{-1} \partial_x^{-1} \partial_t \mathbb{P}_{\ast }SU \mathcal{N} 
\end{align*}
having used the definition \eqref{defdual}, \eqref{op_SUconj}, the original equation \eqref{system_epsilon} and $SUY = 0$.
Note that, compared  to \eqref{dt_J_dual}, we do not need to differentiate in time any cutoff since 
the frequency scales are fixed (either $\geq 1$ or $N^{-1}$),
but we do need to differentiate in time the weight functions ($\mu^{-1}$ and $\Psi_\lambda$) 
which gives the last three terms.

To prove \eqref{dxm1_upphi_triple_prime_into_LO} it suffices to apply Lemma \ref{weight_triple_prime}.
In the case $\ast = N^{-1}$, since $\Psi'''_{\lambda_N} \leq 4\lambda_N^{-2} \Psi'_{\lambda_N}$, this gives
\begin{align*}
\left\vert \int (\partial_x^{-1} \varphi_{2,N^{-1}})^2 \Psi_{\lambda_N}''' \right\vert 
  \leq \frac{4c_1^*N^2}{\lambda_N^2} \int \varphi_{2,\geq}^2\Psi_{\lambda}' + O(\delta^2/\kappa)
\end{align*}
where $c_1^*>0$ is the absolute constant appearing in Lemma \ref{weight_triple_prime}. Therefore, it is enough to choose the absolute constant $D>1$ in the definition of $\kappa(t)$ large enough (as we have done in Lemma \ref{weight_triple_prime}). An analogous bound holds for $\varphi_{2,\geq 1}$ and the weight $\Psi_{\lambda_1}$.
\end{proof}




\medskip
\subsection{High-frequency case}
In this subsection we prove the following:
\begin{prop}\label{propdotH1}
Let $\mathcal{H}_{\geq 1}$ be the functional defined in \eqref{H>1} with $\varphi_2$ defined as in \eqref{defdual}
for $(\varepsilon_1,\varepsilon_2) \in C(\R,H^1\times L^2)$ solution to the system 
\eqref{system_epsilon}.
Fix $\epsilon \in (0,1)$. Then, the following holds:
\begin{align}\label{dotH1}
\begin{split}
\dfrac{d}{dt}\mathcal{H}_{\geq 1}
  \lesssim -\dfrac{1}{2\mu} \int \varphi_{2, \geq 1}^2 \Psi_{\lambda_1}'
  + \dfrac{\epsilon}{\mu}\int \big(w_t^2+w^2\big)\sech^{1/2}(x) 
  + \dfrac{\delta^{1/2}}{\mu}(a_1^2+a_2^2) + F(t) 
\end{split}
\end{align}
where
\begin{align}\label{dotH1rem}
F(t) \in \delta^2  L^1_t([0,T_{\max})
\end{align}
\end{prop}



The next result is the main estimate needed for the proof of Proposition \ref{propdotH1}.

\begin{lem}\label{lem_highfreq_dual_H_estimate_N}
Under the assumptions of Proposition \ref{propdotH1} we have the estimate
\begin{align*}
&\left \vert \int \left(\varphi_{2,\geq 1 } \Psi_{\lambda_1} + \tfrac{1}{2}\partial_x^{-1} \varphi_{2,\geq1 } \Psi_{\lambda_1}'\right)
	\big(1-\gamma\partial_x^2\big)^{-1}\partial_x^{-1} \partial_t \mathbb{P}_{\geq 1} SU \mathcal{N}  \right\vert
\\ 
& \qquad \qquad  \lesssim  \epsilon\int w^2\sech^{1/2}(x) +\epsilon \int w_t^2\sech^2(x) 
	+ 
	\delta^{1/2} \vert a_1a_2\vert
	+ 
	\delta^4 \, \lambda_1^2 + O(\delta^2/\kappa). 
\end{align*}
\end{lem}

\begin{proof}
We split the estimate according to the homogeneity of the terms 
that appear in $\partial_t \mathcal{N}$ with respect to $(\varepsilon_1,\varepsilon_2)$, see \eqref{Nepsilon'}.

\smallskip
\begin{itemize}[leftmargin=*]
\item[$\cdot$] Control of $\partial_t \varepsilon_{1}^3$: 
This case follows directly from Cauchy-Schwarz inequality and the fact that the $\partial_x^{-1}$ 
is in front of a projector to frequencies greater than $1$:
\begin{align*}
& \left\vert \int \left(\varphi_{2,\geq 1}\Psi_{\lambda_1} + \tfrac{1}{2}\partial_x^{-1}\varphi_{2,\geq 1}\Psi_{\lambda_1}'\right)
	\big(1-\gamma\partial_x^2\big)^{-1} \partial_x^{-1} \mathbb{P}_{\geq 1} SU 
	\partial_t \big(\varepsilon_1^3\big) \right\vert 
	\\
	& \qquad \qquad \lesssim \big( \lambda_1{\| \varphi_{2,\geq 1} \|}_{L^2} + {\| \partial_x^{-1}\varphi_{2,\geq } \|}_{L^2} \big)
	{\| \varepsilon_1^2\varepsilon_2 \|}_{L^2}
	\lesssim \lambda_1\delta^4.
\end{align*}

\smallskip
\item[$\cdot$] Control of $\partial_t\big(\varepsilon_1^2(Q+a_1Y)\big)$:
Similarly, by Cauchy-Schwarz and taking advantage of the decay of $Q$ and $Y$, 
we can estimate 
 \begin{align*}
& \left\vert \int \left(\varphi_{2,\geq 1}\Psi_{\lambda_1} + \tfrac{1}{2}\partial_x^{-1}\varphi_{2,\geq 1}\Psi_{\lambda_1}'\right)
	\big(1-\gamma\partial_x^2\big)^{-1} \partial_x^{-1} \mathbb{P}_{\geq 1} 
	SU \partial_t\big(\varepsilon_1^2 (Q+a_1Y)\big)\right\vert
\\ 
& \qquad \qquad \lesssim \big( \lambda_1{\| \varphi_{2,\geq 1} \|}_{L^2} + {\| \partial_x^{-1}\varphi_{2,\geq 1} \|}_{L^2} \big)
	\big( {\| \partial_t \varepsilon_1^2 \, Q  \|}_{L^2} 
	+  {\| \partial_t (\varepsilon_1^2a_1) \,Y \|}_{L^2} \big)
\\
& \qquad \qquad 
	\lesssim \lambda_1\delta \int \big(w_t^2+w^2\big)\sech^{1/2}(x) + \lambda_1 \delta^4 
	+ \lambda_1O(\delta^3/\kappa).
\end{align*}
where we recall that $w$ is defined in \eqref{w_def}, and we have used,
similarly to \eqref{loce}, the estimate
\begin{align}\label{H1pr10}
\begin{split}
\Vert \varepsilon_1 Q\Vert_{L^2}^2 + \Vert \varepsilon_2 Q\Vert_{L^2}^2
  & \lesssim \int ( w^2 + w_t^2) \sech^{1/2}(x) + \int  (\varepsilon_{1,<}^2 + \varepsilon_{2,<}^2) \sech^2(x) + O(\delta^2 \kappa^{-1}\kappa' )
  \\
  & \lesssim
	\int  ( w^2 + w_t^2) \sech^{1/2}(x) + O(\delta^2 \kappa^{-1}).
\end{split}
\end{align}
Since $\lambda_1=\log\kappa^9\ll\delta^{0-}$, this is consistent with the desired bound.

\smallskip
\item[$\cdot$] Control of $\partial_t\big(a_1\varepsilon_1(a_1Y^2+2 QY)\big)$: 
This case follows exactly the same lines as the previous one: 
 \begin{align*}
& \left\vert \int \left(\varphi_{2,\geq 1}\Psi_{\lambda_1} + \tfrac{1}{2}\partial_x^{-1}\varphi_{2,\geq 1}\Psi_{\lambda_1}'\right)
	\big(1-\gamma\partial_x^2\big)^{-1} \partial_x^{-1} \mathbb{P}_{\geq 1} 
	SU \partial_t \big(a_1\varepsilon_1(a_1Y^2+2 QY) \right\vert
\\ 
& \qquad \qquad \lesssim \big( \lambda_1{\| \varphi_{2,\geq 1} \|}_{L^2} + {\| \partial_x^{-1}\varphi_{2,\geq } \|}_{L^2} \big)
	\big( {\| \partial_t (\varepsilon_1 a_1^2) \, Y^2  \|}_{L^2} 
	+  {\| \partial_t (a_1 \varepsilon_1) \,Q Y \|}_{L^2} \big)
\\
& \qquad \qquad \lesssim  \lambda_1 \delta^4
	+ \lambda_1\delta \cdot (|a_1|+|a_2|) \cdot \Big( \int \big(w_t^2+w^2\big)\sech^{1/2}(x) + O(\delta^2 \kappa^{-1}) \Big)^{1/2}
\\
& \qquad \qquad \leq \epsilon^{-1} \lambda_1^2 \delta^4
	+ \epsilon \int \big(w_t^2+w^2\big)\sech^{1/2}(x) + O(\lambda_1\delta^3 \kappa^{-1/2}).
\end{align*}
 which is compatible with the desired bound (recall that $\lambda_1=\log\kappa^9$).
\smallskip
\item[$\cdot$] Control of $\partial_t\big(a_1^3Y^3 + 3a_1^2QY^2\big)$: 
For the 
terms that do not contain any $\varepsilon$ 
we just use a bound by the norms of $a_1,a_2$: 
\begin{align*}
& \left\vert \int \left(\varphi_{2,\geq 1}\Psi_{\lambda_1} + \tfrac{1}{2}\partial_x^{-1}\varphi_{2,\geq 1}\Psi_{\lambda_1}'\right)
	\big(1-\gamma\partial_x^2\big)^{-1} \partial_x^{-1} \mathbb{P}_{\geq 1} 
	SU \partial_t\big(a_1^3Y^3 + 3a_1^2QY^2\big)\right\vert
	\lesssim \lambda_1\delta |a_1| |a_2|.
\end{align*}
This is acceptable for the desired inequality since $\lambda_1 \lesssim \log\log (e^D+t) \ll \delta^{0-}$.
\end{itemize}
This concludes the proof of the lemma.
\end{proof}

\medskip
We now give the proof of Proposition \ref{propdotH1}.

\begin{proof}[Proof of Proposition \ref{propdotH1}]
In view of the formula \eqref{dt_H_identity}, the estimate \eqref{dxm1_upphi_triple_prime_into_LO}
which takes care of the first two terms on the right-hand side,
and the estimate in Lemma \ref{lem_highfreq_dual_H_estimate_N} which takes care of the terms involving the nonlinearity $\mathcal{N}$,
we see that in order to obtain \eqref{dotH1} with \eqref{dotH1rem} it suffices to prove that the other terms in \eqref{dt_H_identity}
are time integrable.
In particular, it is enough to show the following three bounds:
\begin{align}\label{dotH1est1}
& \left\vert \dfrac{\dot{\lambda}_1}{\mu} \int \left(\dfrac{x}{\lambda_1}\right) \left(\Psi'_{\lambda_1}\varphi_{2,\geq 1}
 	+ \tfrac12 \Psi_{\lambda_1}'' \partial_x^{-1}\varphi_{2,\geq 1} \right) \partial_x^{-1}
	\varphi_{2,\geq 1,t} \right \vert \lesssim \delta^2 \mu^{-2},
\\\
\label{dotH1est2}
& \Big| \dfrac{\dot{\lambda}_1}{\mu\lambda_1} \int \Psi_{\lambda_1} \varphi_{2,\geq }\partial_x^{-1} \varphi_{2,\geq 1,t} \Big|
	\lesssim \delta^2 \mu^{-2},
\\
\label{dotH1est3}
& \Big| \dfrac{\dot\mu}{\mu}\mathcal{H}_{\geq 1} \Big| \lesssim \delta \mu^{-2}.
\end{align}
To prove \eqref{dotH1est1} we recall that $\lambda_1 = \log \kappa^9(t)$, so that $\dot\lambda_1 \ll \mu^{-1}$,
use that
\[
\big\Vert \big(\tfrac{\cdot}{\lambda_1}\big)\Psi_{\lambda_1}'(\cdot)\big\Vert_{L^\infty} \lesssim 1,
\]
and Cauchy-Schwarz together with the bound 
\begin{align}\label{est123aux}
\| \partial_x^{-1} \varphi_{2,\geq 1} \|_{L^2} + \| \partial_x^{-1} \varphi_{2,\geq 1, t} \|_{L^2} \lesssim \delta.
\end{align}
For \eqref{dotH1est2} we use Cauchy-Schwarz, $|\Psi_{\lambda_1}| \lesssim \lambda_1$ and \eqref{est123aux}.
The estimate of \eqref{dotH1est3} is also obtained similarly from Cauchy-Schwarz and \eqref{est123aux}.
This concludes the proof of Proposition \ref{propdotH1}.
\end{proof}


\smallskip
We remark that \eqref{dotH1}-\eqref{dotH1rem} will allow us to control $\varphi_{2,\geq 1}$
provided we can control the local norms of $w$ (and the amplitudes $a_1,a_2$).
In turn, these norms 
are controlled by using the local norm of $\varphi_{1,\geq}$ (or, equivalently, of $z$)
as it appears in \eqref{dj_m_dk_ineq},
also thanks to the coercivity estimate \eqref{prop_coer_varepsilon}.  


\medskip
\subsection{Low frequency case}
We now proceed to estimate the evolution of the singular ``low frequency" functional \eqref{HN-1}.

\begin{prop}\label{propdotHN}
Let $\mathcal{H}_{N^{-1}}$ be the functional defined in \eqref{HN-1} with $\varphi_2$ defined as in \eqref{defdual}
for $(\varepsilon_1,\varepsilon_2) \in C(\R,H^1\times L^2)$ a solution to the system 
\eqref{system_epsilon}.
Fix $\epsilon \in (0,1)$. Then, the following holds:
\begin{align}\label{dotHN}
\begin{split}
\frac{1}{N} \dfrac{d}{dt}\mathcal{H}_{N^{-1}}
  \lesssim -\dfrac{1}{\mu} \frac{1}{N} \int \varphi_{2, N^{-1}}^2\Psi_{\lambda_N}'
  + \dfrac{1}{\mu \log\kappa} \frac{\epsilon}{N^{0+}}
  \int \big(w_t^2+w^2\big)\sech^{1/2}(x) 
\\ 
\qquad + \dfrac{\delta^{1/4}}{\mu}(a_1^2+a_2^2) + F_N(t) 
\end{split}
\end{align}
where
\begin{align}\label{dotHNrem}
\log\kappa(t)\cdot\sum_{1\leq N \leq 2\kappa(t)} F_N(t) \in \delta^2  L^1_t([0,T_{\max}]).
\end{align}
\end{prop}

Once again, in view of the virial identity \eqref{dt_H_identity} and Lemma \ref{weight_triple_prime} ,
the main task is to bound the  nonlinear terms in $\mathcal{N}$.

\begin{lem}
\label{lem_dual_H_lf_estimate_N}
Under the same assumptions of Proposition \ref{propdotHN},
for $\epsilon>0$ small, and for all times $t\in(0,T_\mathrm{max})$, the following estimate holds
\begin{align}\label{lemHNest}
\begin{split}
\left\vert \int \left(\varphi_{2, N^{-1} }\Psi_{\lambda_N} + \tfrac{1}{2}\partial_x^{-1}\varphi_{2, N^{-1} }\Psi_{\lambda_N}'\right)
	\big(1-\gamma\partial_x^2\big)^{-1}\partial_x^{-1}\partial_t \mathbb{P}_{ N^{-1}}SU \mathcal{N}  \right\vert  
\\ 
\quad  \lesssim  \lambda_N N^{1/2} \delta^4 + \lambda_N N^{1/2} \delta \big(a_1^2+a_2^2 \big)
	+ \frac{\epsilon N^{1-}}{\log\kappa} \int \big(w_t^2+w^2\big)\sech^{1/2}(x)
\\
\quad  \qquad + \epsilon \int \varphi_{2, N^{-1}}^2 \Psi_{\lambda_N}' + 
	\lambda_N N^{1/2}O(\delta^3/\kappa) + \dfrac{N^{1-}}{\log\kappa} O(\delta^2/\kappa).
\end{split}
\end{align}
\end{lem}

\begin{proof}

We proceed as before analyzing the nonlinear terms by homogeneity.
Since some of the estimates will be similar to those in the proof of Proposition \ref{propdotH1},
we will skip some of the details.

\begin{itemize}[leftmargin=*]

\smallskip
\item[$\cdot$] Control of $\partial_t(\varepsilon_{1}^3)$:
Let us denote
\begin{align*}
& \int \left(\varphi_{2, N^{-1} }\Psi_{\lambda_N}+\tfrac{1}{2}\partial_x^{-1}\varphi_{2, N^{-1} }\Psi_{\lambda_N}'\right)\big(1-\gamma\partial_x^2\big)^{-1}\partial_x^{-1}\mathbb{P}_{ N^{-1}} SU \partial_t \big(\varepsilon_1^3\big) =  \mathcal{N}_{1,1} +  \mathcal{N}_{1,2}  
\\
&  \mathcal{N}_{1,1} :=  \int \left(\varphi_{2, N^{-1} }\Psi_{\lambda_N} \right)
	\big(1-\gamma\partial_x^2\big)^{-1}\partial_x^{-1}\mathbb{P}_{ N^{-1}} SU \partial_t \big(\varepsilon_1^3\big)
\\
& \mathcal{N}_{1,2} := 
	\int \left(\tfrac{1}{2}\partial_x^{-1}\varphi_{2, N^{-1} }\Psi_{\lambda_N}'\right)\big(1-	\gamma\partial_x^2\big)^{-1}\partial_x^{-1}\mathbb{P}_{ N^{-1}} 
	SU \partial_t \big(\varepsilon_1^3\big).
\end{align*}
The first term is the most dangerous and the one that gives the main restriction for $T_{\max}$.
Using Cauchy-Schwarz, recalling the definition of $\Psi_\lambda$ in \eqref{Psilambda}, 
and applying Bernstein inequality going from $L^2$ to $L^1$, we can bound
\begin{align*}
\big\vert \mathcal{N}_{1,1} \big\vert 
& \lesssim \lambda_N \Vert \varphi_{2,N^{-1}}\Vert_{L^2_x}
	 \Vert \big(1-\gamma\partial_x^2\big)^{-1}\partial_x^{-1}\mathbb{P}_{ N^{-1}} SU \partial_t \big(\varepsilon_1^3\big) \Vert_{L^2_x}
\\
& \lesssim \lambda_N \Vert \varphi_{2,N^{-1}}\Vert_{L^2_x}
	 \cdot N^{1/2} \Vert 
	 \partial_t (\varepsilon_1^3 ) \Vert_{L^1_x} \lesssim \lambda_N N^{1/2}\delta^4.
\end{align*}
This is consistent with the desired inequality.
The other term can be estimated using Cauchy-Schwarz, followed by Lemma \ref{weight_triple_prime}, 
and application of Bernstein as above:
\begin{align*}
\big\vert \mathcal{N}_{1,2}\big\vert \lesssim \Vert \partial_x^{-1}\varphi_{2, N^{-1} }\Psi_{\lambda_N}' \Vert_{L^2_x}
	 \Vert \big(1-\gamma\partial_x^2\big)^{-1}\partial_x^{-1}\mathbb{P}_{ N^{-1}} SU \partial_t \big(\varepsilon_1^3\big) \Vert_{L^2_x}
\\
\lesssim \Big( N  \Vert \varphi_{2,N^{-1}}\Psi_{\lambda_N}' \Vert_{L^2_x} + \delta \kappa^{-1/2} \Big)
	 \cdot N^{1/2} \Vert 
	 \partial_t (\varepsilon_1^3 ) \Vert_{L^1_x}
\\
\lesssim \epsilon\int \varphi_{2,N^{-1}}^2\Psi_{\lambda_N}' + \epsilon^{-1}N^3\delta^6 + \delta^4.
\end{align*}
having also used $N \lesssim \kappa$.

\smallskip
\item[$\cdot$] Control of $\partial_t\big(\varepsilon_1 ^2\big(Q+a_1Y\big)\big)$: 
In this case we denote the corresponding integral by  
\begin{align*}
\mathcal{N}_{2} &:= \int \left(\varphi_{2, N^{-1} }\Psi_{\lambda_N}+\tfrac{1}{2}\partial_x^{-1}\varphi_{2, N^{-1} }\Psi_{\lambda_N}'\right)\times 
\\
& \qquad \qquad \qquad  \times \big(1-\gamma\partial_x^2\big)^{-1}\partial_x^{-1}
	\mathbb{P}_{ N^{-1}} SU \partial_t\Big(\varepsilon_1^2(Q+a_1Y)\Big).
\end{align*}
Applying Cauchy-Schwarz, followed by Lemma \ref{weight_triple_prime} and Bernstein inequalities as before,
and using also \eqref{H1pr10} (which holds verbatim with $\sqrt{Q}$ instead of $Q$), 
we can estimate 
\begin{align*}
\big\vert \mathcal{N}_2\big\vert & \lesssim
	\Big(  \lambda_N \Vert \varphi_{2,N^{-1}}\Vert_{L^2_x} 
	+ N  \Vert \varphi_{2,N^{-1}}\Psi_{\lambda_N}' \Vert_{L^2_x} + \delta \kappa^{-1/2} \Big)
	 \cdot N^{1/2} \Vert 
	 (\varepsilon_1^2 + \varepsilon_2^2 ) Q \Vert_{L^1_x}
\\
& \lesssim 
	\lambda_N \delta 
	\cdot  N^{1/2}  
	\Big( \int \big(w^2+w_t^2\big)\sech^{1/2}(x)+ O(\delta^2/\kappa) \Big)
\end{align*}
Since $N \lesssim \kappa \leq \log T_{\max}$, then $ \lambda_N \delta N^{1/2} \ll \epsilon N^{1-} / \log\kappa$, and the terms
above involving $w$ are consistent with the desired \eqref{lemHNest}; the remaining terms 
of $O(\lambda_N N^{1/2} \delta^3 \kappa^{-1} )$ are also accounted for.

\smallskip
\item[$\cdot$] Control of $\partial_t\big(a_1\varepsilon_1\big(a_1Y^2+ 2QY\big)\big)$: 
Denote the corresponding integral by 
\begin{align*}
\mathcal{N}_{3} &:=   \int \left(\varphi_{2, N^{-1} }\Psi_{\lambda_N} + \tfrac{1}{2}\partial_x^{-1}\varphi_{2, N^{-1}}\Psi_{\lambda_N}'\right)
	 \times 
\\ & \qquad \qquad \qquad \times \big(1-\gamma\partial_x^2\big)^{-1}\partial_x^{-1}\mathbb{P}_{ N^{-1}}SU
	\partial_t\Big(a_1\varepsilon_1(a_1Y^2+2QY)\Big).
\end{align*}
Proceeding as above, using again \eqref{H1pr10}, 
we can estimate
\begin{align*}
\big\vert \mathcal{N}_{3} \big\vert 
	&\lesssim \lambda_N\delta \cdot N^{1/2} \big(\vert a_1\vert+\vert a_2\vert\big) \big( \Vert \varepsilon_{1}Y \Vert_{L^2}
	 + \Vert \varepsilon_2Y\Vert_{L^2} \big)
\\ &  \lesssim \epsilon^{-1} N^{0+} \log\kappa \, \lambda_N^2\delta^2\big(a_1^2+a_2^2\big)
 + \frac{\epsilon N^{1-}}{\log\kappa} \Big( \int \big(w_t^2+w^2\big)\sech^2(x) + O(\delta^2/\kappa) \Big) .
\end{align*}
As before, since $N \lesssim \kappa \leq \log T_{\max}$, then $ \lambda_N \delta N^{0+} \log\kappa \leq N^{1/2}$ so that the terms
containing $a_1$ and $a_2$ are consistent with \eqref{lemHNest}.

\smallskip
\item[$\cdot$] Control of $\partial_t\big(a_1^3Y^3 + 3a_1^2QY^2\big)$: 
It only remains to bound the terms that do not contain any $\varepsilon$. 
Denote
\begin{align*}
\mathcal{N}_{4} &:= \int \left(\varphi_{2,N^{-1} }\Psi_{\lambda_N}+\tfrac{1}{2}\partial_x^{-1}\varphi_{2,N^{-1} }\Psi_{\lambda_N}'\right)\times 
\\ & \qquad \qquad \qquad \times \big(1-\gamma\partial_x^2\big)^{-1}\partial_x^{-1}SU\,\widetilde{\mathbb{P}}_{N^{-1}}\partial_t\big(a_1^2\big(a_1Y^3 + 3QY^2\big)\big).
\end{align*}
Similarly to before, using Cauchy-Schwarz and Bernstein inequality, we get that  
\begin{align*}
\big\vert \mathcal{N}_{4} \big\vert \lesssim \lambda_N \delta \cdot N^{1/2} \vert a_1a_2\vert .
\end{align*}
\end{itemize}
This is consistent with \eqref{lemHNest} and concludes the proof of the lemma.
\end{proof}

We now prove Proposition \ref{propdotHN}.

\begin{proof}[Proof of Proposition \ref{propdotHN}]

The starting point is the identity \eqref{dt_H_identity} (divided by $N$).
The estimate \eqref{dxm1_upphi_triple_prime_into_LO}
takes care of the first two terms on the right-hand side of  \eqref{dt_H_identity}, absorbing the second by the first,
and giving a contribution consistent with the first term on the right-hand side of the desired inequality \eqref{dotHN}
and the remainder \eqref{dotHNrem}:
\begin{align}\label{propdotHNlot}
 -\dfrac{1}{\mu N}\int\varphi_{2,N^{-1}}^2 \Psi_{\lambda_N}' 
	+ \dfrac{1}{4\mu N} \int \big(\partial_x^{-1}\varphi_{2,N^{-1}}\big)^2 \Psi_{\lambda_N}''' 
	\leq  -\dfrac{2}{3\mu N}\int\varphi_{2,N^{-1}}^2 \Psi_{\lambda_N}'  
	+ \dfrac{1}{N}O(\delta^2/(\mu \kappa ^{1+})).
\end{align}
Lemma \ref{lem_dual_H_lf_estimate_N} handles the term involving the nonlinearity $\mathcal{N}$ in \eqref{dt_H_identity};
the contribution of this term to \eqref{dotHN}, 
obtained by dividing \eqref{lemHNest} by $N\mu$, is:
\begin{align*}
& \frac{1}{\mu} \Big( \lambda_N N^{-1/2} \delta^4 + \lambda_N N^{-1/2} \delta \big(a_1^2+a_2^2 \big)
	+ \frac{\epsilon}{N^{0+} \log \kappa} \int \big(w_t^2+w^2\big)\sech^{1/2}(x)
\\
& \quad  \qquad + \frac{\epsilon}{N} \int \varphi_{2, N^{-1}}^2 \Psi_{\lambda_N}' + 
	\lambda_N N^{-1/2}O(\delta^3/\kappa) + \dfrac{1}{N^{0+}\log\kappa}O(\delta^2/\kappa) \Big). 
\end{align*}
Let us analyze all the terms above.

The local norm of $\varphi_{2, N^{-1}}$ appears with an $\epsilon$ and is therefore consistent with the first term
on the right-hand side of \eqref{dotHN}. The same holds for the local norm of $w,w_t$,
and for the terms in $a_1,a_2$ upon observing that (recall $N$ is a dyadic index)
$$\sum_{1 \leq N \leq 2\kappa(t)} \lambda_N N^{-1/2} \delta \lesssim \delta \kappa^{1/2} \log^2 \kappa \leq \delta^{1/4}.$$
For the same reason also the term $\lambda_N N^{-1/2} \delta^3/(\mu \kappa)$  can be included in the remainder $F_N$;
the same also hold true for $\delta^2/(N^{0+}\mu\kappa\log\kappa)$.

For the last remaining term we verify directly that (summing over $N$ only loses $\approx \log \kappa$)
\begin{align*}
\int_0^{T_{\max}} \frac{1}{\mu} \, \log\kappa(t) \sum_{1\leq N \leq 2\kappa(t)} \lambda_N N^{-1/2} \, \delta^4 \, dt 
	& \lesssim \delta^4 \int_0^{T_{\max}} \frac{1}{\langle t \rangle} \kappa^{1/2}(t) \log^3 \kappa(t) \, dt 
	\\
	&  \lesssim \delta^4\,  \log^{3/2+\alpha/2 +}(T_{\max})
\end{align*}
which is $O(\delta^2)$ with our definition of $T_{\max}$ in \eqref{secHparam}. 


From the above estimates, we see that in order to conclude it suffices to prove that the 
other terms in \eqref{dt_H_identity} (with $\ast = N^{-1}$, $\lambda=\lambda_N$) 
are remainders in the sense of \eqref{dotHNrem}. 
In particular, it is enough to show the following three (far from optimal) bounds
\begin{align}\label{dotHNest1}
& \Big| \dfrac{\dot{\lambda}_N}{\mu} \int \left(\dfrac{x}{\lambda_1}\right) \left(\Psi'_{\lambda_N}\varphi_{2,N^{-1}}
 	+ \tfrac12 \Psi_{\lambda_N}'' \partial_x^{-1} \varphi_{2,N^{-1}} \right) \partial_x^{-1}
	\varphi_{2,N^{-1},t} \Big| \lesssim \delta^2 \mu^{-3/2},
\\\
\label{dotHNest2}
& \Big| \dfrac{\dot{\lambda}_N}{\mu\lambda_N} \int \Psi_{\lambda_N} \varphi_{2,N^{-1}} \partial_x^{-1} \varphi_{2,N^{-1},t} \Big|
	\lesssim \delta^2 \mu^{-3/2},
\\
\label{dotHNest3}
& \Big| \dfrac{\dot{\mu}}{\mu}\mathcal{H}_{N^{-1}} \Big| \lesssim \delta^2 \mu^{-3/2}.
\end{align}
These all follow in a straightforward fashion using Cauchy-Schwarz,
recalling that our parameters satisfy $\lambda_N = N \log \kappa^9$,
so that $\dot{\lambda}_N \ll 
N \mu^{-1}$, $N \lesssim \kappa \lesssim \log^{1+\alpha}(2+t)$, $\mu \approx t$, and that we can simply bound
\begin{align*}
{\| \partial_x^{-1} \varphi_{2,N^{-1}} \|}_{L^2} + {\| \partial_x^{-1} \varphi_{2,N^{-1},t}  \|}_{L^2} \leq N \delta.
\end{align*}
This concludes the proof of the main Proposition \ref{propdotHN}.
\end{proof}

%



\bigskip
\section{Proof of Theorem \ref{maintheo}}\label{secprmt}
We are finally ready to prove our main theorem. 
First, let us define by $\mathcal{V}_\theta$ the following combination of 
the functionals in \eqref{I>}, \eqref{true_momentum_original_virial_high_frequency},
\eqref{defJ} and \eqref{defK},
 \begin{align}\label{def_Vc}
\mathcal{V}_\theta  := \mathcal{I}_\geq + \mathcal{P}_\geq + \theta\mathcal{J}_\geq - \theta\mathcal{K}_\geq,
\end{align}
with $\theta > 1$ a positive parameter to be chosen later.
Eventually we are going to combine the main inequalities \eqref{final_2_og_virial_hf}
and \eqref{dj_m_dk_ineq} together with the control on $z_t$ (or, equivalently, on $\varphi_{2,\geq}$) obtained 
in Section \ref{sec_Analysis_H}, and 
the role of $\theta > 1$ in \eqref{def_Vc}
is to make sure we can absorb the local norm of $w$ and $w_t$ in \eqref{final_2_og_virial_hf}
by those of $z$ and $z_t$ 
using Lemma \ref{prop_coer_varepsilon} (recall also \eqref{w_def} and \eqref{def_z}).

\medskip
We begin the proof 
with the following: 

\begin{lem}\label{jgeq_kgeq_final}
Let $(\varepsilon_1,\varepsilon_2)$ be any solution to system \eqref{system_epsilon} 
satisfying \eqref{stable_manifold_hyp_thm}, $(\varphi_{1,\geq},\varphi_{2,\geq})$ defined as 
in \eqref{defdual}. Recall the definitions \eqref{w_def} and \eqref{def_z}.
Then, for $\epsilon>0$ small, for all times $t\in(0,T_\mathrm{max})$ the following estimate holds 
\begin{align}\label{dotJ-K}
\begin{split}
\frac{d}{dt} \mathcal{J}_\geq - \frac{d}{dt} \mathcal{K}_\geq 
  \leq - \dfrac{1}{2\mu}\int z_x^2 - \dfrac{1}{8\mu}\int z^2\sech^{1/4}(x) 
  + 
  \dfrac{\epsilon}{\mu}(a_1^2+a_2^2)
\\ 
+ \dfrac{\epsilon}{\mu}\int \big(w_t^2+w^2\big)\sech^{1/2}(x)+F_\star(t),
\end{split}
\end{align}
where $F_\star(t)$ is some function satisfying
\begin{align}\label{dotJ-KF}
\int_0^{T_\mathrm{max}} F_\star(t)dt = O(\delta^2).
\end{align}
\end{lem}

\smallskip
\begin{proof}
Recall the main estimate for $\dot{\mathcal{J}}_\geq - \dot{\mathcal{K}}_\geq$ 
in \eqref{dj_m_dk_ineq}-\eqref{def_r2_dual_virial_hf} and notice that in order to obtain
\eqref{dotJ-K} we only need to control the local norm of $z_t = \partial_t (\varphi_{1,\geq} \sqrt{\Phi'_B})$.
To do this we use 
the decomposition in \eqref{deco_varphi2_sech} followed by Lemma \ref{lemIIN}
(recalling also \eqref{Psilambda} and \eqref{H>1})
and by Propositions \ref{propdotH1} and \ref{propdotHN} to deduce that
\begin{align}\label{zt_final_bound}
\begin{split}
& \dfrac{1}{\mu}\int z_t^2\sech^{1/4}(x)
\\ 
&  \qquad \lesssim \dfrac{1}{\mu}\int \varphi_{2,\geq 1}^2
  \Psi'_{\lambda_1}
  + \log\kappa(t) \sum_{\substack{N\in\mathbb{D}, \\ 1< N\leq2\kappa}} 
  \dfrac{1}{N\mu}\int \varphi_{2, N^{-1}}^2 \Psi_{\lambda_N}' + \dfrac{1}{\mu} O(\delta^2/\kappa)
%
%
%
%
%
%
\\ 
& \qquad  \lesssim \dfrac{\epsilon}{\mu}\int \big(w_t^2+w^2\big)\sech^{1/2}(x)
  + \dfrac{\epsilon}{\mu}\big(a_1^2+a_2^2\big) + F(t) -\dot{\mathcal{H}}_{\geq1} 
\\ 
&  \qquad \quad - \log\kappa(t) \sum_{ \substack{N\in\mathbb{D}, 
\\ 
1< N\leq2\kappa}}\Big(\dfrac1N\dot{\mathcal{H}}_{N^{-1}}-F_N(t)\Big)
  + \dfrac{1}{\mu}O(\delta^2/\kappa);
\end{split}
\end{align}
notice that we have used that
\begin{align*}
& \log\kappa(t) \sum_{ \substack{N\in\mathbb{D}, \\ 1< N\leq2\kappa}} 
 \dfrac{\delta^{1/4}}{\mu}\big(a_1^2+a_2^2\big)
 \lesssim  \dfrac{\epsilon}{\mu} \big(a_1^2+a_2^2\big) 
\end{align*}
since $\log^2 \kappa(t) \delta^{1/4} \ll  \delta^{1/8}$,
and that we are slightly abusing notation by using the same $\epsilon$ than the one in the statement.


Using \eqref{zt_final_bound}
we can then improve 
\eqref{dj_m_dk_ineq} to \eqref{dotJ-K}, 
recalling also the definition of $\mathfrak{R}_2$ in \eqref{def_r2_dual_virial_hf} choosing, for instance, $\epsilon \approx B^{-1}$,
and defining
\begin{align}
\label{Jgeq_Kgeq_finalRem}
F_{\star}(t) := F(t) - \dot{\mathcal{H}}_{\geq1} 
  - \log\kappa(t) \sum_{ 1< N\leq2\kappa(t)}\Big(\dfrac1N\dot{\mathcal{H}}_{N^{-1}} 
  - F_N(t)\Big) + \dfrac{1}{\mu}O(\delta^4) + \dfrac{1}{\mu}O(\delta^2/\kappa).
\end{align}
To conclude the proof of the lemma we then need to show the property \eqref{dotJ-KF}.
Since the last two terms in \eqref{Jgeq_Kgeq_finalRem} are clearly 
consistent with  \eqref{dotJ-KF}, 
and the same holds true for $F(t)$, respectively $ F_N(t)$, in view of \eqref{dotH1rem}, respectively \eqref{dotHNrem},
it will suffice to show
%
%
\begin{align}\label{boundary_H_geq1}
& \left\vert \int_0^{T_\mathrm{max}}\dfrac{d}{dt}\mathcal{H}_{\geq 1}(t) \,dt \right\vert 
  \lesssim \delta^2,
\\ 
\label{boundary_H_N1} 
& \left\vert \int_0^{T_\mathrm{max}} \log\kappa(t) \sum_{1\leq N\leq 2\kappa(t)}\dfrac{1}{N}
  \dfrac{d}{dt}\mathcal{H}_{N^{-1}}(t) \, dt \right\vert \lesssim \delta^2.
\end{align}
Unlike in the case of standard virial functionals,
the presence of the $\partial_x^{-1}$ operator in the definition of $\mathcal{H}_*$, 
makes the estimates of the boundary terms coming from integrating (in time) the left-hand side of \eqref{dotHN} 
slightly more delicate. 

\medskip
{\it Proof of \eqref{boundary_H_geq1}}.
From the Fundamental Theorem of Calculus, and Cauchy-Schwarz,
with $T = T_{\max}$, we get
\begin{align}
\nonumber
\left\vert \int_0^{T} \dfrac{d}{dt}\mathcal{H}_{\geq 1}(t)dt\right\vert
 & \lesssim 
 \left\vert \mathcal{H}_{\geq1}(0)\right\vert + \left\vert \mathcal{H}_{\geq1}(T) \right\vert,
\\
\label{boundarypr1}
& \lesssim 
\dfrac{1}{\mu(0)}\Big(\lambda_1(0)\Vert \mathbb{P}_{\geq1}\varphi_2(0)\Vert_{L^2_x}
 + \Vert \partial_x^{-1}\mathbb{P}_{\geq1}\varphi_2(0)\Vert_{L^2_x}\Big)
 \Vert \partial_x^{-1}\partial_t\mathbb{P}_{\geq1}\varphi_2(0)\Vert_{L^2_x}
 \\
 \label{boundarypr2}
 & + 
\dfrac{1}{\mu(T)}\Big(\lambda_1(T)\Vert \mathbb{P}_{\geq1}\varphi_2(T)\Vert_{L^2_x}
 +\Vert \partial_x^{-1}\mathbb{P}_{\geq1}\varphi_2(T)\Vert_{L^2_x}\Big)
 \Vert \partial_x^{-1}\partial_t\mathbb{P}_{\geq1}\varphi_2(T)\Vert_{L^2_x}.
\end{align}
Then, recalling that at $t=0$, $\lambda_1(0)=\log\kappa(0)=O(1)$, and using $\partial_t\varphi_{2,\geq1}=-\mathcal{L}_0\varphi_{1,\geq1}  
	+   \big(1-\gamma\partial_x^2\big)^{-1} \mathbb{P}_{\geq 1}SU \mathcal{N}$,
we can bound
\begin{align*}
\eqref{boundarypr1} \lesssim \Vert \mathbb{P}_{\geq1} \varphi_2(0) \Vert_{L^2_x}  
\Vert \partial_x^{-1}\partial_t\mathbb{P}_{\geq1}\varphi_2(0)\Vert_{L^2_x}
	\lesssim \delta \big( \Vert \varphi_1(0)\Vert_{H^1_x}+\Vert \mathcal{N}(0)\Vert_{L^2_x} \big) \lesssim \delta^2.
\end{align*}
For the term in \eqref{boundarypr2} we can proceed similarly to obtain
\begin{align*}
\eqref{boundarypr2} \lesssim \frac{1}{\mu(T)} \log \kappa(T) \Vert \mathbb{P}_{\geq1} \varphi_2(T) \Vert_{L^2_x}  
\Vert \partial_x^{-1}\partial_t\mathbb{P}_{\geq1}\varphi_2(T)\Vert_{L^2_x}
	\lesssim \delta^2.
\end{align*}

\medskip
{\it Proof of \eqref{boundary_H_N1}}.
This is more delicate than the previous case. In fact, first of all it is useful to recall that
\[
\kappa(t)\leq \kappa(T_\mathrm{max}) =  c \delta^{-(4/3)^-} .
\]
Moreover, from the definition of $\kappa(t)$ we have the constraint
\[
N/2 \leq \kappa(t) = \log^{ 1+\alpha } \big(e^D+t\big) \quad \Longleftrightarrow
	 \quad t \geq \big(e^{(N/2)^{1/(1+\alpha)}}-e^D\big)\vee 0=:T_N.
\]
Using this we can write the quantity to be estimated in \eqref{boundary_H_N1} as (again, we denote $T=T_{\max}$)
\begin{align*}
I(T) & = \int_0^{T} \log\kappa(t) \sum_{1\leq N\leq 2\kappa(t)}\dfrac{1}{N} \dfrac{d}{dt}\mathcal{H}_{N^{-1}}(t) \, dt
	\\
	& = \int_0^{T} \log\kappa(t) \sum_{\substack{N\in\mathbb{D}, \\ 1< N\leq c \delta^{-4/3}} }
	\dfrac{1}{N} \, \mathds{1}_{\kappa(t)\geq \frac12N} \, \dfrac{d}{dt}\mathcal{H}_{N^{-1}}(t) \, dt .
\end{align*}
The advantage of writing $N<c\delta^{-4/3}$ is to take out the characteristic function $\mathds{1}$, 
while keeping the sum finite, so that we can easily switch the sum and integral signs and, also integrating by parts, obtain: 
\begin{align}
\nonumber
I(T) & = \sum_{\substack{N\in\mathbb{D}, \\ 1< N\leq c \delta^{-4/3}} }\dfrac{1}{N}
		\int_{T_N}^{T_\mathrm{max}} \log\kappa(t) \dfrac{d}{dt}\mathcal{H}_{N^{-1}}(t) \, dt
\\
& = \sum_{\substack{N\in\mathbb{D}, \\ 1< N\leq c \delta^{-4/3}} } 
	\dfrac{1}{N} \Big( \log\kappa(T_{\max}) \mathcal{H}_{N^{-1}}(T_\mathrm{max}) 
	- \log\kappa(T_N) \mathcal{H}_{N^{-1}}(T_N)\Big)\label{intHN}
\\
\nonumber
& - \sum_{\substack{N\in\mathbb{D}, \\ 1< N\leq c \delta^{-4/3} }}\dfrac{1}{N}
		\int_{T_N}^{T_\mathrm{max}}\frac{\dot\kappa(t)}{\kappa(t)} \, \mathcal{H}_{N^{-1}}(t) \, dt. 
\end{align}
It only remains to estimate the above sums.
We begin by bounding the first sum associated with the evaluation at $T_N$. 
Let us fix $N\in\mathbb{D}$, and recall one more time that 
$\partial_t\varphi_{2,N^{-1}}=-\mathcal{L}_0\varphi_{1,N^{-1}} + \big(1-\gamma\partial_x^2\big)^{-1} \mathbb{P}_{N^{-1}}SU \mathcal{N}$, 
so that, using Cauchy-Schwarz and Bernstein inequality,
\begin{align*}
& \dfrac{\log\kappa(T_N) }{N}\vert \mathcal{H}_{N^{-1}}(T_N)\vert 
\\ 
&\qquad = \dfrac{ \log\kappa(T_N) }{N\mu(T_N)} \left\vert \int \big(\Psi_{\lambda_N(T_N)}\varphi_{2,N^{-1}}(T_N)
	+\frac{1}{2}\Psi_{\lambda_N(T_N)}'\partial_x^{-1}\varphi_{2,N^{-1}}(T_N)\big)\partial_x^{-1}\varphi_{2,N^{-1},t}(T_N)\right\vert
\\ 
& \qquad \lesssim \dfrac{\log\kappa(T_N) }{N\mu(T_N)}
	\cdot \lambda_N(T_N) \Vert \varphi_{2,N^{-1}}(T_N) \Vert_{L^2_x} \cdot N \Vert \partial_t\varphi_{2,N^{-1}}(T_N) \Vert_{L^2_x}
%
\\ 
& \qquad \lesssim \dfrac{\log\kappa(T_N) \lambda_N(T_N)}{\langle T_N\rangle } \cdot \delta^2 
	= \dfrac{9N\log^2(\kappa(T_N))}{\langle T_N\rangle} \cdot \delta^2.
\end{align*}
Therefore, summing up in $N$, recalling the definition of $T_N$ above, we conclude that 
\begin{align*}
\sum_{\substack{N\in\mathbb{D}, \\ 1< N\leq c\delta^{-4/3}}} \dfrac{\log\kappa(T_N)}{N} \vert \mathcal{H}_{N^{-1}}(T_N)\vert
	\lesssim \delta^2\sum_{\substack{N\in\mathbb{D}, \\ N>1}} \dfrac{N \log^2(\kappa(T_N))}{\langle T_N\rangle}\lesssim \delta^2.
\end{align*}
Finally recalling that $T_{\max} \approx \exp(\delta^{-4/3-})$, it is not difficult to see that, exactly the same argument 
shows that
\begin{align*}
\sum_{\substack{N\in\mathbb{D}, \\ 1< N\leq c\delta^{-4/3}}} \dfrac{\log\kappa(T_{\max}) }{N} \vert \mathcal{H}_{N^{-1}}(T_\mathrm{max})\vert   
	\lesssim \delta^2\sum_{\substack{N\in\mathbb{D}, \\ 1< N\leq c\delta^{-4/3}}}
	\dfrac{N\log^2(\kappa(T_\mathrm{max}))}{\langle T_\mathrm{max}\rangle}\lesssim \delta^2.
\end{align*}
To conclude the proof of \eqref{boundary_H_N1} we need to estimate the integral in \eqref{intHN}.
For this it suffices to use that, under our assumptions,
\begin{align*}
\big| \mathcal{H}_{N^{-1}}(t) \big| \lesssim \frac{1}{\langle t \rangle} \lambda_N(t) N \cdot \delta^2 \lesssim \frac{\delta^2}{\langle t \rangle^{1/2}}
\end{align*}
so that, since $\dot\kappa / \kappa \ll \langle t \rangle^{-1}$,
\begin{align*}
\Big| \sum_{\substack{N\in\mathbb{D}, \\ 1< N\leq c \delta^{-4/3} }}\dfrac{1}{N}
		\int_{T_N}^{T_\mathrm{max}} \frac{\dot\kappa(t)}{\kappa(t)} \, \mathcal{H}_{N^{-1}}(t) \, dt \Big|
		& \lesssim \sum_{\substack{N\in\mathbb{D}, \\ 1< N\leq c \delta^{-4/3} }} \dfrac{1}{N} \frac{1}{T_N^{1/2}} \delta^2 \lesssim \delta^2
\end{align*}
This concludes the proof of the lemma.
\end{proof}


\medskip
\subsection{Virial identity for $\mathcal{V}_\theta$.} 
For ease of reference we recall here our two main virial estimates \eqref{final_2_og_virial_hf} and  \eqref{dj_m_dk_ineq},
making the dependence on the absolute constants explicit:
the first one reads
\begin{align}\label{final_2_og_virial_hf'}
\begin{split}
& \dfrac{d}{dt} \mathcal{I}_\geq + \dfrac{d}{dt} \mathcal{P}_\geq
\\ 
& \leq -\dfrac{1}{2\mu}\int w_t^2\sech^{1/2}(x) - \dfrac{1}{\mu}\int w_x^2
  + \dfrac{C_1}{\mu} \int w^2\sech^{1/2}(x) 
  + O(\delta) \frac{1}{\mu} a_1^2 + \mathfrak{R}_1(t),
\end{split}
\end{align}
and the second one reads
\begin{align}\label{dotJ-K'}
\begin{split}
\frac{d}{dt} \mathcal{J}_\geq - \frac{d}{dt} \mathcal{K}_\geq 
  \leq -\dfrac{1}{2\mu}\int z_x^2 - \dfrac{1}{8\mu}\int z^2\sech^{1/4}(x) 
  + \dfrac{\epsilon}{\mu} (a_1^2+a_2^2)
\\ 
+ \dfrac{\epsilon}{\mu}\int \big(w_t^2+w^2\big)\sech^{1/2}(x)+F_\star(t)
\end{split}
\end{align}
Then, recalling the definition of $\mathcal{V}_\theta$ in \eqref{def_Vc} and using 
\eqref{final_2_og_virial_hf'} and \eqref{dotJ-K'} 
we obtain that 
\begin{align}\label{dt_I_P_J_K}
\begin{split}
\dfrac{d}{dt}\mathcal{V}_\theta \leq & - \dfrac{\theta}{2\mu}\int z_x^2 
	- \dfrac{\theta}{8\mu}\int z^2\sech^{1/4}(x)
	- \dfrac{1 - 2\theta \epsilon}{2\mu}\int w_t^2\sech^{1/2}(x) 
\\
& - \dfrac{1}{\mu}\int w_x^2
 + \dfrac{C_1 +\theta \epsilon}{\mu} \int w^2\sech^{1/2}(x)
 + \dfrac{2\theta\epsilon}{\mu}(a_1^2+a_2^2) + \mathfrak{R}_\star(t),
\end{split}
\end{align}
where (see \eqref{def_r1_og_virial_fh} 
for the definition of $\mathfrak{R}_1$) 
\begin{align}\label{dt_I_P_J_KRem}
\mathfrak{R}_\star(t) & := \mathfrak{R}_1(t) 
 + F_\star (t),
 \qquad \left\vert \int_0^{T_\mathrm{max}}\mathfrak{R}_\star(t)dt\right\vert\lesssim \delta^2
\end{align}
in view of \eqref{dotJ-KF} and \eqref{propORR1}.
%
%
%
%

To handle the integral of $w^2\sech^{1/2}(x)$  in \eqref{dt_I_P_J_K}  
we use that, as a direct consequence of Lemma \ref{prop_coer_varepsilon}
and the definitions \eqref{w_def} and \eqref{def_z}, we have
\begin{align}\label{coer_w_z}
\int w^2\sech^{1/2}(x) \leq C_3 \Big( \int z_x^2+\int z^2\sech^{1/4}(x) \Big) +O(\delta^2/\kappa),
\end{align}
for some absolute $C_3>0$;
this allows us to absorb the term on the second line of 
\eqref{dt_I_P_J_K} that contains $w^2\sech^{1/2}(x)$,
by the first and second integrals, that is, those that contain $z_x^2$ and $z^2\sech^{1/4}(x)$ 
by choosing  $\theta \geq 1$ large enough
depending only on the constants $C_1$, $C_2$ and $C_3$ above.  
Given this $\theta$ we then fix $\epsilon$ small enough (hence $B$ large enough)
such that the last term in the first line of \eqref{dt_I_P_J_K} is negative;
for instance, we can choose any pair satisfying 
$\theta > 8 C_1C_3$ and $\epsilon C_3 \ll 1$.
We then conclude that,
for some absolute constant $\theta \geq 1$ 
the following inequality holds:
\begin{align}\label{dt_I_P_J_K_2}
\begin{split}
\dfrac{d}{dt}\mathcal{V}_\theta 
	\leq  -\dfrac{1}{\mu}\int z_x^2-\dfrac{1}{\mu}\int z^2\sech^{1/4}(x)-\dfrac{1}{\mu}\int w_t^2\sech^{1/2}(x)
\\ 
 - \dfrac{1}{\mu}\int w_x^2 + \dfrac{\sqrt{\epsilon}}{\mu}\big(a_1^2+a_2^2\big) + \mathfrak{R}_\star(t)
\end{split}
\end{align}
for some $\mathfrak{R}_\star(t)$ satisfying \eqref{dt_I_P_J_KRem}.

To conclude the proof of the main theorem we only need to control the terms involving $(a_1,a_2)$.
We do this in the next subsection.

\medskip
\subsection{Control of $a_1(t)$, $a_2(t)$ and end of the proof}
Let us recall that $(a_1(t),a_2(t))$ satisfy the system of equations 
\begin{align*}
\begin{cases}
\dot{a}_1(t)=a_2(t),
\\ \dot{a}_2(t)=3a_1(t) + \langle \mathcal{N},Y\rangle.
\end{cases}
\end{align*} 
Based on the above system, let us define $\mathcal{B}(t):=\tfrac{1} {\mu(t)}a_1(t)a_2(t)$. Recalling the estimate for $\langle \mathcal{N}, Y\rangle$ in \eqref{NY}
we find that 
\begin{align*}
\dfrac{d}{dt}\mathcal{B}&=\dfrac{1}{\mu}a_2^2+\dfrac{3}{\mu}a_1^2+\dfrac{a_1}{\mu}\langle \mathcal{N},Y\rangle -\dfrac{\mu'}{\mu}\mathcal{B}
\\ 
& \gtrsim \dfrac{1}{\mu}a_2^2+\dfrac{1}{\mu}a_1^2-\dfrac{\mu'}{\mu}\mathcal{B}-\dfrac{\delta}{\mu}\int w^2\sech^2(x) +\dfrac{1}{\mu}O(\delta^4)+\dfrac{1}{\mu}O(\delta^2/\kappa).
\end{align*}
Adding this to \eqref{dt_I_P_J_K_2}, using again \eqref{coer_w_z},
we infer that  
\begin{align*}
\dfrac{d}{dt}\mathcal{B} - \dfrac{d}{dt}\mathcal{V}_\theta \gtrsim \dfrac{1}{\mu}\int z_x^2+\dfrac{1}{\mu}\int z^2\sech^{1/4}(x)
	+ \dfrac{1}{\mu}\int w_t^2\sech^{1/2}(x) 
\\ 
+ \dfrac{1}{\mu}\int w_x^2+\dfrac{1}{\mu}\big(a_1^2+a_2^2\big) + \mathfrak{R}_\star(t)
\end{align*}
for some $\mathfrak{R}_\star(t)$ satisfying \eqref{dt_I_P_J_KRem}.
%
%
Integrating in time from $0$ to $T_\mathrm{max}$, 
along with the fact that 
\[
\vert \mathcal{I}_\geq(t)\vert+\vert \mathcal{P}_\geq(t)\vert + \vert \mathcal{J}_\geq(t)\vert+\vert \mathcal{K}_\geq (t)\vert
	+ \vert\mathcal{B}(t)\vert = O(\delta^2) \quad \hbox{ for all } \ t\in \R,
\]
and using property \eqref{coer_w_z}, 
we conclude that
\begin{align*}
&\int_0^{T_\mathrm{max}} \dfrac{1}{\mu(t)}\int_\R \Big(z_x^2(t,x) + z^2(t,x)\sech^{1/4}(x) 
\\ 
& \qquad \qquad  \qquad \quad +w_t^2(t,x)\sech^{1/2}(x)+w_x^2(t,x)+w^2(t,x)\sech^{1/2}(x)\Big)dxdt=O(\delta^2),
\end{align*}
and that 
\begin{align*}
\int_0^{T_\mathrm{max}}\dfrac{1}{\mu(t)} 
  \big( a_1^2(t) + a_2^2(t) \big) 
  dt=O(\delta^2).
\end{align*}
Consequently, in view of the decomposition $\varepsilon_1=\varepsilon_{1,<}+\varepsilon_{1,\geq}$, 
the inequality \eqref{lemma<conc} controlling the low frequencies
and the coercivity property \eqref{prop_coer_varepsilon} with \eqref{def_z},
we finally obtain that 
\[
\int_0^{T_\mathrm{max}} \dfrac{1}{\mu(t)}\int_\R \varepsilon_1^2\sech^{1/2}(x)dxdt=O(\delta^2).
\]
The proof of Theorem \ref{maintheo} is complete.



\bigskip
\section{Supporting material}

\subsection{The distorted Fourier transform}\label{appdFT}
In this section we give a quick introduction to the distorted Fourier Transform 
since some of its basic properties are needed in our analysis.  We will mostly restrict our attention to the specific potential appearing in 
the linearization around the Klein-Gordon soliton, but analogous properties hold in the case of
the linearized operator at the Sine-Gordon kink.
In both of these cases all formulas are explicit, and fairly simple, but the general theory can be used in more general cases. Of course all of our arguments do not rely on such explicit formulas. 
We refer to the classical works \cite{Y10,DT79} and to \cite{GP20} 
for a more general introduction to this topic. 

The distorted Fourier basis associated with $H := -\partial_x^2 + V(x)$ for a general,
sufficiently decaying potential $V$, is given by 
\begin{align}\label{edFT}
e(x,\xi):=\dfrac{1}{\sqrt{2\pi}}
\begin{cases}
T(\xi)f_+(x,\xi) & \hbox{ for } \, \xi\geq 0, 
\\ 
T(-\xi)f_-(x,-\xi) & \hbox{ for } \, \xi<0,
\end{cases}
\end{align}
where $f_\pm(x,\xi)$ are the (normalized) Jost solutions given by
\begin{align}\label{JostODE}
(-\partial_x^2+V(x))f_\pm(x,\xi)=\xi^2f_\pm(x,\xi)  
  \quad \hbox{with} \quad \lim_{x\to \pm \infty}\big\vert f_\pm(x,\xi)-e^{\pm i x\xi}\big\vert=0,
\end{align}
and $T(\xi)$ is the transmission coefficient, which can be determined from the relation 
\begin{align}\label{TW}
T(\xi) W(f_+(\cdot,\xi), f_-(\cdot,\xi)) = -2i\xi, 
\end{align}
where $W$ is the Wronskian.

The potential appearing in the linearized operator \eqref{introlinop}, 
belongs to the family of \textit{P\"oschl-Teller} potentials
which have the form $V_\ell(x) = -\ell(\ell+1) \sech^2(x)$.
The operators $-\partial_x^2 + V_\ell(x)$ are all calculable, since they can be conjugated to the flat 
operator $-\partial_x^2$ see \cite{DT79,LSch23}. 
In our particular case $V(x):=-6\sech^2(x) = V_2(x)$ and, taking the conjugate of the identity \eqref{op_SUconj} 
we have
\begin{align}\label{op_SUconj*}
(-\partial_x^2 + V(x)) U^* S^* = - U^* S^* \partial_x^2,
  \qquad U^*S^* = \partial_x^2 - 3\tanh(x) \partial_x + 3 \tanh(x)^2 - 1.
\end{align}
From this we can find the Jost functions in the form
\begin{align}\label{Jost1}
f_+(x,\xi) = c(\xi) U^* S^* e^{ix\xi}, \qquad f_-(x,\xi) = c(\xi) U^* S^* e^{-ix\xi},
\end{align}
with the suitable choice of the normalizing constant
\begin{align}
c(\xi) = (-\xi^2 - 3i\xi + 2)^{-1}
\end{align}
to ensure the asymptotics \eqref{JostODE} as $|x| \rightarrow \infty$.
We can then explicitly calculate from \eqref{Jost1} that
\begin{align}
\label{Jostformula}
\begin{split}
f_+(x,\xi) & = \frac{-\xi^2 - 3i\xi \tanh(x) + 3\tanh(x)^2 - 1}{-\xi^2 - 3i\xi + 2} e^{ix\xi},
\\
f_-(x,\xi) & = f_+(-x,\xi)
\end{split}
\end{align}
where the last identity follows by uniqueness of solutions of the ODE \eqref{JostODE}. We then calculate the transmission coefficient, using the value of the Wronskian at $x=\infty$:
\begin{align}\label{Tcoeff}
T(\xi) = -2i\xi W(f_+(\cdot,\xi), f_-(\cdot,\xi)) = \frac{-\xi^2 - 3i\xi + 2}{-\xi^2 + 3i\xi + 2}.
\end{align}
Note that $|T(\xi)|=1$, which means that $V$ is a reflectionless potential.

Finally, the distorted Fourier transform is defined by
\begin{align}\label{defdFT}
\widetilde{\mathcal{F}}f(\xi):=\int_\R \overline{e(x,\xi)} f(x)dx. 
\end{align}
where an explicit formulas for $e(x,\xi)$ 
is obtained putting together \eqref{edFT}, \eqref{Jostformula} and \eqref{Tcoeff}
(in particular $f_-(x,-\xi) = f_+(-x,-\xi)$)
\begin{align}\label{eformula}
e(x,\xi):=\dfrac{1}{\sqrt{2\pi}}
\frac{-\xi^2 - 3i\xi \tanh(x) + 3\tanh(x)^2 - 1}{-\xi^2 + 3i|\xi| + 2} \, e^{ix\xi} \qquad \mbox{for} \,\, \xi \in \R,
\end{align}


The following general theorem summarizes the main properties of the distorted Fourier Transform.

\begin{thm}\label{thmdFT}
Assume that $V\in\mathcal{S}(\R)$ and has no bound states. 
Then:

\begin{itemize}

\smallskip
\item[-] The distorted Fourier transform $\widetilde{\mathcal{F}}$ is an isometry on $L^2$, that is, \[
\langle f,g\rangle =\big\langle \widetilde{\mathcal{F}}f,\widetilde{\mathcal{F}}g\big\rangle \quad \hbox{ for all } f,g\in L^2
\]

\smallskip
\item[-]  $\widetilde{\mathcal{F}}$ is a bijection, with
\[
\widetilde{\mathcal{F}}^{-1}\phi(x)=\int e(x,\xi)\phi(\xi)d\xi.
\]

\smallskip
\item[-]  $\widetilde{\mathcal{F}}$ diagonalizes $H=-\partial_x^2+V$, that is, 
\[
\widetilde{\mathcal{F}}\big(Hf\big)(\xi)=\xi^2\widetilde{\mathcal{F}}f(\xi).
\] 
In particular, for any bounded function $a$ one has 
$\widetilde{\mathcal{F}}\big( a(H) f \big)(\xi)= a(\xi^2) \widetilde{\mathcal{F}}f(\xi)$.


\end{itemize}
\end{thm}

\begin{rem}
Even though our potential $V=-6\sech^2(x)$ has bound states, 
the formulas in Theorem \ref{thmdFT} 
still holds if $f$ and $g$ belong to $L^2_c$, where $L^2_c$ denotes
the projection onto the continuous spectral subspace of $L^2(\R)$ relative to $H=-\partial_x^2+V$.
\end{rem}

\begin{rem}[Additional properties in the case $V=-6\sech^2(x)$]\label{remdFTadd}

Consider the linear operator in \eqref{introlinop}, then the following holds:

\begin{itemize}

\smallskip
\item[(i)] The wave operators 
\begin{align}  
\mathcal{W} := \widetilde{\mathcal{F}}^{-1} \widehat{\mathcal{F}} , \qquad 
\mathcal{W}^{-1} = \mathcal{W}^\ast = \widehat{\mathcal{F}}^{-1} \widetilde{\mathcal{F}},
\end{align}
are bounded on $L^p$ and $W^{1,p}$ for all $p\in[1,\infty]$,
and intertwine $H=-\partial_x^2 + V$ and $H_0 = -\partial_x^2$:
$$f(H) = \mathcal{W} f(H_0) \mathcal{W}^\ast.$$
In particular, Bernstein-type inequalities hold in the case of distorted projections:
for all $1\leq p \leq q \leq \infty$
\begin{align*}
{\big\| \widetilde{\mathbb{P}}_{K} f \big\|}_{L^q} \lesssim K^{(1/p-1/q)} {\| f \|}_{L^p}.
\end{align*}

\smallskip
\item[(ii)] The dFT (associated to any even potential) preserves parity: 
if $f$ is even/odd then $\wt{\mathcal{F}}(f)$ is even/odd.

\end{itemize}

\end{rem}


\medskip
\subsection{A useful lemma} 
\label{appendix_proof_lemma_weight_triple_prime}
In this subsection we state and prove the lemma that allows us to
absorb some of the (linear) lower order terms appearing in our virial estimates
into the leading order ones,
when there is a scale separation between the function and the weight;
see for example \eqref{dxm1_upphi_triple_prime_into_LO}.

\begin{lem}\label{weight_triple_prime} 
Let $f\in L^2$ with $\Vert f\Vert_{L^2}=\delta$, 
such that that its Fourier transform is supported on frequencies $\vert \xi\vert \geq \tfrac12N^{-1}$, 
where $N\geq1$. Define $\lambda_N := N \log \kappa^9$, with $\kappa\geq \tfrac N2\vee D$ where 
$D$ is an absolute constant defined below. Then, the following inequality holds
\begin{align}\label{dxm1_estimate_when_gap}
\left \vert \int (\partial_x^{-1}f)^2\sech^2\big(\tfrac{x}{\lambda_N}\big)\right\vert 
  \leq c_1^* N^2\int f^2\sech^2\big(\tfrac{x}{\lambda_N}\big) + O(\delta^2/\kappa^{1+}).
\end{align}
\end{lem}

\begin{proof}
For the sake of simplicity, 
let us denote the weight function 
$\sech\big(\tfrac{x}{\lambda_N}\big)$ by $\Psi_{\lambda_N}'(x)$. 
We split this weight function in frequencies as follows 
\[
\Psi_{\lambda_N}'=\Psi_{\lambda_N,<}'+\Psi_{\lambda_N,\geq}' 
\quad \hbox{where} \quad \Psi_{\lambda_N,<}':=\mathbb{P}_{< (4N)^{-1}}(\Psi_{\lambda_N}') 
\quad \hbox{and} \quad \Psi_{\lambda_N,\geq}' :=\mathbb{P}_{\geq (4N)^{-1}}(\Psi_{\lambda_N}').
\]
Then, by using Plancherel Theorem, Young convolution inequality, 
followed by Bernstein,
along with the fact that $\Vert  \hat{\Psi}_{\lambda_N,\geq}\Vert_{L^1}= O (e^{-\pi\lambda_N/(16N)}) $
due to \eqref{hatsech}, 
we obtain that 
\begin{align*}
\int (\partial_x^{-1}f\cdot \Psi_{\lambda_N,\geq}')^2 
& \lesssim  
  {\| \widehat{\mathcal{F}}(\partial_x^{-1}f) \|}_{L^2}^2 
  {\| \widehat{\mathcal{F}}(\Psi_{\lambda_N,\geq}') \|}_{L^1}^2 
\\
& \lesssim N^{2}\delta^2 \cdot e^{-\pi\lambda_N/(8N)} \lesssim N^2 \delta^2/\kappa^{3+} = O(\delta^2/\kappa^{1+}). 
\end{align*}
Note that the right-hand side of the above estimate is compatible with \eqref{dxm1_estimate_when_gap}.  Next, in order to deal with $\Psi_{\lambda_N,<}'$ we write 
\begin{align}\label{commutator_deco_fl_gs}
\partial_x^{-1}f\cdot \Psi_{\lambda_N,<}' =\partial_x^{-1}\big(f\Psi_{\lambda_N,<}'\big)-[\partial_x^{-1},\Psi_{\lambda_N,<}']f,
\end{align}
where $[\cdot,\cdot]$ stands for the standard commutator operator.
Then, on the one-hand, due to the presence of $\mathbb{P}_{<(4N)^{-1}}$ in the definition of $\Psi_{\lambda_N,<}$, we have  
\[
\supp\mathcal{F}\big(f\Psi_{\lambda_N,<}'\big)\subset (-\tilde cN^{-1},\tilde cN^{-1})^c,
\]
for some $\tilde{c}\in\R_+$, with $\tilde{c}\sim 1$, and hence, 
\begin{align}\label{triplep_c1star}
\int \big(\partial_x^{-1}(f\Psi_{\lambda_N,<}')\big)^2
  \leq \frac{c_1^*}{2} N^2\int \big(f\Psi_{\lambda_N,<}'\big)^2
  \leq \frac{c_1^*}{2} N^2\int f^2\sech^2\big(\tfrac{x}{\lambda_N}\big) + O\big(\delta^2/\kappa^{1+}\big),
\end{align}
for some absolute constant $c_1^*>0$,
where we have used again that $\lambda_N> N$ to rewrite $\Psi_{\lambda_N,<}'$ 
in terms of $\sech(\frac{x}{\lambda_N})$ plus a term contributing to the remainder, namely, 
$\Psi_{\lambda_N,<}'=\Psi_{\lambda_N}'-\Psi_{\lambda_N,\geq}'$. 

On the other hand, using Plancherel Theorem we have 
\begin{align}\label{triplep_commutator_identity_dxm1}
\begin{split}
\int \big([\partial_x^{-1},\Psi_{\lambda_N,<}']f\big)^2 
  & =\int \Big(\partial_x^{-1}(f\Psi_{\lambda_N,<}')-\Psi_{\lambda_N,<}'\cdot\partial_x^{-1}f\Big)^2
\\ & =\tilde{c}\int \Big\vert \dfrac{(\widehat{f}*\widehat{\Psi}_{\lambda_N,<})}{\xi} 
  - \widehat{\Psi}_{\lambda_N,<}*\Big(\dfrac{\widehat{f}}{\xi}\Big)\Big\vert^2
\\ & =\tilde{c} \int \Big\vert \int \dfrac{\mu}{\xi(\xi-\mu)}\widehat{f}(\xi-\mu)\widehat{\Psi}_{\lambda_N,<}(\mu) \Big\vert^2 
\\ & =\int \Big(\partial_x^{-1}\big(\partial_x^{-1}f\cdot \partial_x\Psi_{\lambda_N,<}'\big)\Big)^2      
\\ &\leq c_2^* N^2\int \big(\partial_x^{-1}f\cdot \partial_x\Psi_{\lambda_N,<}'\big)^2 
\\ & \leq  \dfrac{c_2^* N^2}{\lambda_N^2}\int \big(\partial_x^{-1}f 
  \cdot  \Psi_{\lambda_N}'\big)^2
  + c_2^* N^2\int \big(\partial_x^{-1}f\cdot \partial_x\Psi_{\lambda_N,\geq}' \big)^2,
\end{split}
\end{align}
for some absolute constant $c_2^*>0$, where, from the second last to the last inequality, we have rewritten
$\Psi_{\lambda_N,<}'$ as $\Psi_{\lambda_N}'-\Psi_{\lambda_N,\geq}'$, 
and hence, we can use that 
\[
\vert \Psi_{\lambda_N}''(x)\vert=\big\vert \tfrac{1}{\lambda_N}
  \sech\big(\tfrac{x}{\lambda_N}\big)\tanh\big(\tfrac{x}{\lambda_N}\big)\big\vert 
  \lesssim \tfrac{1}{\lambda_N}\Psi_{\lambda_N}'(x).
\] 
Then, defining $D:=e^{8c_*}$, where $c_*:=c_1^*\vee c_2^*$ 
is an absolute constant depending on $c_1^*$ and $c_2^*$ appearing in \eqref{triplep_c1star} 
and \eqref{triplep_commutator_identity_dxm1}, we conclude that the first term 
on the right-hand side of \eqref{triplep_commutator_identity_dxm1}
can be reabsorbed on the left side of \eqref{dxm1_estimate_when_gap}, 
using the fact that $\lambda_N\gg c_2^*N$.
The last remaining term can be 
bounded in the same fashion as at the beginning of this proof:
using Plancherel Theorem with Young convolution inequality, we estimate
\begin{align*}
N^2\int \big(\partial_x^{-1}f\cdot \partial_x\Psi_{\lambda_N,\geq}' \big)^2
  \lesssim O\left(\tfrac{N^4\delta^2}{\lambda_N^2}e^{-\pi\lambda_N/(8N)}\right).
\end{align*} 
Due to the definition of $N$ and $\lambda_N$, 
we conclude that the right-hand side of the latter inequality is $O(\delta^2/\kappa^{1+})$, 
which concludes the proof of the lemma.
\end{proof}


\end{document}